\documentclass[final,leqno]{siamltex704}

\usepackage{amsmath,amssymb}% for \begin{bmatrix}, \begin{smallmatrix}, \begin{align} etc.
\usepackage{amsfonts}% for \mathbb{} etc. `blackboard' for uppercase only.
\usepackage{exscale}
\usepackage{graphicx}
\usepackage[pdftex]{thumbpdf}      %%% thumbnails for pdflatex. Run thumbpdf at the end
\usepackage[pdftex,                %%% hyper-references for pdflatex
pdfstartview=FitH,%
bookmarks=true,%                   %%% generate bookmarks ...
bookmarksnumbered=true,%           %%% ... with numbers
bookmarksopen=true,%
hypertexnames=false,%              %%% needed for correct links to figures !!!
breaklinks=true,%                  %%% break links if exceeding a single line
  colorlinks=true,%
  linkcolor=blue,anchorcolor=blue,%
  citecolor=blue,filecolor=blue,%
  menucolor=blue]%
{hyperref}

\usepackage{etoolbox}  % if then statement to get longer version for tech. report
\newtoggle{LONG}
\toggletrue{LONG} % set longer tech report
\makeatletter
\def\@cite#1#2{{\rm [}{{\rm#1}\if@tempswa , #2\fi}{\rm ]}}
\makeatother

\def\inv{^{-1}}

\newtheorem{algorithm}{Algorithm}[section]

\title{
Graph partitioning using matrix values for preconditioning symmetric positive definite systems
\thanks{This
work was supported in
part by Iowa State University under the contract DE-AC02-07CH11358 with
the U.S.  Department of Energy and by the Director, Office of Science,
Division of Mathematical, Information,
and Computational Sciences of the U.S. Department of Energy under
contract number DE-AC02-05CH11231 and by the U.S. Department of
Energy under the grant DE-FG-08ER25841.
The first two authors also benefited from resources from the
Minnesota Supercomputing Institute.}}

\author{
Eugene Vecharynski\thanks{Computational Research Division, 
Lawrence Berkeley National Laboratory, Berkeley, CA 94720
({\tt eugene.vecharynski@gmail.com).}}
\and  Yousef Saad\thanks{
Department of Computer Science and
Engineering, University of Minnesota, 200 Union Street S.E.,
Minneapolis, MN 55455, USA ({\tt saad@cs.umn.edu}).}
\and Masha Sosonkina\thanks{
Department of Modeling, Simulation and Visualization Engineering, Old Dominion University, Norfolk, VA 23529  ({\tt
msosonki@odu.edu}).}
}

\begin{document}
%siam_id=72574
%CODEN=SJMAEL
%\slugger{sisc}{?}{?}{?}{?--?}
\maketitle

\setcounter{page}{1}

%\renewcommand{\thefootnote}{\fnsymbol{footnote}}
%\footnotetext[1]{Department of Computer Science and
%Engineering, University of Minnesota, 200 Union Street S.E.,
%Minneapolis, MN 55455, USA ({\tt \{eugenev, saad\}@cs.umn.edu}). This work has been supported by the
%U.S. Department of Energy under the grant DE-FG-08ER25841}
%\footnotetext[2]{???}
%\renewcommand{\thefootnote}{\arabic{footnote}}

\begin{abstract}
Prior  to  the parallel  solution  of a  large  linear  system, it  is
required to perform a partitioning of its equations/unknowns. Standard
partitioning algorithms  are designed using the  considerations of the
efficiency of the parallel matrix-vector multiplication, and typically
disregard the information on the coefficients of the matrix.
This  information, however,  may have  a significant  impact  on the
quality  of  the  preconditioning  procedure used  within  the  chosen
iterative  scheme.   In  the  present  paper, we  suggest  a  spectral
partitioning algorithm,  which takes  into account the  information on
the matrix coefficients and  constructs partitions with respect to the
objective  of   enhancing  the   quality  of  the nonoverlapping   additive  Schwarz
(block Jacobi) preconditioning    for     symmetric    positive    definite    linear
systems. 
For a set of test problems with large variations in magnitudes of matrix coefficients,
%Numerical  results for a  set of test problems  
our numerical experiments
demonstrate a
noticeable  improvement in  the convergence  of the  resulting solution
scheme when using the new partitioning approach.
\end{abstract}


% 15A06 -- Linear equations
% 65F08 -- Preconditioners for iterative methods
% 65F10 -- Iterative methods for linear systems
% 65N22 -- Solution of discretized equations

%\begin{DOI}
%
%\end{DOI}

\pagestyle{myheadings}
\thispagestyle{plain}
\markboth{EUGENE VECHARYNSKI, YOUSEF SAAD, AND MASHA SOSONKINA}{GRAPH PARTITIONING USING MATRIX VALUES} %50 Characters Limit

\section{Introduction}\label{sec:intro}

Partitioning of  a linear system  for its parallel  solution typically
aims  at   satisfying  two  standard   objectives:  \textit{minimizing
  communication  volume} and  \textit{maintaining load  balance} among
different  processors.  Both  of these  requirements are  motivated by
efficiency considerations  of  the  parallel  matrix-vector
product,  which lie  in the  heart of  the  iterative solution
methods.  Once performed  the partitioning  is then 
exploited to construct a parallel preconditioner---another
crucial ingredient which contributes to the overall performance of the
solver.   However,  the   quality  of   the  resulting
preconditioner may depend significantly on the given partitioning, which
generally targets the  efficiency   of  the   parallel  matrix-vector
multiplication, but   ignores its effect on  the resulting 
preconditioner. This preconditioner can be of poor quality
especially  in the cases  when the  coefficient matrices  have entries
with large variations in magnitudes.

For the purpose of obtaining an effective preconditioner, 
we suggest to remove the requirement on the communication volume
and, instead, consider partitionings that favor the quality of 
the preconditioner.
In particular, we focus on the \textit{nonoverlapping additive Schwarz (AS) preconditioners}
for symmetric positive definite (SPD) linear systems~\cite{DDSmith:96, DDWidlund:05}.
The proposed partitioning algorithm 
aims at optimizing the quality of the AS procedure by attempting to minimize
the condition number of the preconditioned matrix,
while maintaining a good load balance.  The new strategy is tested on
several linear systems that arise from discretizations
of partial differential equations (PDE's) with \textit{strongly varying} coefficients.

The choice of the nonoverlapping AS, which is a form of block-diagonal, or block Jacobi, preconditioning,
is motivated by several factors.
First, the procedure represents the simplest Domain Decomposition (DD) type
preconditioner, which is theoretically well-understood and is often of practical interest
due to its high degree of parallelism.
Second,
block-diagonal preconditioning constitutes an important element of a number of more
powerful preconditioning schemes, e.g., overlapping or multilevel Schwarz methods,
substructuring type algorithms, etc.
Our expectation is that
the partitions that improve the quality of the
nonoverlapping AS preconditioners 
% will 
%provide a similar effect for 
%also positively affect the robustness of 
are also capable of increasing the robustness of these preconditioning schemes.
%In particular, we demonstarate that our
%expectation is confirmed for the case of the \textit{overlapping} AS preconditioning,
%obtained by adding a few neighboring nodes to the nonoverlapping subdomains.

The problem of partitioning a linear system $Ax=b$ is commonly
formulated in terms of the adjacency graph $G(A) = (V,E)$ of the coefficient matrix $A=\left(a_{ij}\right)$.
Here, $V = \left\{ 1, 2, \ldots, n \right\}$
is the set of vertices (nodes) corresponding to the equations/unknowns of the system,
and $E$ is the set of edges $(i,j)$, where $(i,j) \in E$ iff $a_{ij} \neq 0$.
Throughout, we assume that $A$ is SPD, i.e., $A = A^* \succ 0$, which, in particular, implies that
the graph $G(A)$ is undirected.

The \textit{standard} goal of graph partitioning is to partition $G(A)$ into $s$ subgraphs
$G_k = (V_k,E_k)$, where $V_k \subseteq V$ and $E_k \subseteq E$, such that
\begin{equation}\label{eqn:partn}
\bigcup_{k = 1,s} V_k = V, \qquad \bigcap_{k=1,s} V_k = \emptyset, \qquad |V_k| \approx n/s \; ,
\end{equation}
and the size of the edge cut between $G_k$ 
(i.e., the size of the set of edges whose end points are in different $G_k$) 
is minimized.
%Note that the additional constraint is imposed on
%the cardinalities of $V_k$, which should be approximately the same.
%, i.e., $|V_k| \approx n/s$.
Equations and unknowns with numbers in $V_k$ are then typically mapped to the same processor;
$s$ corresponds to the total number of processors.
The requirement on the small edge cut aims at reducing the cost of communications related to the
parallel matrix-vector multiplication. The condition $|V_k| \approx n/s$
ensures load balancing.
We note that there are alternative models for graph partitioning, based, e.g.,
on bipartite graphs~\cite{Hendrickson.Kolda:00} or hypergraphs~\cite{Catalyurek.Aykanat:99}.
We do not consider these models in the present paper.

The graph partitioning problem is NP-complete.
However, there exist a variety of heuristics for solving the problem;
see, e.g.,~\cite{Fiduccia.Mattheyses:82, Goehring.Saad:94, Kernighan.Lin:70, Miller.Teng.Thurston.Vavasis:93, Pothen.Simon.Liou:90}.
Efficient implementations of partitioning routines are often based on multilevel algorithms,
e.g.,~\cite{Hendrickson.Leland_multilev:95, Karypis.Kumar:98},
and are made available in a number of graph partitioning software packages,
such as Chaco~\cite{Chaco:94}, JOSTLE~\cite{Jostle:07}, MeTiS~\cite{Metis:09}, SCOTCH~\cite{Scotch:08}, etc.

If   the  preconditioner   quality   becomes  an   objective  of   the
partitioning,  then  along  with  the adjacency  graph  $G(A)$, it  is
reasonable  to consider weights  $w_{ij}$ assigned  to  the edges
$(i,j)\in E$, where $w_{ij}$'s are determined by the coefficients of the
matrix $A$.  The corresponding algorithm  should then be able  to take
these weights into account and  properly use them to perform graph
partitioning.   An example  of such  an algorithm  has  been discussed
in~\cite{Saad.Sosonkina:99}.

Indeed, one may consider partitioning as part of the pre-processing
phase of the preconditioner set up.
%if one  considers partitioning  as  an ``early  phase'' of  a
%preconditioning procedure (which, in  the purely algebraic setting, is
%based  solely  on  the  knowledge   of  $A$),
Then  the  use  of  the
coefficients  of  $A$ at  the  partitioning  step,  e.g., through  the
weights $w_{ij}$, represents a natural option. This approach, however,
faces   a number of issues.  For  example, given  a preconditioning
strategy,  how does  one  assign  the weights?   What  are the  proper
partitioning  objectives? How  can  the partitioning  be performed  in
practice?

In the present work,
%We address
these questions
are addressed
for the  case of SPD
linear   systems  and  nonoverlapping AS preconditioners.
Our rationale is to
relate partitioning to the results of the convergence
theory for the Preconditioned Conjugate Gradient (PCG) method with
block-diagonal preconditioning.
More specifically,
we regard bipartitioning (i.e., partitioning of the graph into two parts) 
as an optimization problem that aims at minimizing
an  upper  bound  on   the  condition  number  of  the  preconditioned
matrix over all possible balanced bipartitions.
As a result, we derive a recursive bisection procedure that is built upon
a simple weighting scheme
and a \textit{modification} of the standard partitioning objective.

A straightforward approach for dealing with the partitioning problem of this paper
would be to assign edge weights as the magnitudes of the corresponding matrix entries
and apply a state-of-the-art partitioning algorithm.
While this heuristic indeed often improves convergence, it can be substantially outperformed by the new partitioning strategy as
will be demonstrated in our numerical experiments.
%Such a heuristics indeed often leads to
%an improved convergence, however,
%as demonstrated in our numerical experiments,
%can be substantially outperformed by the new partitioning strategy. % described in this paper.
%
%An interesting phenomenon lies in the observation that, along with the amount of the coefficient
%information discarded to construct the AS preconditioners, the convergence of the preconditioned
%linear solver is strongly affected by the shape of the underlying subdomains.

%
%We show that, given a natural edge weighting scheme, the better
%partitions can be obtained by \textit{modifying the traditional objective}
%of minimizing the total (weighted) cut size.

%for a given $A$,  the proposed approach  is based on
%the idea of constructing  a (bi)partition,
%which attempts to minimize
%an  upper  bound  on   the  condition  number  of  the  preconditioned
%matrix over all possible balanced (bi)partitions.

%
The presented algorithm relies on the computation of
eigenvectors corresponding to the smallest eigenvalues of a generalized eigenvalue problem,
which simultaneously involves weighted and standard graph Laplacians.
As such, the new strategy is a form of the \textit{spectral recursive bisection} (RSB)
% which takes advantage of a
based on a ``nonstandard'' eigenvalue problem.
%

%Although the formal discussion is focused on the case
%of the \textit{nonoverlapping} AS, we show numerically that, in practice,
%adding several ``layers'' of neighboring nodes to the obtained sets $V_k$ in~(\ref{eqn:partn})
%leads to decompositions of $V$, which provide a good quality of
%the \textit{overlapping} AS preconditioning.

Spectral graph partitioning has roots in the works of Fiedler~\cite{Fiedler:73, Fiedler:75}
and Donath and Hoffman~\cite{Donath.Hoffman:72, Donath.Hoffman:73}.
It is extensively used in many applications, including VLSI circuit design, data clustering,
image segmentation; see, e.g.,~\cite{DingMcut:01, Luxburg:07, Shi.Malik:00} and the references therein.
In the scientific computing community, the spectral partitioning
was popularized by Pothen et al. in~\cite{Pothen.Simon.Liou:90} and further
studied, e.g., in~\cite{Hendrickson.Leland:95, Simon:91, Spielman.Teng:96}.

Spectral graph partitioning is typically characterized by a good quality of the resulting
partitions. This can be attributed to the fact that spectral algorithms
utilize global information in contrast to 
combinatorial algorithms that typically rely on local information.
The common criticism of spectral partitioning is its large computational cost,
due to the eigenvalue calculations. % involving graph Laplacians.

Although efficient solution of a specific eigenvalue problem is not the focus
of this work, we favor the use of preconditioned eigensolvers. % , such as LOBPCG~\cite{Knyazev:01}.
An attractive feature of this approach is that it splits the computation
into \textit{preconditioning phase} and \textit{outer iterations},
which are based on the matrix-vector multiplications. Ideally, one can expect that
the preconditioning scheme encapsulates multilevel strategies
(e.g., graph coarsening, uncoarsening, refinement)
inherent to the available combinatorial graph partitioners, combined with
only a few outer iterations. In our numerical experiments, however, for demonstration purposes,
we use Incomplete Cholesky (IC) preconditioning.

%An attractive feature of this approach is in that the choice of an efficient preconditioner,
%e.g., based on the multigrid or an icomplete factorization, may provide convergence
%to an eigenpair only in a few iterations, whereas the requirement on storage is kept
%reasonably low. In our numerical experiments, we use Incomplete Cholesky (IC) preconditioning.

The partitioning problem addressed in this paper is closely related to the more traditional task
of \textit{constructing block-diagonal preconditioners}. Indeed, the latter also admits graph formulations
and relies on the use of the matrix coefficients; see, e.g.,~\cite{Fritzsche.Frommer.Szyld:07}.
The main difference, however, lies in the fact that, along with improving the preconditioning quality,
the presented partitioning scheme imposes the additional load balance constraint,
which is not required for the conventional block-diagonal preconditioning.
%whereas the
%approaches for block-diagonal preconditioning skip this requirement.

The paper is organized as following. In Section~\ref{sec:bdiag},
we briefly review several known results concerning the block-diagonal preconditioning for SPD matrices.
These results motivate the new partitioning scheme, introduced in Section~\ref{sec:cbs_partn}.
In Section~\ref{sec:numer}, we report on a few  numerical examples, where the presented approach is compared
to a state-of-the-art partitioning algorithm, MeTiS, with different weighting schemes.

\section{Block-diagonal preconditioning}\label{sec:bdiag}

Consider a block $2$-by-$2$ matrix
\begin{equation}\label{eqn:bk_sys}
A =
\left(
\begin{array}{cc}
 A_{11}      &  A_{12}   \\
 A_{12}^*        & A_{22} \\
\end{array}
\right)\;,
\end{equation}
where the diagonal blocks $A_{11}$ and $A_{22}$ are square of size $m$ and $(n-m)$, respectively;
the off-diagonal block $A_{12}$ is $m$-by-$(n-m)$.
Let $T$ be a block-diagonal preconditioner,
\begin{equation}\label{eqn:bk_prec}
T =
\left(
\begin{array}{cc}
 T_1      &  0   \\
 0        &  T_2 \\
\end{array}
\right)\;,
\end{equation}
where $T_j = T_j^* \succ 0$, $j = 1,2$.
The dimensions of $T_1$ and $T_2$ are same as those of $A_{11}$ and
$A_{22}$, respectively.

Since both $A$ and $T$ are SPD, the convergence of an iterative method
for $Ax=b$,  such as PCG,  is   fully  determined   by  the   spectrum   of  the
preconditioned matrix  $T\inv A$.  If no information on  the exact location
of eigenvalues  of $T\inv A$ is available, then  the worst-case convergence
behavior of PCG  is traditionally described in terms  of the condition
number  $\kappa(T\inv A)$, where  $\kappa(T\inv A)  =  \lambda_{\max}(T\inv A)  /
\lambda_{\min}(T\inv A)$  with  $\lambda_{\max}(T\inv A)$ and  $\lambda_{\min}(T\inv A)$
denoting   the   largest  and   the   smallest   eigenvalues  of   the
preconditioned  matrix, respectively.  The  question which arises is how we can
bound
$\kappa(T\inv A)$ for an arbitrary $A$ and a block-diagonal $T$.
The answer to this question is given, e.g.,
in~\cite[Chapter 9]{Axelsson:94}.
Below, we briefly state the main result.

\begin{definition}\label{def:cbs}
Let $U_1$ and $U_2$ be finite dimensional spaces, such that $A$ in~(\ref{eqn:bk_sys})
is partitioned consistently with $U_1$ and $U_2$.
The constant
\begin{equation}\label{eqn:cbsA}
\gamma = \max_{w_1 \in W_1, w_2 \in W_2} \frac{| (w_1, A w_2) |}{(w_1,A
w_1)^{1/2} (w_2,Aw_2)^{1/2}}\;,
\end{equation}
where $W_1$ and $W_2$ are subspaces of the form
\begin{equation}\label{eqn:cbsW}
W_1 = \left\{
u =
\left(
\begin{array}{c}
 u_1   \\
 \textbf{0}     \\
\end{array}
\right), u_1 \in U_1
\right\}, \
W_2 = \left\{
u =
\left(
\begin{array}{c}
 \textbf{0}   \\
 u_2     \\
\end{array}
\right), u_2 \in U_2
\right\}
\end{equation}
is called the \emph{Cauchy-Bunyakowski-Schwarz (CBS) constant}.
\end{definition}

In~(\ref{eqn:cbsA}), $(u,v) = v^* u$ denotes the standard inner product.
We note that $\gamma$ can be interpreted as a cosine of an angle between subspaces
$W_1$ and $W_2$.
Thus, since, additionally, $W_1 \cap W_2 = (\textbf{0} \ \textbf{0})^*$, it is readily seen that $0 \leq \gamma < 1$.
Also we note that $\gamma$ is the smallest possible constant satisfying the strengthened Cauchy-Schwarz-Bunyakowski
inequality $|(w_1, A w _2)| \leq \gamma (w_1, A w_1)^{1/2}(w_2, A w_2)^{1/2}$,
which motivates its name.

%The definition of $\gamma$ in~(\ref{eqn:cbsA}) can be written in terms of the blocks of $A$ in~(\ref{eqn:bk_sys}), i.e.,
%\begin{equation}\label{eqn:cbs}
%\gamma = \max_{u_1 \in U_1, u_2 \in U_2} \frac{|(u_1,A_{12}u_2)|}{(u_1,A_{11} u_1)^{1/2}(u_2, A_{22} u_2)^{1/2}}.
%\end{equation}

%The CBS constant $\gamma$, defined by~(\ref{eqn:cbsA}),
%plays an important role in estimating $\kappa(T\inv A)$ for the class of SPD matrices $A$
%and block-diagonal preconditioners $T$.

\begin{theorem}[\cite{Axelsson:94}, Chapter 9]\label{thm:A1}
If $T_1 = A_{11}$ and $T_2 = A_{22}$ in~(\ref{eqn:bk_prec}),
and $A$ in~(\ref{eqn:bk_sys}) is SPD, then
$\kappa(T\inv A) \leq (1 + \gamma)/(1 - \gamma)$.
%\begin{equation}\label{eqn:kappa}
%\[
%\kappa(T\inv A) \leq \frac{1 + \gamma}{1 - \gamma}\; .
%\]
%\end{equation}
%Note: there is a typo in the book, ``='' instead of ``$\leq$''.
\end{theorem}

%Theorem~\ref{thm:A1} is a special case of a more general result, where $T_1$ and $T_2$
%are not necessarily exactly equal to the diagonal blocks, but satisfy the
%equivalence relations
%\begin{equation}\label{eqn:spectr_equiv}
%\begin{array}{c}
%\alpha_1 (A_{11}u,u) \leq (T_1 u, u) \leq \alpha_0 (A_{11}u,u),\\
%\beta_1 (A_{22}u,u) \leq (T_2 u, u) \leq \beta_0 (A_{22}u,u)
%\end{array}
%\end{equation}
%for some positive constants $\alpha_i$, $\beta_i$.

%\begin{theorem}[\cite{Axelsson:94}, Chapter 9]\label{thm:A2}
%Let $T_1$ and $T_2$ in~(\ref{eqn:bk_prec}) satisfy~(\ref{eqn:spectr_equiv}) with $\alpha_0 \geq \beta_0$, and let $A$ in~(\ref{eqn:bk_sys}) be SPD.
%Then
%\begin{eqnarray}
%%\label{eqn:kappa_gen}
%\nonumber
%\kappa(T\inv A) & \leq & \frac{\alpha_0}{\alpha_1 (1 - \gamma^2)} \left( \frac{1}{2}\left(1 + \frac{\alpha_1}{\beta_1}\right) + \left[ \left( \frac{1}{2}\left(1 - \frac{\alpha_1}{\beta_1}\right) \right)^2 + \frac{\alpha_1}{\beta_1}\gamma^2\right]^{1/2} \right) \times \\
%\nonumber
% & & \left( \frac{1}{2}\left(1 + \frac{\beta_0}{\alpha_0}\right) + \left[ \left( \frac{1}{2}\left(1 - \frac{\beta_0}{\alpha_0}\right) \right)^2 + \frac{\beta_0}{\alpha_0}\gamma^2\right]^{1/2} \right).
%\end{eqnarray}
%\end{theorem}

The bound given by Theorem~\ref{thm:A1} is sharp.
%In what follows, we use the CBS constants to
%construct a new approach for graph partitioning.
In the subsequent sections, we use this result to
derive a new partitioning algorithm.

%%The bounds given by Theorem~\ref{thm:A1} and Theorem~\ref{thm:A2} are sharp.
%We also note that the CBS constant appears in the expression for
%the bound of the condition number $\kappa(A_{22}^{-1} S)$ of the Schur complement $S = A_{22} - A_{12}^*A_{11}^{-1}A_{12}$ preconditioned with the $(2,2)$ block of $A$; see, e.g.,~\cite{Axelsson:94}.
%%In particular, if the nullspace of $A_{12}$ in~(\ref{eqn:bk_sys}) is nontrivial, then
%%\[
%%\kappa(A_{22}^{-1} S ) = \frac{1}{1 - \gamma^2}.
%%\]
%This makes the presented partitioning approaches, based on CBS constants, attractive for certain solution strategies based on Schur complements.
%%We do not discuss this in the present report.

\section{Partitioning using matrix coefficients}\label{sec:cbs_partn}
Given  decomposition $\left\{ V_k  \right\}_{k=1}^s$ of  the set  $V =
\left\{ 1,  2, \ldots, n \right\}$ (possibly  with overlapping $V_k$),
we  consider the AS  preconditioning for  an SPD  system  $Ax = b$.
The     preconditioning      procedure     is     given     in
Algorithm~\ref{alg:as}. By  $A(V_l,V_k)$ we denote a  submatrix of $A$
located at the intersection of  rows with indices in $V_l$ and columns
with  indices in $V_k$.   Similarly, $r(V_k)$  denotes the subvector of
$r$, containing entries from positions $V_k$.
In this section, we focus on the case where sets (subdomains)
$\left\{ V_k \right\}_{k=1}^s$ are nonoverlapping, i.e.,~(\ref{eqn:partn}) holds. This means that 
Algorithm~\ref{alg:as} gives a \textit{nonoverlapping} AS preconditioner.

\begin{algorithm}[AS preconditioner]\label{alg:as}
\emph{Input:} A, r, $\left\{ V_k \right\}_{k=1}^s$. \emph{Output:} $ w = T\inv r $.
%\vspace{0.05in}
\begin{enumerate}
\item For $k = 1, \ldots, s$, Do
  \item \hspace{0.2in}
         \emph{Set} $A_k := A(V_k,V_k)$, $r_k := r(V_k)$, \emph{and} $w_k = \textbf{0} \in \mathbb{R}^n$.
  \item  \hspace{0.2in}
        \emph{Solve} $A_k \delta  = r_k$.
  \item  \hspace{0.2in}
        \emph{Set} $w_k(V_k) := \delta$.
\item EndDo
\item $w = w_1 + \ldots + w_s$.
\end{enumerate}
\end{algorithm}
\vspace{0.1in}

%, which is a form of
%the block Jacobi iteration.
%In such a setting, Algorithm~\ref{eqn:sys} represents
%the simplest variant of a parallel domain decomposition type preconditioner.
%
%Let us be given a partition $\left\{ V_i \right\}_{i=1}^s$ of the set $V = \left\{ 1, \ldots, n \right\}$ satisfying~(\ref{eqn:partn}).

%Indeed,
Let $P$ be a permutation matrix
%Clearly, if $P$ is a permutation matrix,
which corresponds to the reordering of $V$ according to the
partition $\left\{ V_k \right\}_{k=1}^s$, where
the elements in $V_1$ are labeled first, in $V_2$ second, etc. Then the AS preconditioner $T$,
given by Algorithm~\ref{alg:as},
can be written in the matrix form as
%\begin{equation}\label{eqn:reorder}
$ T = P^T \bar T P$,  where $\bar T = \mbox{blockdiag}\left\{A_1 \ldots, A_s\right\}$
%\end{equation}
%where
and $A_k = A(V_k,V_k)$. Thus, Algorithm~\ref{alg:as} results in the block-diagonal, or block Jacobi,
preconditioner, up to a permutation of its rows and columns.

%In the following subsection we define an optimal \textit{bipartitioning} for Algorithm~\ref{alg:as}
%with two nonoverlapping subdomains.

\subsection{Optimal bipartitions}\label{subsec:opt_bi}

Let $s=2$, so that~(\ref{eqn:partn}) corresponds to a \textit{bipartition}
\begin{equation}\label{eqn:partitIJ}
V = I \cup J, \qquad I \cap J = \emptyset,  \qquad |I|= |J| = n/2 \; .
\end{equation}
%where $I \equiv V_1$ and $J \equiv V_2$.
%
Here, we assume that $n$ is even.
This guarantees the existence of \textit{fully balanced}
bipartitions, such that vertex sets $I$ and $J$ are of the same size, $n/2$.
Similarly, we assume that each connected component of $G(A)$
also has an even number of vertices.
The above assumptions, however, will not be a restriction
for the practical algorithm described below.

The following theorem provides a relation between a given bipartition and $\kappa(T\inv A)$.
The theorem is a direct consequence
 of Theorem~\ref{thm:A1} and is based on the fact that
symmetric permutations preserve the spectra.

\begin{theorem}
Let $\left\{ I, J \right\}$ in~(\ref{eqn:partitIJ}) be a bipartition of $V$
(possibly unbalanced).
%, where $|I| = m > 0$.
%
Let $T$ be the AS preconditioner for system $Ax=b$ with an SPD matrix $A$,
given by Algorithm~\ref{alg:as}, with respect to the bipartition $\left\{ I, J \right\}$.
Then,
\begin{equation}\label{eqn:kappa}
\kappa(T\inv A)  \leq \frac{1+\gamma_{IJ}}{1 - \gamma_{IJ}}\; ,
\end{equation}
where
\begin{equation}\label{eqn:cbsC}
\gamma_{IJ} = \max_{u \in W_I, v \in W_J} \frac{|(u,A v)|}{(u,A
u)^{1/2}(v, A v)^{1/2}}\; .
\end{equation}
The spaces $W_I$ and $W_J$ are the subspaces of $\mathbb{R}^n$ with dimensions $m$ and ($n-m$), respectively,
such that
\begin{equation}\label{eqn:WIJ}
W_I = \left\{ u \in \mathbb{R}^n \ : \ u(J) = \textbf{0} \right\}\; , \
%W_I = \left\{ u \in \mathbb{R}^n \ : \ u_k = 0, \ k \notin I \right\}, \
W_J = \left\{ v \in \mathbb{R}^n \ : \ v(I) = \textbf{0} \right\}\; .
%W_J = \left\{ v \in \mathbb{R}^n \ : \ v_k = 0, \ k \notin J \right\}.
\end{equation}
\end{theorem}

\iftoggle{LONG}{
\begin{proof}
%According to~(\ref{eqn:reorder}),
For the given bipartition $\left\{ I, J\right\}$ in~(\ref{eqn:partitIJ}), the preconditioner $T$,
constructed by Algorithm~\ref{alg:as}, is of the form
\begin{equation}\label{eqn:prec2x2}
T = P^T B P, \
B =
\left(
\begin{array}{cc}
 A_{I}      & 0   \\
 0        & A_{J}
\end{array}
\right)\; ,
\end{equation}
where $A_{I} = A(I,I)$, $A_J = A(J,J)$, and $P$ is a
permutation matrix corresponding to the reordering of $V$ with
respect to the partition $\left\{ I, J\right\}$.
%, so that the $n/2$ elements
%from $I$ are labeled first.
In particular, for any $x$, the vector $y=Px$ is such that $y = (x(I) \ x(J))^T$,
i.e., the entries of $x$ with indices in $I$ become the first $m$ components
of $y$, while the entries with indices in $J$ get
positions from $m + 1$ through $n$.

%Both, the coefficient matrix $A$ and the preconditioner $T$ in~(\ref{eqn:prec2x2})
%are SPD. Thus, it is desirable, in terms of the convergence speed of an
%iterative method, e.g., PCG,
%that $\left\{ I, J\right\}$ is chosen to provide
%the smallest, over all possible bi-partitions satisfying~(\ref{eqn:partitIJ}),
%
We observe that the condition number of the matrix $T\inv A$ is the same as the condition number
of the matrix $B\inv C$, where $C = P A P^T$
and $B$ in~(\ref{eqn:prec2x2}). Indeed, since a unitary
similarity transformation
\[
P(T\inv A)P^T = P(P^T B\inv P A) P^T = B\inv (P A P^T) = B\inv C\;,
\]
preserves the eigenvalues of $T\inv A$, we have
$\kappa(T\inv A) = \kappa(B\inv C)$,
where $\kappa(\cdot) = \lambda_{max}(\cdot)/\lambda_{min}(\cdot)$.

The matrix $C$ represents a symmetric permutation of $A$ with respect to the
given bipartition $\left\{ I, J\right\}$, and, thus, can be written in the $2$-by-$2$
block form,
\begin{equation}\label{eqn:C}
C = PAP^T=
\left(
\begin{array}{cc}
 A_{I}      &  A_{IJ}   \\
 A_{IJ}^*        & A_{J} \\
\end{array}
\right)\; ,
\end{equation}
where % the blocks
$A_{I} = A(I,I)$, $A_J = A(J,J)$, and $A_{IJ} = A(I,J)$.
Since $C$ is SPD and the preconditioner $B$ in~(\ref{eqn:prec2x2}) is the block diagonal of $C$,
we apply Theorem~{\ref{thm:A1}} to get the upper bound on the condition number $\kappa(B\inv C)$,
and hence bound~(\ref{eqn:kappa}) on $\kappa(T\inv A)$, where, according to Definition~\ref{def:cbs},
the CBS constant $\gamma \equiv \gamma_{IJ}$  is given by
\begin{eqnarray}
\nonumber \gamma_{IJ}        & = & \max_{w_1 \in W_1, w_2 \in W_2} \frac{| (w_1, C w_2) |}{(w_1,C w_1)^{1/2} (w_2, C w_2)^{1/2}} \\
\nonumber               & = & \max_{w_1 \in W_1, w_2 \in W_2} \frac{| (w_1, PAP^T w_2) |}{(w_1,PAP^T w_1)^{1/2} (w_2, PAP^T w_2)^{1/2}} \\
\nonumber               & = & \max_{w_1 \in W_1, w_2 \in W_2} \frac{|
(P^T w_1, AP^T w_2) |}{(P^T w_1, AP^T w_1)^{1/2} (P^T w_2, A P^T w_2)^{1/2}}\;.
\end{eqnarray}
The matrix $P^T$ defines the permutation that is the ``reverse'' of the one corresponding to $P$. Thus, the substitution
$u = P^T w_1$ and $v = P^T w_2$ leads to expression~(\ref{eqn:cbsC})--(\ref{eqn:WIJ}) for $\gamma_{IJ}$,
%the following expression for $\gamma$ with respect to the given bi-partition $\left\{ I,J\right\}$,
%\begin{equation}\label{eqn:cbsC}
%\gamma \equiv \gamma_{IJ} = \max_{u \in W_I, v \in W_J} \frac{|(u,A v)|}{(u,A u)^{1/2}(v, A v)^{1/2}},
%\end{equation}
%where $W_I$ and $W_J$ are the ($n/2$)-dimensional subspaces of $\mathbb{R}^n$, such that
%\begin{equation}\label{eqn:WIJ}
%W_I = \left\{ u \in \mathbb{R}^n \ : \ u_k = 0, \ k \notin I \right\}, \
%W_J = \left\{ v \in \mathbb{R}^n \ : \ v_k = 0, \ k \notin J \right\}, \
%\end{equation}
where the $W_I$ and $W_J$ contain vectors, which can have nonzero entries
only at positions defined by $I$ or $J$, respectively.
\end{proof}
}

%Since no simple exact expression is generally available for $\kappa(T\inv A)$,
The sharp upper bound~(\ref{eqn:kappa}) represents a meaningful indicator
of the preconditioner quality.
Thus, as an optimal bipartition, we can choose $\left\{ I, J \right\}$, such that
$(1+\gamma_{IJ})/(1-\gamma_{IJ})$ or, equivalently, the CBS constant $\gamma_{IJ}$,
is minimized.
More precisely, we define an optimal bipartition $\left\{ I_{opt},J_{opt} \right\}$
%for the nonoverlapping AS preconditioner
to be such that
\begin{equation}\label{eqn:opt}
\left\{ I_{opt},J_{opt} \right\}  =
\underset{
\begin{array}{c}
\scriptstyle{I, J \subset V = \left\{ 1, \ldots, n \right\}}\; ,  \\
\scriptstyle{|I| = |J| = \frac{n}{2}, J = V \setminus I}
\end{array}
}{\operatorname{argmin}}
%\max_{u \in W_I, v \in W_J} \frac{|(u,A v)|}{(u,A u)^{1/2}(v, A v)^{1/2}}\;,
\gamma_{IJ} \;,
\end{equation}
where
%$W_I$ and $W_J$ are the subspaces defined in~(\ref{eqn:WIJ}).
$\gamma_{IJ}$ is defined in~(\ref{eqn:cbsC}).

%\textcolor{blue}{Need to clarify the 'it is not clear' -- replaced:
%In practice,
%it is not clear if optimization problem~(\ref{eqn:opt}) can be efficiently
%solved. 
%By:} 
A straightforward solution of 
optimization problem~\eqref{eqn:opt} entails evaluating 
\eqref{eqn:cbsC} for a very large, namely
$n!/[(n/2)!]^2$, possible choices of the partitions $\left\{ I, J \right\}$.
While this indicates that the problem is likely to be NP-hard,
it is not clear whether an efficient solution to~\eqref{eqn:opt} can
be found.

Therefore, our idea is to replace~(\ref{eqn:opt})
by a related simpler problem, such that the minimizer of the latter
%The solution of such a problem
%which only
\textit{approximates} (in terms of smallness of the CBS constant)
$\left\{ I_{opt}, J_{opt}\right\}$
rather than determines it exactly.
%without determining it exactly.
%
Below, we discuss several approaches.

%\subsection{Approximating optimal bipartitions. The minimal averaged cut (Acut)}\label{subsec:uv}
\subsection{The minimal averaged cut}\label{subsec:uv}
In order to obtain a simpler optimization problem, let us replace $\gamma_{IJ}$
in~(\ref{eqn:opt}) by some $\tilde \gamma_{IJ}$, such that
%for each $\left\{ I, J \right\}$ the new
$\tilde \gamma_{IJ}$
captures information on $\gamma_{IJ}$ and is easy to compute.
We define $\tilde \gamma_{IJ}$ as following.

Given $\left\{ I, J \right\}$,
instead of maximizing the ratio
\begin{equation}\label{eqn:qty}
\frac{|(u,A v)|}{(u,A u)^{1/2}(v, A v)^{1/2}}\; ,
\end{equation}
as required for computing $\gamma_{IJ}$ in~(\ref{eqn:cbsC}),
we introduce a set of pairs
\begin{equation}\label{eqn:sample}
S = \left\{ (e_i, e_j) \ : \ i \in I, j \in J, a_{ij} \neq 0 \right\} \subset W_I \times W_J \; ,
\end{equation}
where $e_k \in \mathbb{R}^n$ denotes the unit vector with $1$
at position $k$,
and calculate~(\ref{eqn:qty}) on all $(e_i, e_j) \in S$.
Note that the cardinality of $S$ is equal to the size of the cut
between $I$ and $J$, further denoted by $\mbox{cut}(I,J)$.
The resulting values are averaged.
This gives a constant $\tilde{\gamma}_{IJ}$, such that
\begin{equation}\label{eqn:subopt}
\tilde{\gamma}_{IJ} = \frac{w(I,J)}{\mbox{cut}(I,J)}\;,
\end{equation}
%where $w(I,J) = \sum_{i \in I, j \in J}  |a_{ij}|/\sqrt{a_{ii} a_{jj}}$.
%represents the weighted cut,
where $w(I,J) = \sum_{i \in I, j \in J} w_{ij}$, $w_{ij} = |a_{ij}|/\sqrt{a_{ii} a_{jj}} \in (0,1)$.

The constant $\tilde \gamma_{IJ}$ %in~(\ref{eqn:subopt})
has a transparent meaning in terms of the adjacency graph $G(A)$.
Assigning the weights $w_{ij}$ to the edges $(i,j)$,~(\ref{eqn:subopt})
can be interpreted as a ratio of the cut weight, $w(I,J)$,
to the cut size; or, equivalently, as the \textit{averaged} cut weight.
At the same time, $\tilde \gamma_{IJ}$ is closely related to $\gamma_{IJ}$. In
particular, $\tilde \gamma_{IJ} \leq \gamma_{IJ}$ for all $\left\{ I, J \right\}$.
Thus, we %suggest to
replace optimization problem~(\ref{eqn:opt}) by %the problem
%a problem of  a minimizer for $\gamma_{IJ}$ over all balanced bipartitions
%$\left\{ I, J \right\}$, by the problem
%of finding $\left\{ \tilde{I}_{opt}, \tilde{J}_{opt} \right\}$, such that
%
\begin{equation}\label{eqn:min_subopt}
\{ \tilde I_{opt}, \tilde J_{opt} \}  =
\underset{
\begin{array}{c}
\scriptstyle{I, J \subset V = \left\{ 1, \ldots, n \right\}}\; ,  \\
\scriptstyle{|I| = |J| = \frac{n}{2}, J = V \setminus I}
\end{array}
}{\operatorname{argmin}}
%\max_{u \in W_I, v \in W_J} \frac{|(u,A v)|}{(u,A u)^{1/2}(v, A v)^{1/2}}\;,
\tilde \gamma_{IJ} \;.
\end{equation}
%where $\tilde \gamma_{IJ}$ is defined in~(\ref{eqn:subopt}).
%
The minimizer $\{ \tilde I_{opt}, \tilde J_{opt} \}$
%in~(\ref{eqn:subopt})--(\ref{eqn:min_subopt})
is expected to approximate $\left\{ I_{opt}, J_{opt} \right\}$.
%so that the resulting preconditioner has quality close to optimal.
Formally,~(\ref{eqn:subopt})--(\ref{eqn:min_subopt}) represents %is
the problem of graph bipartitioning,
where the targeted cut has the smallest, in average, weight.
We call such a cut \textit{the minimal averaged cut (Acut)}.

%Along with minimizing the cut weight, which is the task of standard
%partitioning schemes, objective~(\ref{eqn:subopt}) also contains the concurrent
%requirement on maximizing the number of edges included in the cut.
Minimization of objective~(\ref{eqn:subopt}) is achieved by bipartitions that
provide a balance between the two concurrent requirements on  
minimizing the cut weight and maximizing the number of edges included in the cut.
%
%Such a modification leads to constructing cuts that go through a large number of
%``light-weighted'' edges.
Thus,~(\ref{eqn:subopt})--(\ref{eqn:min_subopt}) gives cuts that include a relatively large number of 
``light-weighted'' edges.
As shown in our numerical experiments, for a number of systems arising from discretizations
of partial differential equations with strongly varying coefficients,
this allows us to recognize the boundaries of subregions that correspond to different coefficient magnitudes.
The resulting subdomains are experimentally shown to 
%have shapes that 
improve preconditioning quality.

%
%As a result, Acut
%%is able to
%%and, accordingly,
%determine the shapes of subdomains
%that give an improved preconditioning quality.

\subsection{Spectral Acut computations}\label{subsec:spectr_Acut}
Optimization problem~(\ref{eqn:min_subopt}) is still hard to solve exactly
since it encompasses the original graph partitioning problem which is 
known to be NP-complete~\cite{SimTeng97}. 
Therefore, we address below the question of \textit{approximating} Acut
for practical applications.
We propose a spectral bipartitioning technique.

%Let us reformulate optimization problem~(\ref{eqn:min_subopt}) in terms of bilinear forms involving
%graph Laplacians.
%
Let $p$ denote the indicator vector of size $n$,
with the components defined as
\begin{equation}\label{eqn:indicator}
p(k) = \left\{
\begin{array}{c}
\ 1,  \ k \in I\; , \\
-1, \ k \in J\; .
\end{array}
\right.
\end{equation}
Then, for a given $\left\{ I, J \right\}$,
\iftoggle{LONG}{
\begin{eqnarray}
\nonumber 4 w (I,J) & = & \sum_{(i,j) \in E} w_{ij} (p(i) - p(j))^2 = \sum_{(i,j) \in E} w_{ij}(p(i)^2 + p(j)^2) - 2 \sum_{(i,j) \in E} w_{ij} p(i) p(j) \\
\nonumber           & = & \sum_{i=1}^n d_w(i) p(i)^2 - \sum_{i,j = 1}^n w_{ij} p(i) p(j)\; ,
\end{eqnarray}
}{
\[
4 w (I,J) = \sum_{i=1}^n d_w(i) p(i)^2 - \sum_{i,j = 1}^n  w_{ij} p(i)p(j) \;,
\]
}
where $d_w(i) = \sum_{j \in N(i)} w_{ij}$ is the weighted degree of the vertex $i$; $N(i)$ denotes the vertices adjacent to $i$.
Thus, $w(I,J)$ can be written as a bilinear form, %evaluated at the indicator vector $p$,
\begin{equation}\label{eqn:wcut}
4 w(I,J) = p^T L_w p, \qquad L_w = D_w - W\; ,
\end{equation}
where $D_w = \mbox{diag}(d_w(1), \ldots, d_w(n))$ is the weighted degree matrix and
$W = (w_{ij})$ denotes the weighted adjacency matrix.
%($w_{ij} = w_{ji} = \frac{|a_{ij}|}{\sqrt{a_{ii} a_{jj}}}$,  $w_{ii}=0$) of $G(A)$,
%and $L_w$ is the weighted graph Laplacian.
%
Similarly,
%for the same bipartition $\left\{ I, J \right\}$,
%we repeat the above derivations with
setting $w_{ij} = 1$ for all edges $(i,j)$ of $G(A)$, we get the expression for $\mbox{cut}(I,J)$,
\begin{equation}\label{eqn:cut}
4 \mbox{cut}(I,J) = p^T L p, \qquad L = D - Q\; ,
\end{equation}
where $D$ is the degree matrix and $Q=(q_{ij})$ is the adjacency matrix.
%($q_{ij} = q_{ji} = 1$, iff $(i,j) \in E$, and $q_{ij} = 0$ otherwise, $q_{ii}=0$) of $G(A)$,
%and $L$ denotes the standard graph Laplacian.
The matrices $L_w$ and $L$ denote the weighted and unweighted \textit{graph Laplacians}.
Thus, given $\{I,J\}$,~(\ref{eqn:wcut}) and (\ref{eqn:cut}) allow representing~(\ref{eqn:subopt}) as a ratio
of two bilinear forms, i.e., $\tilde \gamma_{IJ} = p^T L_w p/p^T L p$.

Let us assume that $G(A)$ has $q \geq 1$ connected components $(V_l,E_l)$,
where $|V_l|$ are even.
We introduce vectors $z_1, \ldots, z_q$, such that
\begin{equation}\label{eqn:z}
z_l(k) = \left\{
\begin{array}{l}
1,  \ k \in V_l\; , \\
0, \ k \notin V_l\; ,
\end{array}
\right.
\end{equation}
i.e., the entries of $z_l$ corresponding to vertices in $V_l$ are $1$, and $0$ elsewhere.
Note that if $q=1$, then we obtain a single vector of ones.

Problem~(\ref{eqn:min_subopt}) can now be written as
\begin{equation}\label{eqn:rq_discrete}
\tilde p_{opt}  =
\underset{p}{\operatorname{argmin}}
%\min_{p}
\frac{p^T L_w p}{p^T L p}, \qquad p^T z_l = 0, \ \qquad l = 1,\ldots,q\; ,
\end{equation}
where the minimizer $\tilde p_{opt}$ is searched over all feasible indicator vectors.
The condition $p^T z_l = 0$ ensures that all the connected components
are bipartitioned into two equal-sized sets of vertices. Hence, $|I| = |J| = n/2$, as required.
%Note that if $q=1$, then we obtain a single vector of ones.

%$\textbf{1} = (1 \ 1 \ \ldots \ 1)^T$ is $n$-vector of ones.

In order to approximate solution of~(\ref{eqn:rq_discrete}), we relax the problem by
embedding it into the real space.
More specifically,
we consider the minimization
%instead of~(\ref{eqn:rq_discrete}), we consider
%attempt to find $v \in \mathbb{R}^n$, such that
\begin{equation}\label{eqn:rq}
\min_{v \in \mathbb{R}^n}
\frac{v^T L_w v}{v^T L v}, \qquad v \in  \mbox{span}\{z_1,\ldots,z_q\}^{\perp},
%\min_{v}
%\frac{v^T L_w v}{v^T L v}, \qquad v^T \textbf{1} = 0\; .
\end{equation}
of the \textit{generalized Rayleigh quotient} on the orthogonal complement of the subspace spanned
by vectors $z_l$ in~(\ref{eqn:z}). We expect the minimizer of~(\ref{eqn:rq}) to provide an approximation
to the optimal indicator vector $\tilde p_{opt}$ from~(\ref{eqn:rq_discrete}).

Both $L_w$ and $L$ are symmetric positive semi-definite, with the dimension of the nullspace equal to the number of
connected components of $G(A)$. In particular,
%note that
\[
\mbox{null}(L_w) = \mbox{null}(L) = \mbox{span}\{z_1, \ldots, z_q\}.
\]
Thus, $v^T L_w v/v^T L v$ in~(\ref{eqn:rq}) is minimized on the orthogonal complement
of the nullspace of the two matrices, where both $L$ and $L_w$ are SPD.
This implies that the minimum in~(\ref{eqn:rq}) exists.
It is achieved on the eigenvector associated with smallest eigenvalue of the symmetric generalized eigenvalue problem
\begin{equation}\label{eqn:evp}
L_w v = \lambda L v, \qquad v \in  \mbox{span}\{z_1,\ldots,z_q\}^{\perp} \;.
\end{equation}
%restricted to $\mbox{null}(L_w)^{\perp} = \mbox{null}(L)^\perp$,

Solution of~(\ref{eqn:evp}) can be viewed as an analogue of the Fiedler vector~\cite{Fiedler:73, Fiedler:75}.
%Given the minimizing eigenvector,
%%associated with the smallest eigenvalue of~(\ref{eqn:evp}),
%we form
The bipartition is formed by assigning the indices of its $\lceil n/2 \rceil$ smallest components
to $I$ and the rest to $J$. As shown in our numerical experiments,
%unlike the Fiedler vector,
the eigenvector of~(\ref{eqn:evp}) may
deliver disconnected subdomains, even though the original graph is connected.
%
%If the load balancing requirement can be slightly violated,
%then other ways for constructing bipartitions are possible, e.g.,
%based on explicitly checking the values of $p^T L_w p/p^T L p$ for a few
%candidate bipartitions given by the eigenvector.
%In this paper, however, we do not consider such approaches.
%%
Note that the assumption on the even sizes of the vertex sets of $G(A)$
and its connected components is not restrictive any more, and is skipped
for relaxed problem~(\ref{eqn:evp}).

Finally, let us observe that if the weights $w_{ij}$ % $=|a_{ij}|/\sqrt{a_{ii} a_{jj}}$
are the same for all edges, then the graph Laplacians $L_w$ and $L$
represent multiples of each other.
%
%This happens for matrices $A$ with a very special, regular, behavior of entries,
%e.g., for the discrete Laplace operator with the constant coefficients.
%
In this case, the solution of~(\ref{eqn:evp}) is given by multiple orthogonal
eigenvectors that
correspond to the only nonzero eigenvalue of multiplicity $n-q$,
i.e., the result of spectral Acut computations is highly uncertain.
Such a situation is an indicator of the fact that
the coefficient matrix $A$ has extremely regular behavior of its entries.
Therefore, if all $w_{ij}$ are the same, or only slightly different,
we suggest using \textit{standard} partitioning criteria. %, e.g., as the one addressed below.

\subsection{The minimal weighted cut}\label{subsec:mincut}

We now relate the problem of minimizing the CBS constant to the
\textit{standard} objective for graph partitioning.
In particular, we consider replacing~(\ref{eqn:opt}) by minimization of
the cut weight under the load balance constraint.

Given a bipartition $\left\{ I, J \right\}$ in~(\ref{eqn:partitIJ}),
similarly to~(\ref{eqn:sample}),
we define the set
\begin{equation}\label{eqn:sample_mincut}
\bar S = \left\{ (e_i, e_j) \ : \ i \in I, j \in J \right\} \subset W_I \times W_J\; .
\end{equation}
%where $e_k \in \mathbb{R}^n$ denotes the unit vector with $1$ at position $k$
%and zeros elsewhere.
%
Unlike~(\ref{eqn:sample}), (\ref{eqn:sample_mincut})
contains \textit{all} pairs $(e_i, e_j)$ with $i \in I$ and $j \in J$,
including those that correspond to $a_{ij}=0$.
Instead of maximizing~(\ref{eqn:qty}), as required to compute $\gamma_{IJ}$,
we evaluate~(\ref{eqn:qty}) on all $(n/2)^2$ pairs in~(\ref{eqn:sample_mincut})
and then find the average.
Thus, for a given $\{I,J\}$ we define a constant
$\bar{\gamma}_{IJ} \leq \gamma_{IJ}$, such that
\begin{equation}\label{eqn:subopt_mincut}
\bar{\gamma}_{IJ} =  \frac{4}{n^2} w(I,J)\; ,
\end{equation}
where $w(I,J)$ is the cut weight, defined as in~(\ref{eqn:subopt}).

Following the pattern of the previous subsections,
instead of~(\ref{eqn:opt}), we suggest solving optimization
problem
%minimize $g(I,J) = \bar{\gamma}_{IJ}$ in~(\ref{eqn:subopt_mincut}), i.e., find
%such a bipartition $\left\{ \bar{I}_{opt}, \bar{J}_{opt} \right\}$ that
\begin{equation}\label{eqn:min_subopt_mincut}
\{ \bar I_{opt}, \bar J_{opt} \}  =
\underset{
\begin{array}{c}
\scriptstyle{I, J \subset V = \left\{ 1, \ldots, n \right\}}\; ,  \\
\scriptstyle{|I| = |J| = \frac{n}{2}, J = V \setminus I}
\end{array}
}{\operatorname{argmin}}
\bar \gamma_{IJ} \;,
\end{equation}
where $\{\bar I_{opt}, \bar J_{opt}\}$ is expected to approximate
an optimal $\{I_{opt},J_{opt}\}$ in~(\ref{eqn:opt}).
It is readily seen that~(\ref{eqn:subopt_mincut})--(\ref{eqn:min_subopt_mincut}) delivers the well-known
problem of graph partitioning, which aims at finding
equal-sized vertex sets $I$ and $J$ with the minimal cut weight.
The solution of~(\ref{eqn:subopt_mincut})--(\ref{eqn:min_subopt_mincut}) can be approximated
by an available graph partitioning scheme which admits edge weighting. In our numerical examples,
we use MeTiS.
%by a number of heuristics.

The principal difference between~(\ref{eqn:subopt})--(\ref{eqn:min_subopt})
and~(\ref{eqn:subopt_mincut})--(\ref{eqn:min_subopt_mincut}) is in
that~(\ref{eqn:subopt_mincut})--(\ref{eqn:min_subopt_mincut}) delivers cuts of
a \textit{cumulatively} small weight, which may generally contain both ``heavy-weighted''
and ``light-weighted'' edges.
At the same time, cuts given
by~(\ref{eqn:subopt})--(\ref{eqn:min_subopt}) favor edges with small weights.
For example if $A$ is given by a PDE with jumps in coefficients,
%Assuming large variations in matrix entries,
such edges connect subdomains corresponding to different coefficient magnitudes, whereas
edges with large $w_{ij}$ are in the interior of these subdomains.
Therefore,
unless forced by the load balance constraint,
%unlike~(\ref{eqn:subopt_mincut})--(\ref{eqn:min_subopt_mincut}),
(\ref{eqn:subopt})--(\ref{eqn:min_subopt}) excludes the undesirable possibility of ``cutting''
inside jump subregions.

%Therefore, below, we focus on the Acut computations.
%In our numerical experiments,
%MINcut is computed by the standard Recursive Spectral Bisection (RSB), see, e.g.,~\cite{Saad:03},
%and MeTiS~\cite{Metis:09}.

%In particular, repeating the derivations in Subsection~\ref{subsec:uv} for this problem,
%we can define the bipartition $\left\{ I, J\right\}$
%according to the components of the eigenvector corresponding to the smallest eigenvalue of
%\begin{equation}\label{eqn:evp_mincut}
%L_w v = \lambda v, \qquad v \in \textbf{1}^{\perp}\; ,
%\end{equation}
%where $L_w$ is the \textit{weighted} graph Laplacian. The recursive application of this procedure
%leads to the partitioning scheme, which is similar to the standard RSB algorithm
% with the difference that now edge weights $w_{ij} = \frac{a_{ij}}{\sqrt{a_{ii} a_{jj}}}$ are encapsulated into the graph Laplacian.

\subsection{Acut by recursive spectral bisection}\label{subsec:recur}

Let $\left\{ I, J \right\}$ be given by Acut described in
Subsection~\ref{subsec:spectr_Acut}.
A natural way to construct further partitions~(\ref{eqn:partn}) is to apply the bipartitioning process separately for
subgraphs of $G(A)$ corresponding to $I$ and $J$,
then to the resulting partitions and so on, until all the subpartitions are sufficiently small.
We summarize this in the following algorithm. % that we refer to as the Acut-RSB.

\vspace{0.1in}
\begin{algorithm}[Acut-RSB($A$)]\label{alg:partn1}
\emph{Input:} $A$. \emph{Output:} Partition $\left\{ V_i \right\}$. % of $\left\{ 1,2, \ldots,n \right\}$.
%\vspace{0.2in}
\begin{enumerate}
\item \emph{Form} $G(A)$. \emph{Assign weights} $w_{ij} = |a_{ij}|/\sqrt{a_{ii} a_{jj}}$.
\item \emph{Construct} $L_w = D_w - W$ \emph{and} $L = D - Q$.
\item \emph{Find connected components} $\{E_l,V_l\}$. \emph{Define} $z_l$ \emph{in}~(\ref{eqn:z}).
\item \emph{Find the eigenvector associated with the smallest eigenvalue of}~(\ref{eqn:evp}).
\item \emph{Define} $\{I,J\}$ \emph{based on the computed eigenvector}.
\item \emph{Apply the algorithm recursively:} \\
\textbf{If} $|I| > \mbox{\em maxSize}$ \textbf{then} \emph{call Acut-RSB}($A(I,I)$), \textbf{else} \emph{return} $I$.\\
\textbf{If} $|J| > \mbox{\em maxSize}$ \textbf{then} \emph{call Acut-RSB}($A(J,J)$), \textbf{else} \emph{return} $J$. \\
\end{enumerate}
\end{algorithm}
\vspace{0.1in}

The parameter {\em maxSize} in Algorithm~\ref{alg:partn1}
is provided by the user,
and should be chosen to ensure that $|V_k| \approx n/s$.
%the desired number $s$ of sets $V_k$ in~(\ref{eqn:partn}) is obtained.
%
The connected components in Step 3 can be detected by standard algorithms based on the
breadth-first search or the depth-first search~\cite{Cormen:2001}.
Note that weights $w_{ij}$ are the same on all levels of the
recursion. In practice, they are assigned only once, on the top level.

Clearly, if eigenvalue problem~(\ref{eqn:evp}) in step 4 is replaced by
computing the Fiedler vector of either $L_w v = \lambda v$ or $L v = \lambda v$,
then Algorithm~\ref{alg:partn1} reduces to the well-known RSB scheme; see, e.g.,~\cite{Saad:03}.
We suggest that before the recursive Acut-RSB calls in Step 6, one checks if
$A(I,I)$ and $A(J,J)$ indeed have large variations in magnitudes of their entries,
e.g., by assessing the variance of the coefficients.
If the entries of $A(I,I)$ or $A(J,J)$ exhibit a regular behavior, we recommend
invoking one of the \textit{standard}, unweighted, partitioning algorithms.

\section{Numerical results}\label{sec:numer}

%In this section, we demonstrate the effects of different
%partitioning schemes on the preconditioning quality.
%
%In particular,
In this section,
we apply Acut-RSB to several SPD test problems and
compare the partitioning results to more traditional approaches
delivered by the MeTiS algorithm with various weighting schemes.
For each test problem, we use the partitions produced by all %different
algorithms to construct the AS preconditioners, which are then
supplied to the PCG iteration. We further refer to this solution
scheme as PCG--AS.

In our examples, we consider three types of MeTiS partitions.
The first type results from applying the algorithm to the \textit{unweighted}
adjacency graph, so that the coefficient information is completely skipped.
In contrast,
the other two types are obtained by running MeTiS for the weighted
graphs, where the weights are based on the matrix coefficients.
In particular, we use weighting schemes $y = \{ y_{ij} \}$ and
$t = \{ t_{ij}\}$, with the edge weights
$y_{ij} = \lceil \: \gamma |a_{ij}| / \sqrt{a_{ii} a_{jj}} \: \rceil$ and
%$y_{ij} = \lceil \: \gamma w_{ij} \: \rceil$ and
%
$t_{ij} =  \lceil \: \delta |a_{ij}|  \: \rceil$.
Here, the parameter $\gamma$ is a large integer and $\delta >0$.
%is
%chosen from $(0,1]$.
%
We denote the corresponding partitioning methods by MeTiS($y$) and MeTiS($t$),
respectively.
Note that a closely related weighting scheme has been used in~\cite{Have.Masson.Nataf.Szydlarski.Zhao:11}.

Let us remark that $y_{ij}$ can be written as $y_{ij} = \lceil \gamma w_{ij} \rceil$,
where $w_{ij}$ are the weights
in the CBS constant based partitioning formulations~(\ref{eqn:subopt})--(\ref{eqn:min_subopt})
and~(\ref{eqn:subopt_mincut})--(\ref{eqn:min_subopt_mincut}).
Thus, whereas Acut-RSB targets~(\ref{eqn:subopt})--(\ref{eqn:min_subopt}),
MeTiS($y$) is expected to approximate partitions in~(\ref{eqn:subopt_mincut})--(\ref{eqn:min_subopt_mincut}).
The constant $\gamma$ and the ``ceil'' operation in the definition of $y_{ij}$
are introduced to satisfy the MeTiS requirement on the integer weight values 
(recall that $w_{ij} \in (0,1)$).

In our numerical experiments, we have observed that MeTiS($y$) gives better results than
the RSB scheme based
on the Fiedler vector for $L_w v = \lambda v$, which approximates solution
of~(\ref{eqn:subopt_mincut})--(\ref{eqn:min_subopt_mincut}). % after relaxing the problem.
Therefore, the latter is not reported in the examples below.

The choice of $t_{ij} = \lceil \: \delta |a_{ij}|  \: \rceil$ for MeTiS($t$) is motivated 
by the computational experience,
suggesting that each edge should be weighted with (essentially) the absolute value of the
corresponding matrix entry. While straightforward and well-known among 
practitioners,
we were not able to find direct references to such a weighting 
scheme in the literature.

In our tests, we define $t_{ij}$ with $\delta$ from $(0,1]$. This choice of $\delta$ ``damps'' the
magnitudes of $|a_{ij}|$. It ensures that the weighted cut sizes computed by MeTiS($t$) do not exceed
the upper limit of the integer data type size, in which case the behavior of the algorithm 
can be unpredictable. 
Similar to $y_{ij}$, the ``ceil'' operation is introduced to maintain integer
weights. Note that if $A$ is diagonally scaled, $y_{ij}$ and $t_{ij}$ coincide.

Throughout, the reported partitioning results produced by MeTiS($y$) and MeTiS($t$) correspond to the
best values of $\gamma$ and $\delta$ observed during our experiments. In all cases,
MeTiS has been invoked with the ``PartGraphRecursive'' option, which enables the \textit{recursive
bipartitioning}~\cite{Karypis.Kumar_kway:98} and guarantees the \textit{strict} load balance.
Note that all our experiments are performed in {\sc matlab}; the MeTiS library is accessed
through the MEX interface.

Unlike MeTiS, which represents a combinatorial partitioning technique,
the new Acut-RSB algorithm relies on computing eigenvectors.
In all of our tests, as an eigensolver, we use Locally Optimal Block Preconditioned Conjugate
Gradient (LOBPCG) method~\cite{Knyazev:01}.
Our choice has been motivated mainly by the fact that,
unlike the  Lanczos algorithm~\cite{Parlett:98},  LOBPCG can solve
\textit{generalized} eigenvalue problems, such as~(\ref{eqn:evp}),
without requiring
any factorization of the singular matrices $L_w$ and $L$.
Additionally, LOBPCG encapsulates preconditioning to accelerate convergence
%can be easily configured to perform iterations on the subspace $\mbox{span}\{ z_1, \ldots, z_q \}^{\perp}$,
and has a relatively modest storage requirement.

%i.e.,
%represent the spectral partitioning techniques.
%
%We recall that Algorithm~\ref{alg:partn1} computes an eigenpair of problem~(\ref{eqn:evp}).
%The RSB algorithm targets the Fiedler vector in~(\ref{eqn:fiedler}). The approaches
%based on the (weighted) MINcut and Mcut use the eigenvectors of problems~(\ref{eqn:evp_mincut})
%and~(\ref{eqn:evp_mcut}), respectively.
%%

The LOBPCG algorithm is a form of a (block) three-term recurrence, which locally
optimizes the Rayleigh quotient; see~\cite{Knyazev:01} for more details.
Given a suitable SPD preconditioner, the method is known to be efficient
for large-scale eigenvalue computations.
In our experiments, we construct preconditioners using 
%Incomplete Cholesky (IC) 
IC factorization
of matrices $L_w + \sigma I$ with a drop tolerance of $10^{-3}$.
To ensure that the IC procedure
is correctly applied to the SPD matrices,
the parameter $\sigma$ is assigned a small positive value. In particular, we set $\sigma = 0.1$.

In all LOBPCG runs, % for problem~(\ref{eqn:evp}),
%we set the block size to $1$.
we choose a random initial guess from $\mbox{span}\{ z_1, \ldots, z_q \}^{\perp}$
and at each step project the preconditioned residuals to this subspace,
so that the LOBPCG iterations are kept in $\mbox{span}\{ z_1, \ldots, z_q \}^{\perp}$.
Note that the amount of storage and computations required to orthogonalize
against $\mbox{span}\{ z_1, \ldots, z_q \}^{\perp}$ does not depend on $q$, i.e.,
is the same as for the $n$-vector of all ones.

Finally, let us remark that the partitioning runtimes exhibited by MeTiS
have been notably lower than those of Acut-RSB. This can be attributed 
to the spectral nature of the new algorithm. Additionally, the runtime difference has been
exacerbated by the fact that, in contrast to the MEX-interfaced MeTiS library calls,
our Acut-RSB code is purely in {\sc matlab}. 
%including the eigensolver, preconditioning,
%and breadth-first search. 

However, as demonstrated below, despite longer runtimes, Acut-RSB often leads 
to a significantly smaller number of iterations of the linear solver. 
In this context, the partitioning cost (i.e., the pre-processing phase of the preconditioner
set up) becomes less important. In fact, our goal is precisely this: to get a more reliable iterative
solution procedure by paying a higher cost in the preprocessing.

\paragraph{\textbf{\emph{$2$D diffusion equation}}}

Let us consider the diffusion equation on a unit square,
\begin{equation}\label{eqn:pde}
-\frac{\partial}{\partial x} \left( a(x,y) \frac{\partial u}{\partial x}\right) -
\frac{\partial}{\partial y} \left( b(x,y) \frac{\partial u}{\partial
y}\right) = f(x,y)\;, \ (x,y) \in [0,1] \times [0,1],
\end{equation}
with zero Dirichlet boundary conditions and
the coefficients given by the piecewise constants
\begin{equation}\label{eqn:j_xy}
%\[
a(x,y) = b(x,y) = \left\{\begin{array}{cl}
	10^5, & 0.25<x,y<0.75 \\
	1,   & \mbox{otherwise}\; ;
	   \end{array}\right.
%\]
\end{equation}
which strongly vary (jump) across the two subdomains; see Figure~\ref{fig:grid1}.
%
%Here, and throughout, we choose the coefficients $a(x,y)$ and $b(x,y)$ to be piecewise constants,
%which strongly vary (jump) across prespecified subdomains. In particular

\begin{figure}[h]
 \begin{center}
 \includegraphics[width = 6cm]{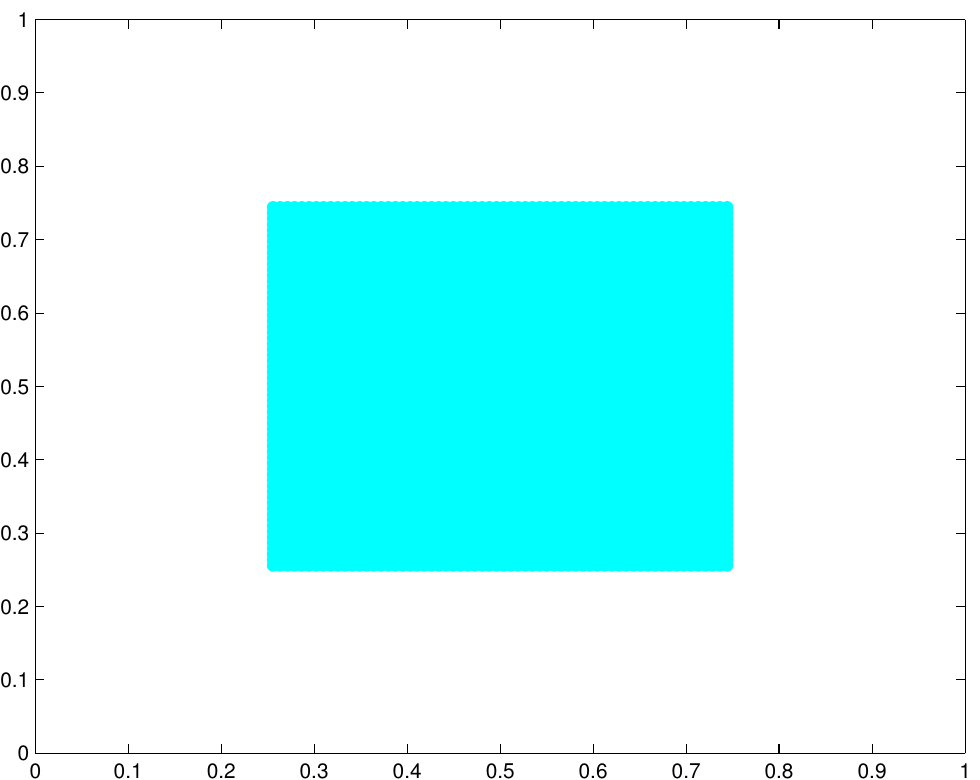}
 \end{center}
 \caption{Location of the jump in coefficients~(\ref{eqn:j_xy}) for problem~(\ref{eqn:pde}).
  The shaded (cyan for the color plot) subdomain corresponds to the jump region $(0.25,0.75) \times (0.25,0.75)$.}
 \label{fig:grid1}
\end{figure}

%In the current example, we consider two geometries for the jump locations.
%%
%In the first case, we set $a(x,y) = b(x,y) = 1$ with $a(x,y) = b(x,y) = 100$
%in the square subregion $0.25 < x,y < 0.75$.
%%
%The second geometry is somewhat more complex, with
%the jumps located in the ``checkerboard'' fashion (``$5$-by-$5$ black-white checkerboard''), where
%$a(x,y) = b(x,y) = 1$ in ``white'' positions and $a(x,y) = b(x,y) = 100$ in ``black'';
%see Figure~\ref{fig:grids}.

In order to discretize~(\ref{eqn:pde}), we introduce a $128$-by-$128$ (interior points)
uniform grid and use the standard $5$-point finite difference (FD) stencil.
The resulting linear system, $Ax = b$, is SPD of size $n=16,384$.
For testing purposes, the right-hand side $b$ is randomly chosen.
Here, and below, the FD matrices have been generated using the SPARSKIT library~\cite{Saad:SPARSKIT}.

Figure~\ref{fig:bipartn1} shows the \textit{bipartitions} produced by different methods.
It is readily seen that Acut-RSB precisely detects the jump region and places most
of the cut along its boundary, without ``cutting'' inside.
In the framework of the DD type methods,
the latter is consistent with a common recommendation
to include the regions corresponding to different model parameters into
separate subdomains.
The eigenvector, used to define the bipartition in Acut-RSB, is shown in Figure~\ref{fig:ev1} (left).
The values of $\gamma$ and $\delta$ for the weights in MeTiS($y$) and MeTiS($t$) have been set
to $10^5$ and $1$, respectively. The convergence tolerance for LOBPCG is $10^{-7}$.

\begin{figure}[h]
 \begin{center}
 \includegraphics[width = 6.0cm]{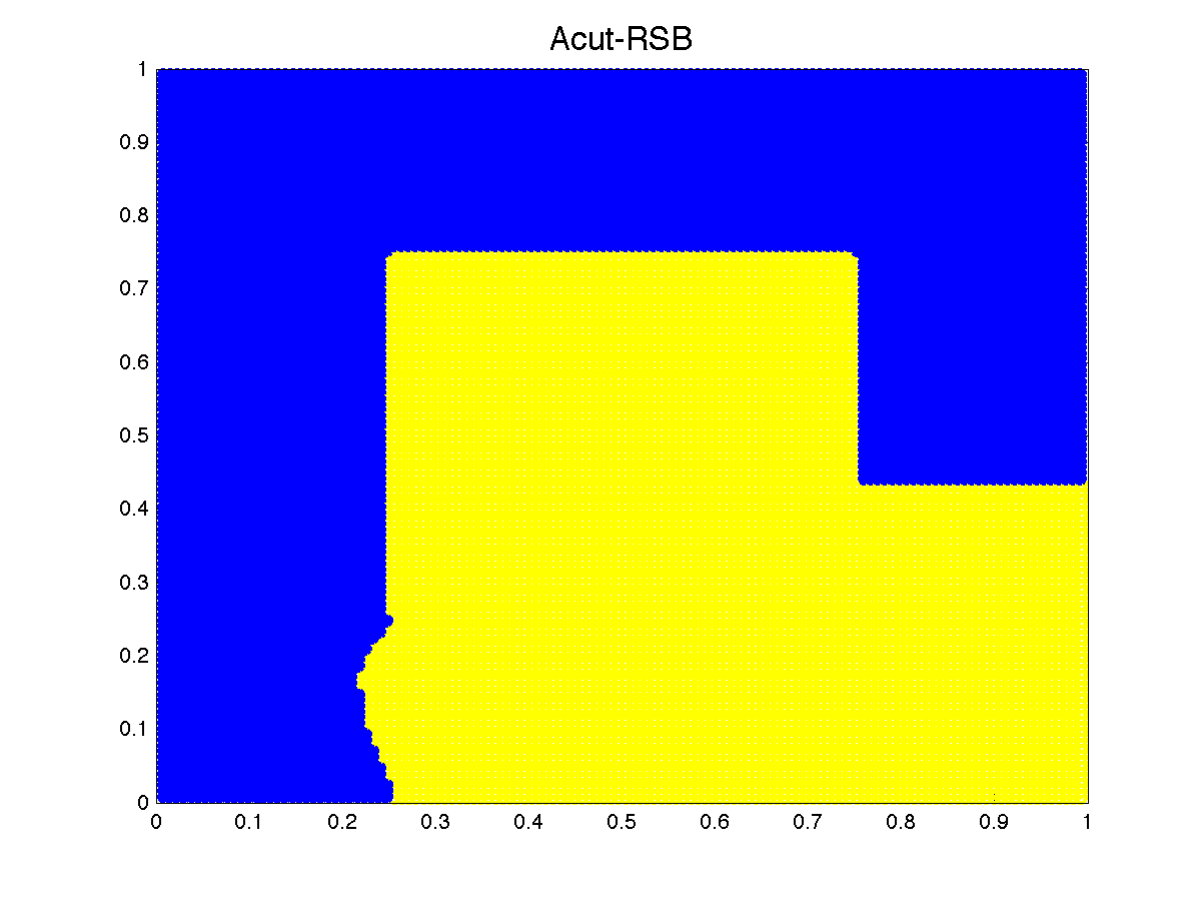}
 \includegraphics[width = 6.0cm]{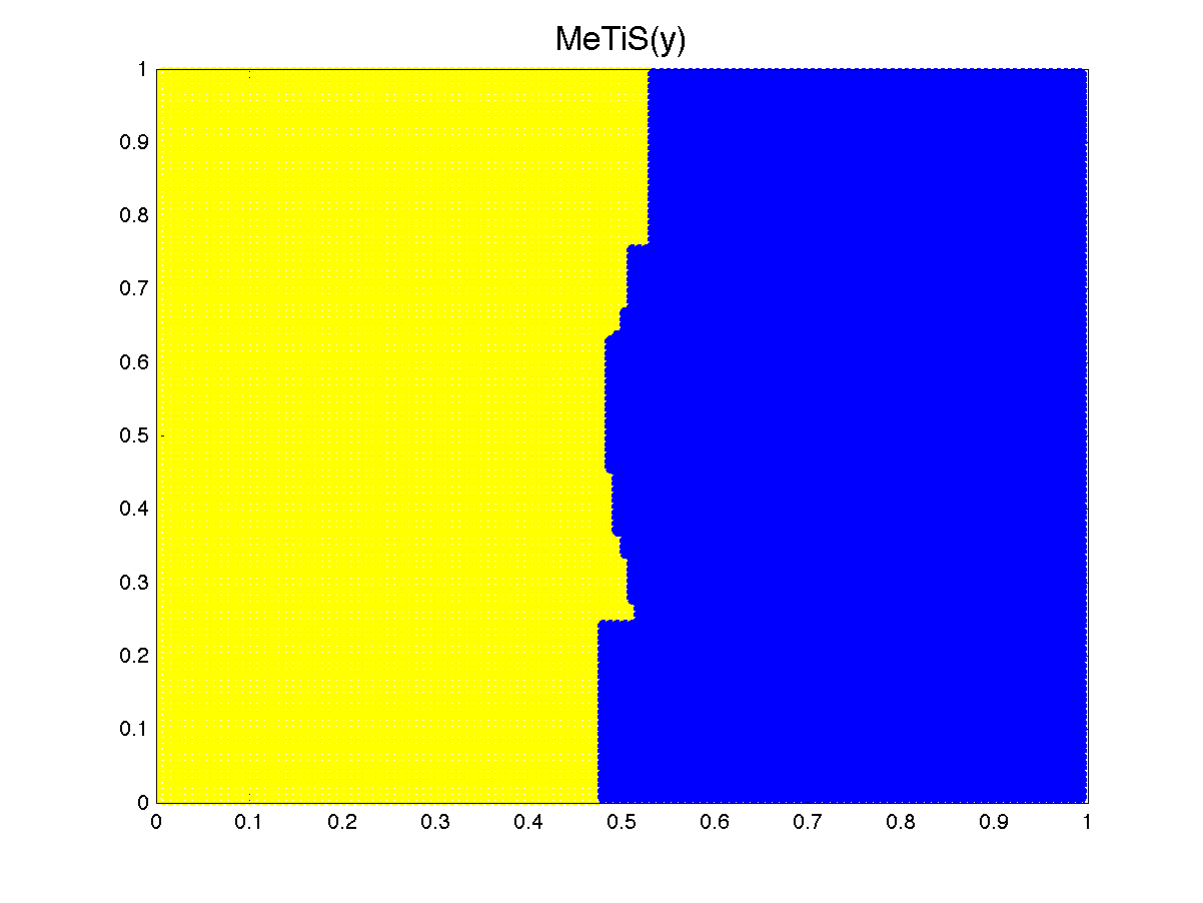} \\
 \includegraphics[width = 6.0cm]{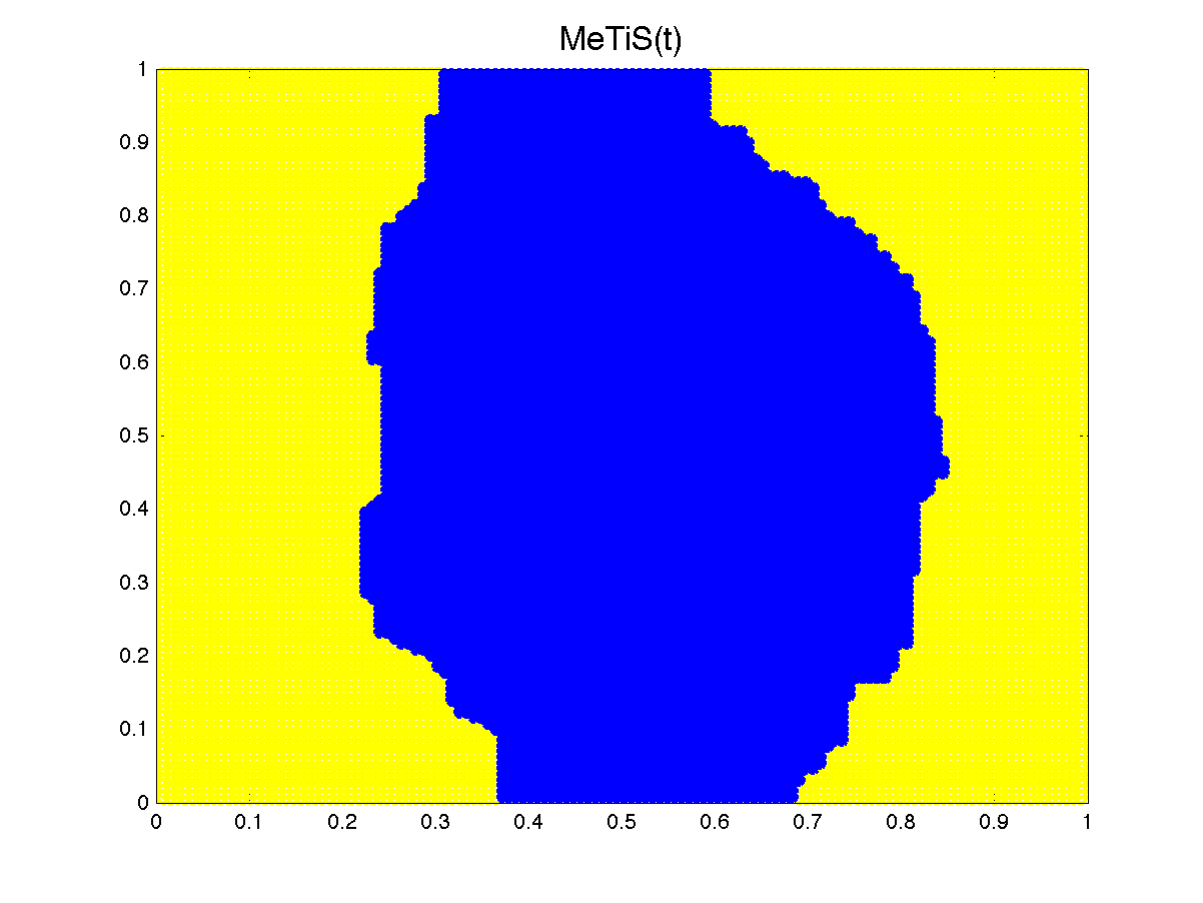}
 \includegraphics[width = 6.0cm]{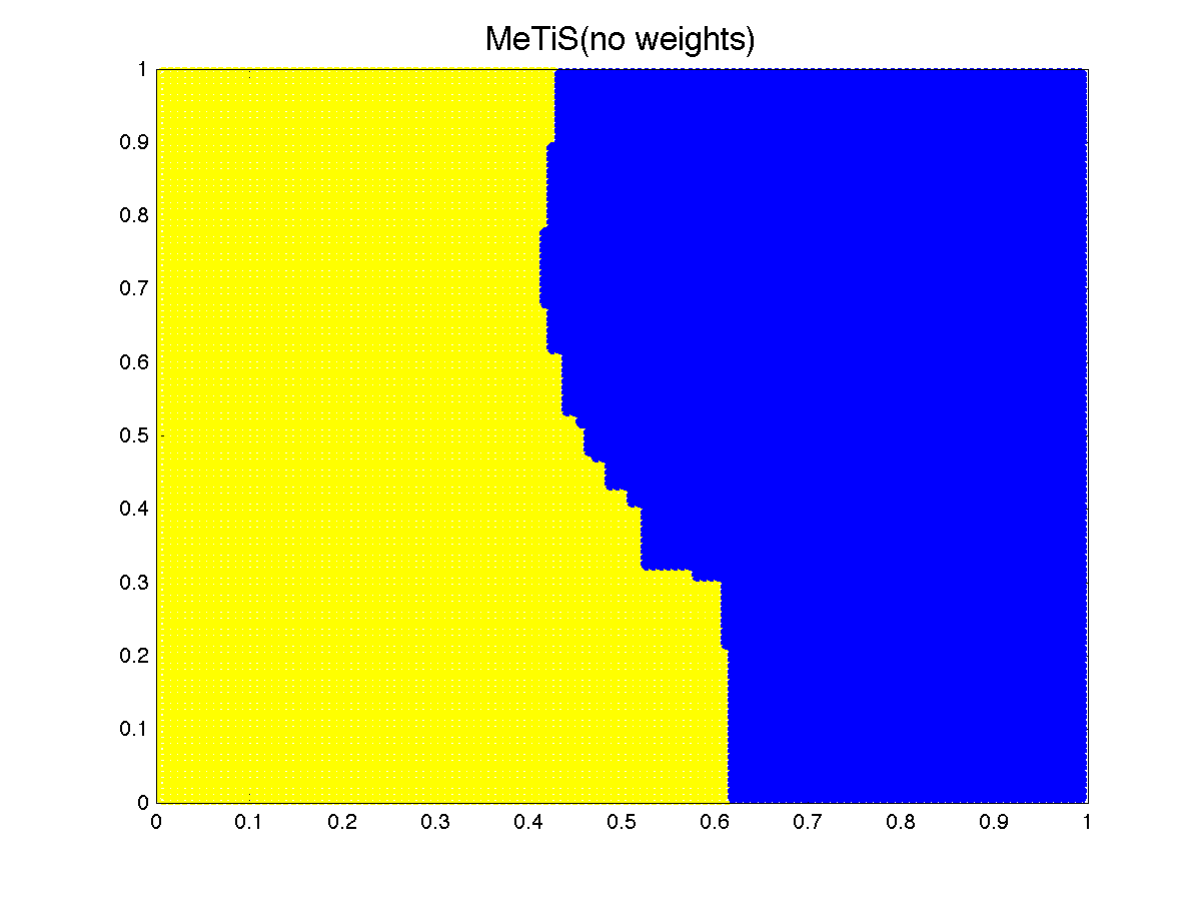}
 \end{center}
 \caption{Bipartitions for problem~(\ref{eqn:pde}) with coefficients in~(\ref{eqn:j_xy}).}
 \label{fig:bipartn1}
\end{figure}

Similar to Acut-RSB, we observe that the bipartition produced by MeTiS($t$) also
recognizes the region of the jump and isolates it in a separate subdomain.
The visual difference, however, is in that the boundaries of this subdomain are not as smooth
as for the Acut and do not tightly follow the boundaries of the jump region.

In Figure~\ref{fig:ex1_4D}, we present the results obtained after an additional step
of the recursive bipartitioning (with the  same $\gamma$, $\delta$, and the LOBPCG
convergence tolerance as above).
In this case, the partitioning procedures deliver 4 subdomains.
Note that, regardless of the weighting scheme, all the three MeTiS runs perform long cuts within
the jump region. In contrast, Acut-RSB essentially ``crops'' the central square, with the top right
corner assigned to a different subdomain merely to ensure the strict load balance.

\begin{figure}[h]
 \begin{center}
 \includegraphics[width = 6.0cm]{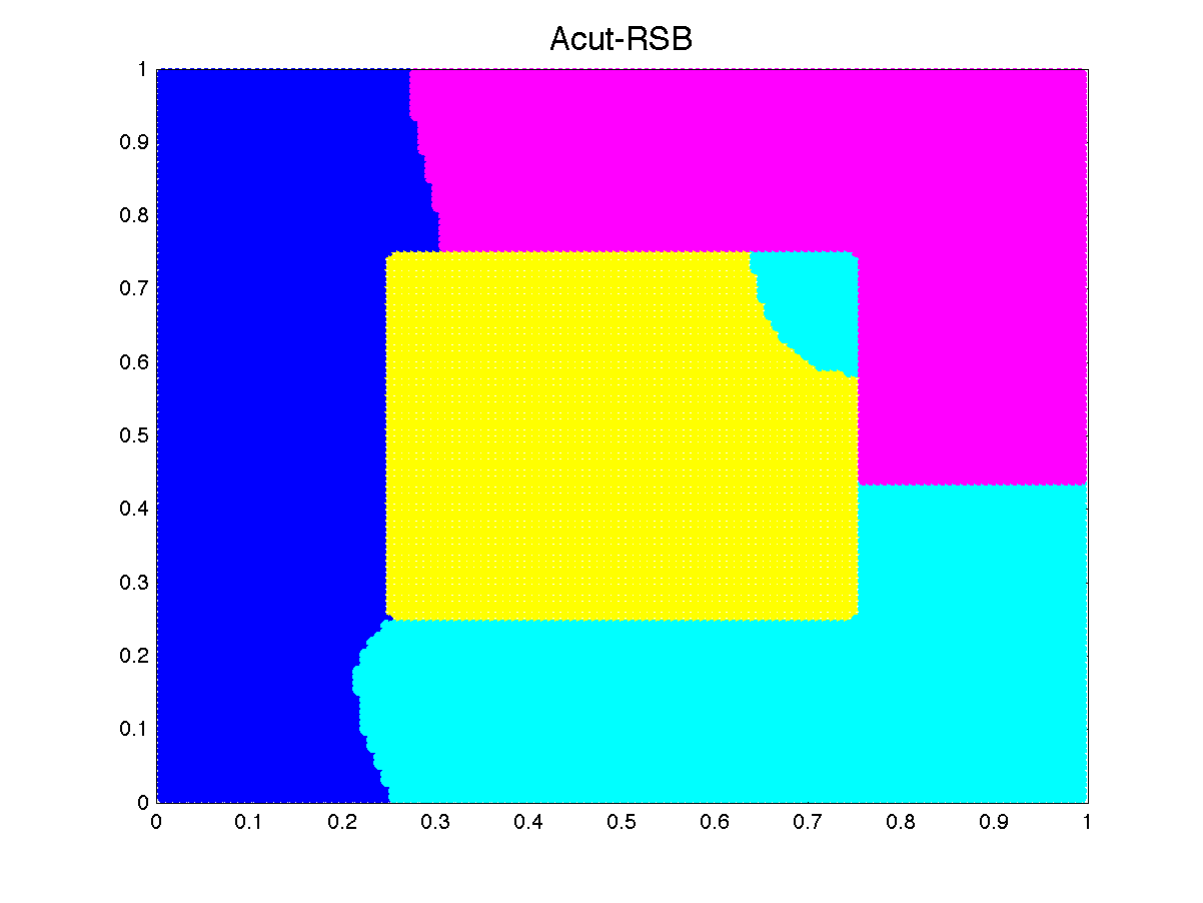}
 \includegraphics[width = 6.0cm]{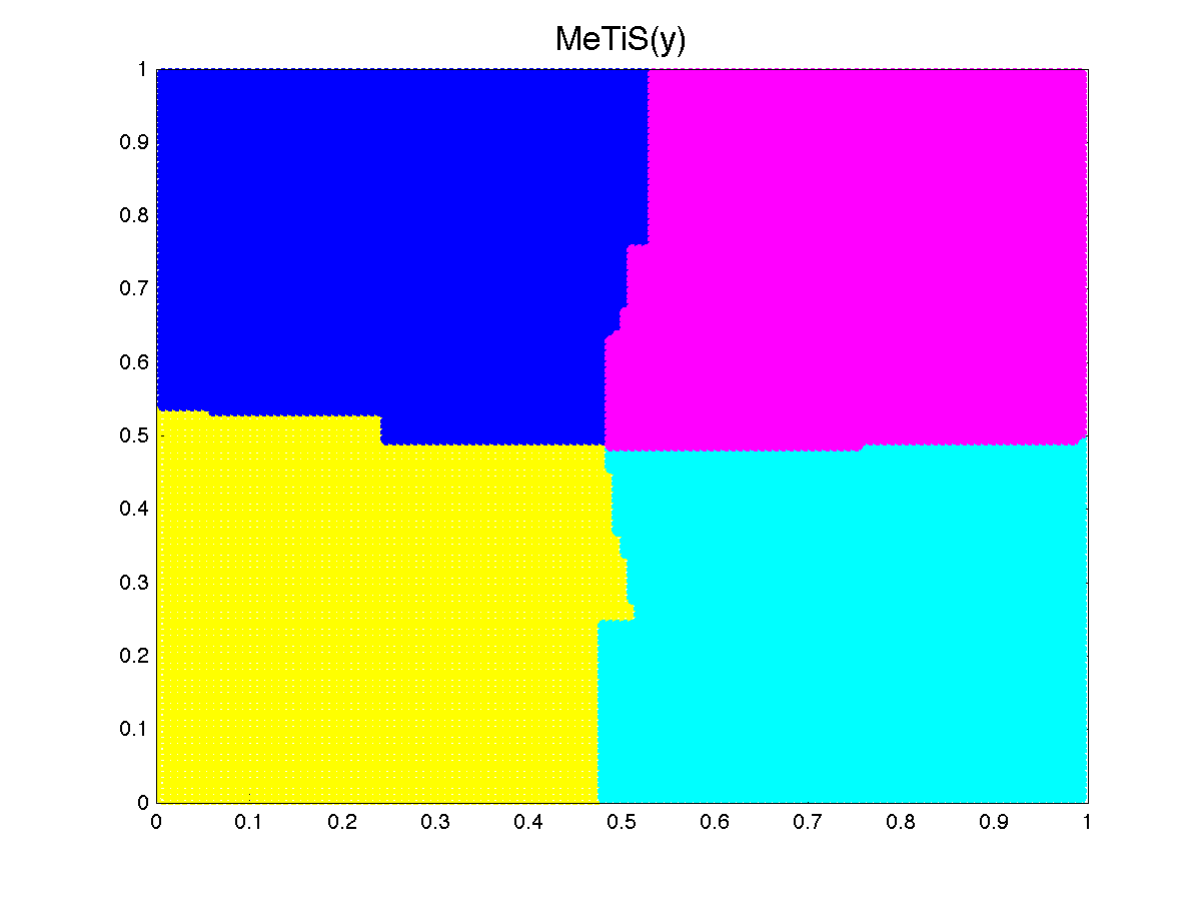} \\
 \includegraphics[width = 6.0cm]{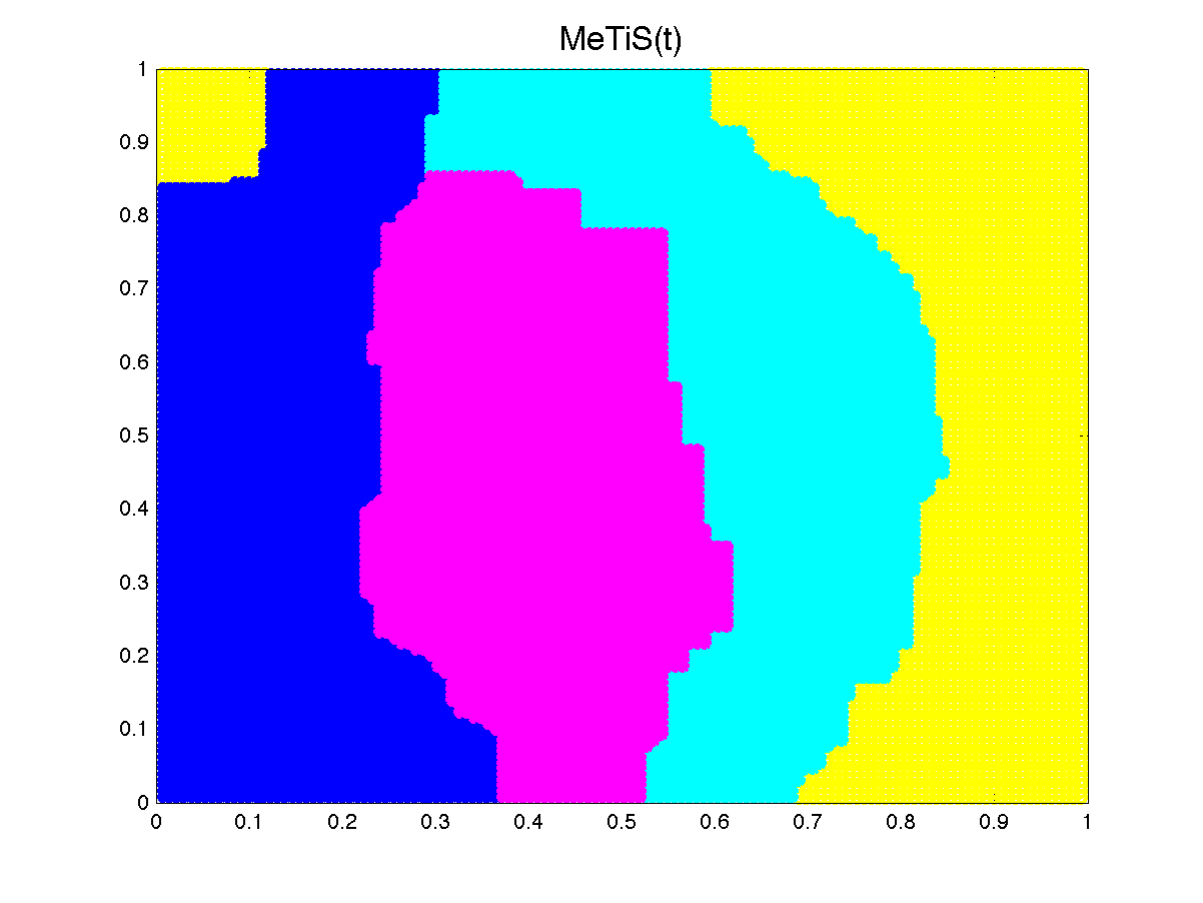}
 \includegraphics[width = 6.0cm]{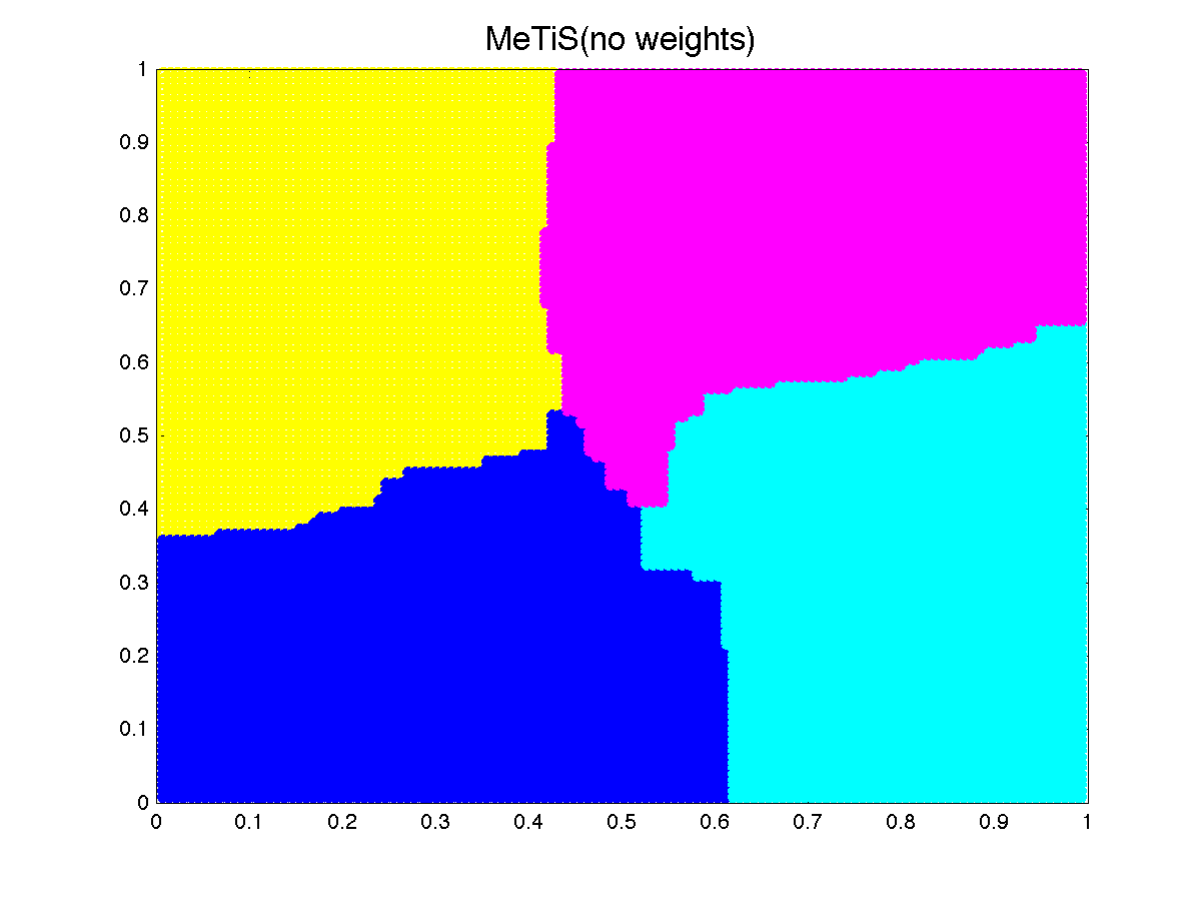}
 \end{center}
 \caption{Partitioning into $4$ subdomains for problem~(\ref{eqn:pde}) with coefficients in~(\ref{eqn:j_xy}).}
 \label{fig:ex1_4D}
\end{figure}

Since one of the two subdomains resulting from the Acut-RSB bipartioning does not contain a jump
region (see Figure~\ref{fig:bipartn1}, top left), we use a standard partitioning approach at the
recursive step on this subdomain. In particular, we use the spectral bisection based on the
Fiedler vector~\cite{Pothen.Simon.Liou:90}.

Throughout all of our experiments, we set the number of targeted subdomains 
to be small, i.e., at most 16 as in the 3D elasticity example considered below.     
Due to simplicity of the geometries of the tested problems,   
convergence of the Acut based solver deteriorates if the number of subdomains is increased. 
In this case, Acut-RSB begins to partition \textit{inside} the regions with similar coefficient
magnitudes, i.e., it is forced by the load balance constraint
to discard ``heavy'' edges. 
At some point, this hinders the convergence.
Generally, we expect that the number of subdomains, and hence of processors, will 
depend on the problem geometry and coefficients.

% Test 1 (center 130 by 130 pts grid)
% test1_MLa input file
% 156 cut size by metis (no weights)
% gamma = 1e+5 (for metis(y) weights)
% delta = 1 (for metis entry weights)
% jump 10^5
% 364 lobpcg steps
% tol lobpcg 1e-7
% PartGraphRecursive for all metis runs to ensure ``perfect'' load balance
% workspace in test1_reg.mat

% Test 2 (checkerboard 130 by 130 pts grid)
% test2_MLa input file
% 156 cut size by metis (no weights)
%% gamma 1e+5
% delta = 1 (for metis entry weights)
% jump 10^5
% 461 lobpcg steps
% tol lobpcg 1e-7
% workspace in test2_reg.mat

In Figure~\ref{fig:bipartn1_cv}, we show the effects of partitioning on the
convergence of a preconditioned iterative scheme.
In particular, we plot the convergence curves of PCG--AS with the preconditioners
defined on different partitions.
In both cases, with 2 and 4 subdomains, we observe that PCG--AS based on the
Acut-RSB partitioning in Algorithm~\ref{alg:partn1} performs a significantly
smaller number of iterations. % to achieve the desired tolerance level.

\begin{figure}[h]
 \begin{center}
 \includegraphics[width = 6cm]{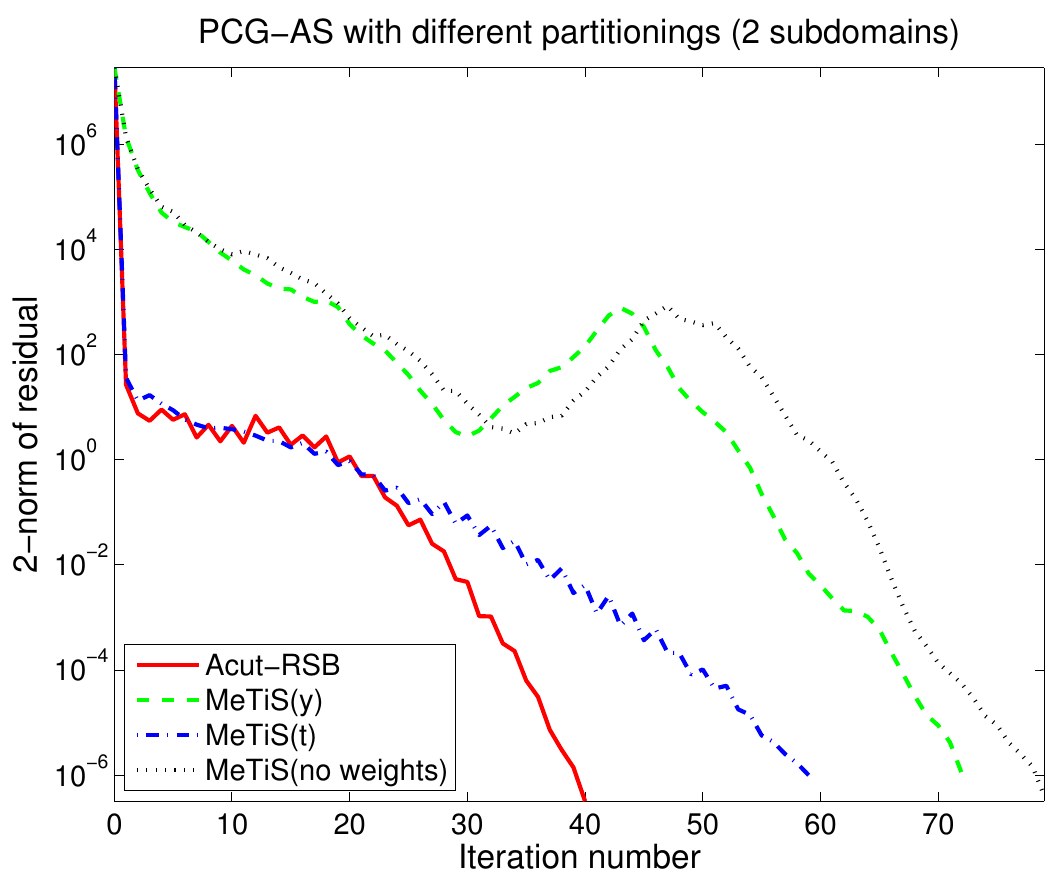}
 \includegraphics[width = 6cm]{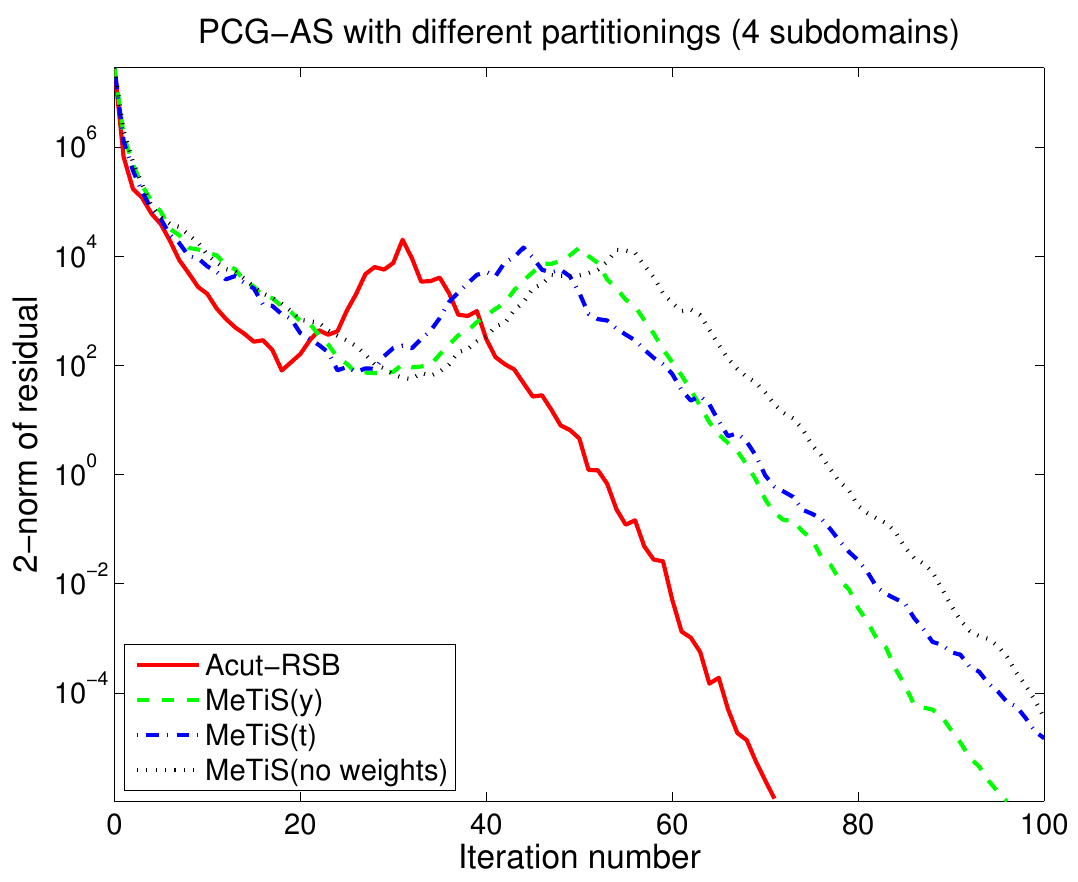}
 \end{center}
 \caption{Convergence of PCG--AS with different partitions for problem~(\ref{eqn:pde})
  with coefficients in~(\ref{eqn:j_xy}). The linear system of size $n = 16,384$ is partitioned into
  2 (left) and 4 (right) subdomains.}
 \label{fig:bipartn1_cv}
\end{figure}

In Table~\ref{tbl:reg}, we show the relative cut sizes and amounts of the coefficient
information discarded from the matrix after partitioning to construct the preconditioners.
In particular, we report the quantities
\[
\mbox{relcut} = \frac{\mbox{cut size}}{\mbox{nnz}} \times 100\%, \qquad \mbox{relcoef} = \frac{\sum_{(i,j)\in \mbox{cut}}|a_{ij}|}{\sum_{i,j} |a_{ij}|} \times 100\%,
\]
where ``cut size'' is the number of edges in the given cut and ``nnz'' denotes the number
of nonzeros in $A$, i.e., the total number of edges in the adjacency graph; ``cut'' is the edge cut of interest.
%
%Note that, intuitively, the smaller ``relcoef'' the higher is the quality of the corresponding preconditioner,
%i.e., it is reasonable to expect that the best nonoverlapping AS preconditioners are delivered by partitions
%which ensure the smallest sum of (absolute values of) the discarded coefficients.
%
%From Table~\ref{tbl:reg}, 

We see that
the cut sizes corresponding to Acut-RSB are relatively large
in terms of the edge count, compared to the unweighted MeTiS and MeTiS($y$). At the same time,
the amount of the coefficient information discarded for preconditioning is significantly smaller.
%which is consistent with the convergence results in Figure~\ref{fig:bipartn1_cv}.
The similar properties are exhibited by MeTiS($t$).
%, which is to be expected since the corresponding
%partitioning objective directly targets construction of the cut with the smallest coefficient sum.
%
Note that, in the case of bipartitioning, 
%An interesting observation is given by the ``relcoef'' values that correspond to the
%Acut-RSB and MeTiS($t$) bipartitionings. While 
``relcoef'' for Acut-RSB is only slightly smaller than that for MeTiS($t$), whereas 
the quality of the associated preconditioner is much higher;
see Figure~\ref{fig:bipartn1_cv} (left).

\begin{table}
\begin{center}
\caption{Relative cut sizes and amounts of the coefficient information discarded to construct preconditioners for
problem~(\ref{eqn:pde}) with coefficients in~(\ref{eqn:j_xy}).}
\label{tbl:reg}
\begin{tabular}{|c||c|c||c|c|}
\cline{2-5}
\multicolumn{1}{c|}{} & \multicolumn{2}{c||}{2 subdomains} & \multicolumn{2}{c|}{4 subdomains}\tabularnewline
\hline
Partitioning & relcut & relcoef & \multicolumn{1}{c|}{relcut} & relcoef\tabularnewline
\hline
\hline
Acut-RSB & 0.76 & $3\times10^{-5}$ & 1.25 & 0.42\tabularnewline
\hline
MeTiS($y$) & 0.44 & 0.85 & 0.86 & 1.62\tabularnewline
\hline
MeTiS($t$) & 1.06 & $4\times10^{-5}$ & 1.67 & 0.88\tabularnewline
\hline
MeTiS(no w.) & 0.48 & 1.07 & 1.01 & \multicolumn{1}{c|}{2.13}\tabularnewline
\hline
\end{tabular}
\end{center}
\end{table}

%Note that while, 
To the best of our knowledge, there are no available convergence bounds that are based 
on ``relcoef''. 
However, from the practical point of view,  
%intuitively, 
it is reasonable to expect 
that the smaller ``relcoef'' the higher is the quality of the corresponding preconditioner.
%i.e., 
%In other words, 
%it is reasonable to expect 
%that the highest quality nonoverlapping AS preconditioners are delivered by partitions
%which ensure the smallest 
%%sum of (absolute values of) the discarded coefficients.
%``relcoef'' value.
%
In particular, if a standard graph partitioner is at hand, then common
approaches for assigning the edge weights, such as in MeTiS($t$), are motivated exactly
by this heuristic. 
While often achieving the smallest values of
%the sum of (absolute values of) the discarded coefficients
%(i.e., minimizing ``relcoef'') 
``relcoef'' 
indeed leads to better results, 
we demonstrate 
%in this section 
that the dependence does not hold in general.
%
%, i.e., faster convergence can be given by partitions       
%
%scheme for assigning  
%edge weighting scheme is motivated exactly by this heuristic 
%common approaches for assigning the edge weights, such as in MeTiS($t$), are motivated exactly
%by this heuristic. 
%
%An interesting observation is given by the ``relcoef'' values that correspond to the
%Acut-RSB and MeTiS($t$) bipartitionings. While ``relcoef'' is only slightly smaller for Acut-RSB,
%the quality of the associated preconditioner is much higher than that given by MeTiS($t$);
%see Figure~\ref{fig:bipartn1_cv} (left).
%
%This leads to the possible conclusion that the 

More precisely, we keep track of ``relcoef'' to show that the preconditioning quality is affected not only by 
the sum of (absolute values of) the discarded coefficients, 
but also by the 
% \textit{shape} of the underlying partitions.
\textit{actual combination of edges} in the cut.
While the traditional MeTiS based approaches can succeed in the former, Acut-RSB is capable of 
choosing more suitable edge combinations for the resulting cuts.
This point is further elaborated on in the following examples.

\paragraph{\textbf{\emph{$2$D diffusion equation: the ``checkerboard'' jump location}}}

Let us now consider a discretization of equation~(\ref{eqn:pde})
with zero Dirichlet boundary conditions, where the jumps in coefficients are located in the ``checkerboard''
fashion,
\begin{equation}\label{eqn:j_xy_checker}
%\[
a(x,y) = b(x,y) = \left\{\begin{array}{cl}
	10^5, & \mbox{``black''} \\
	1,   & \mbox{``white''}\; ;
	   \end{array}\right.
%\]
\end{equation}
as shown in Figure~\ref{fig:grid1_checker}. As in the above example, we use a 5-point FD stencil
on a 128-by-128 uniform grid, which leads to an SPD linear system of size $n=16,384$.

\begin{figure}[h]
 \begin{center}
 \includegraphics[width = 6cm]{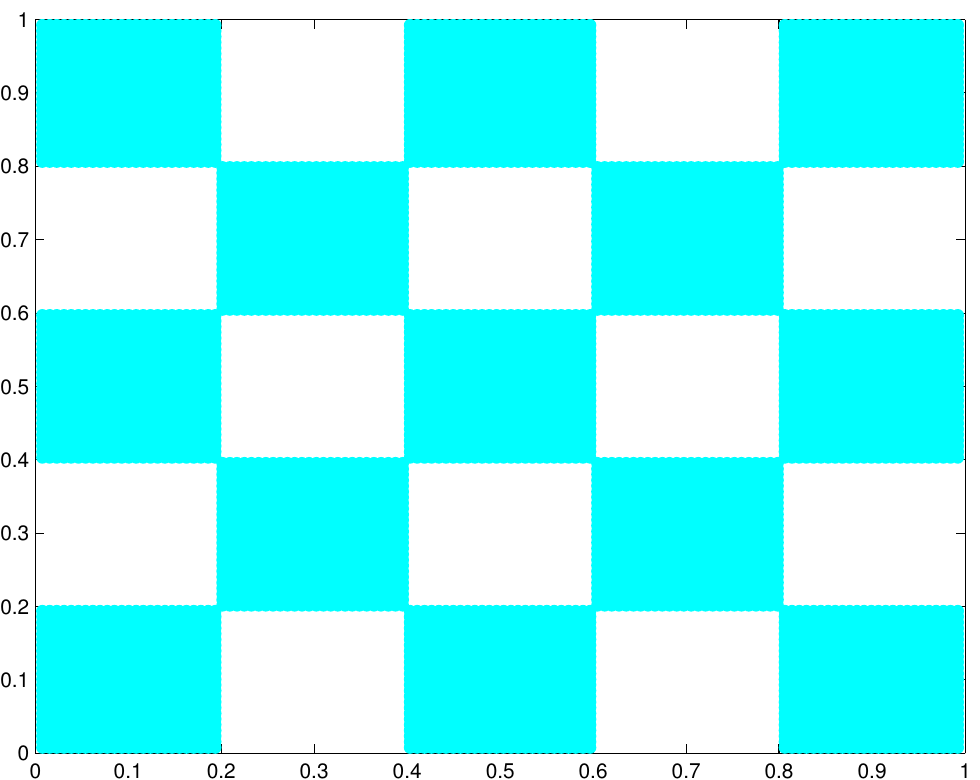}
 \end{center}
 \caption{The ``checkerboard'' location of the jump in coefficients~(\ref{eqn:j_xy_checker}) for problem~(\ref{eqn:pde}).
  The shaded (cyan for the color plot) subdomains correspond to the ``black'' jump regions.}
 \label{fig:grid1_checker}
\end{figure}

The results of bipartitioning produced by different methods are presented in Figure~\ref{fig:bipartn1_checker}.
Similar to the previous example, Acut-RSB perfectly detects the jump regions and avoids ``cutting'' inside.
Note that the resulting two subdomains are disconnected. The eigenvector of~(\ref{eqn:evp}), used to define the Acut-RSB
bipartition, is shown in Figure~\ref{fig:ev1} (right). The parameters $\gamma$ and $\delta$ are the same as in the previous example,
i.e., $10^5$ and $1$, respectively. The LOBPCG convergence tolerance is set to $10^{-7}$.

\begin{figure}[h]
 \begin{center}
 \includegraphics[width = 6.0cm]{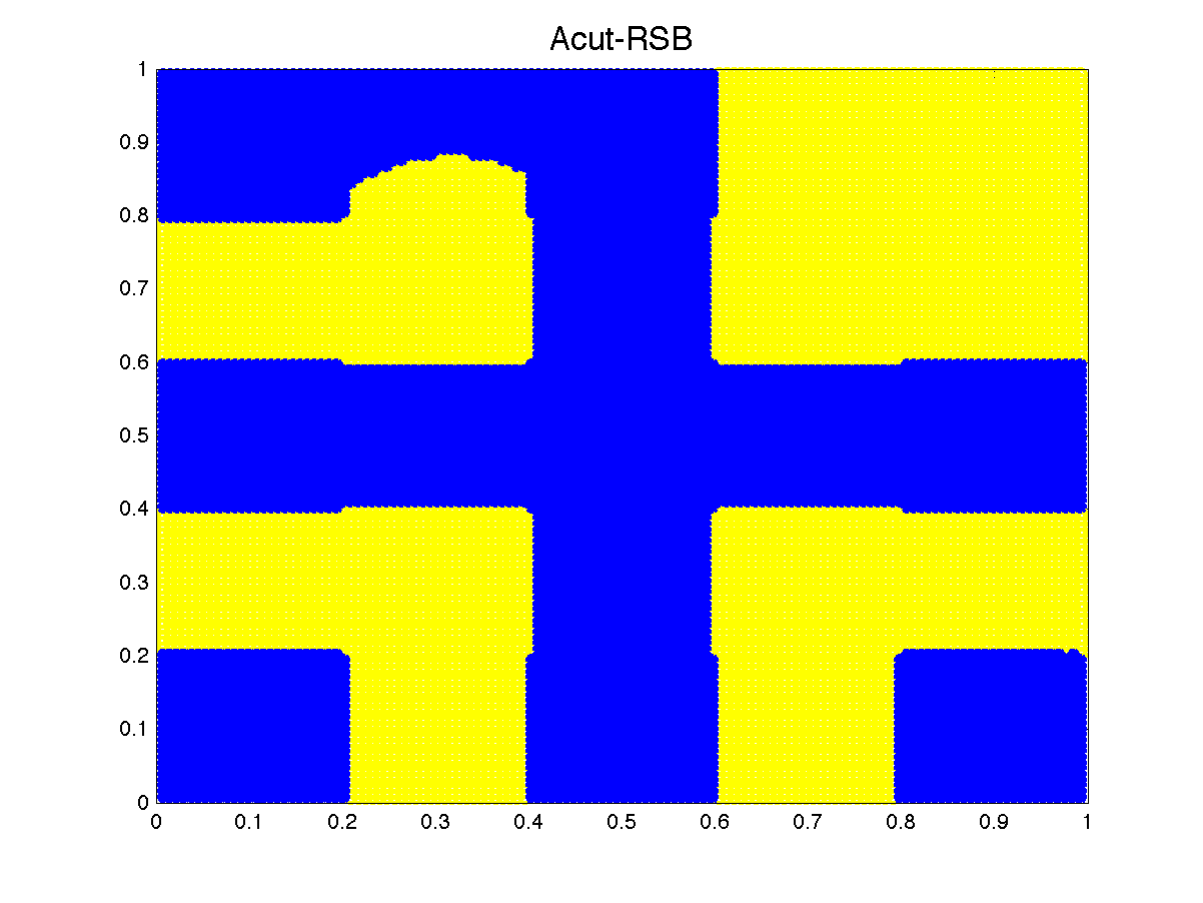}
 \includegraphics[width = 6.0cm]{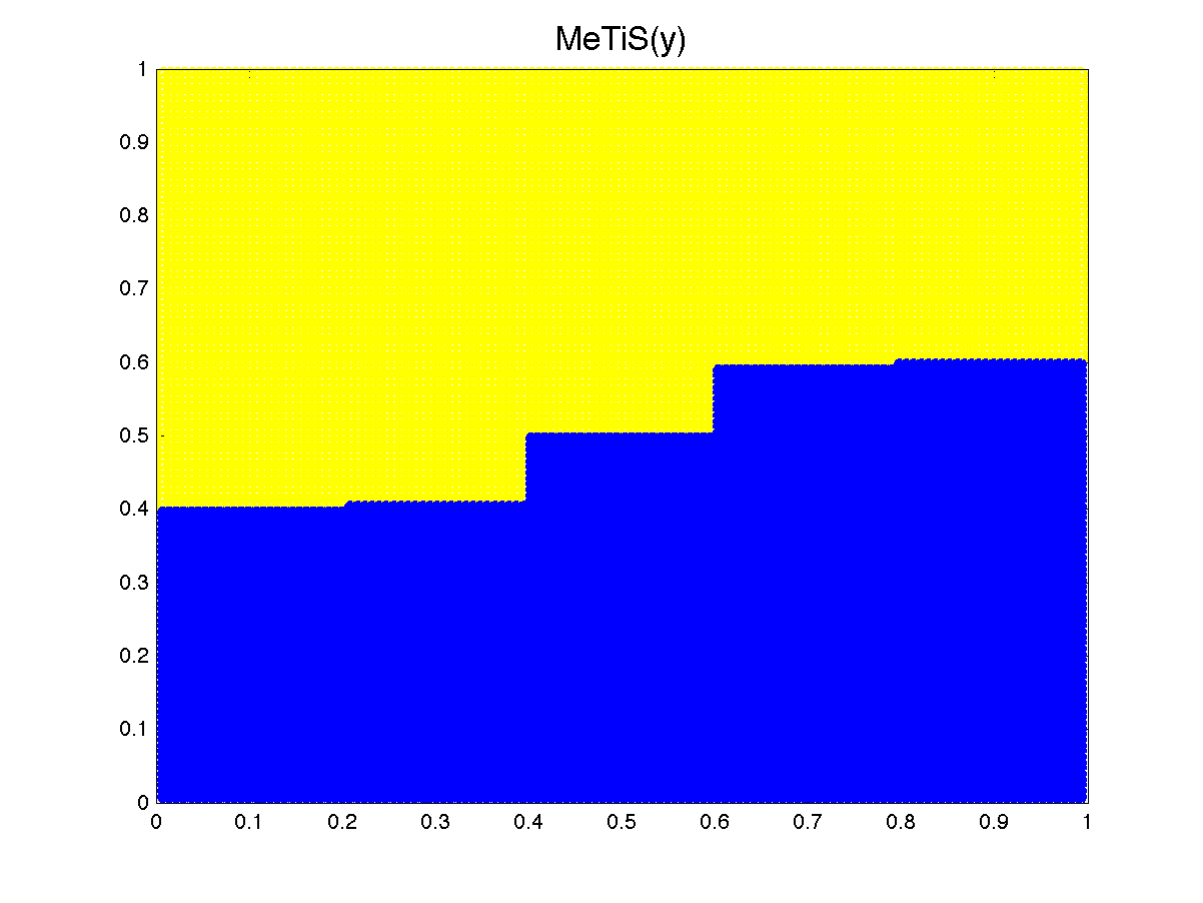} \\
 \includegraphics[width = 6.0cm]{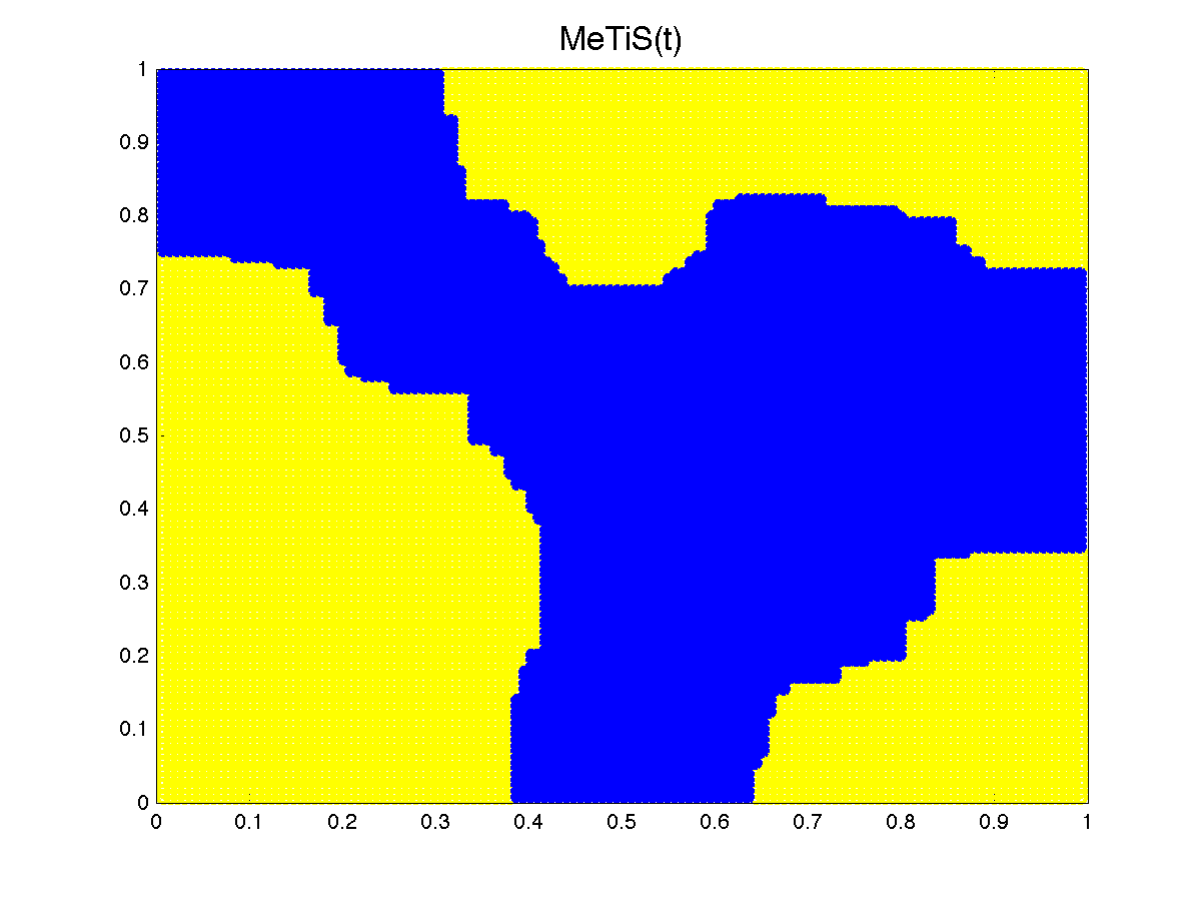}
 \includegraphics[width = 6.0cm]{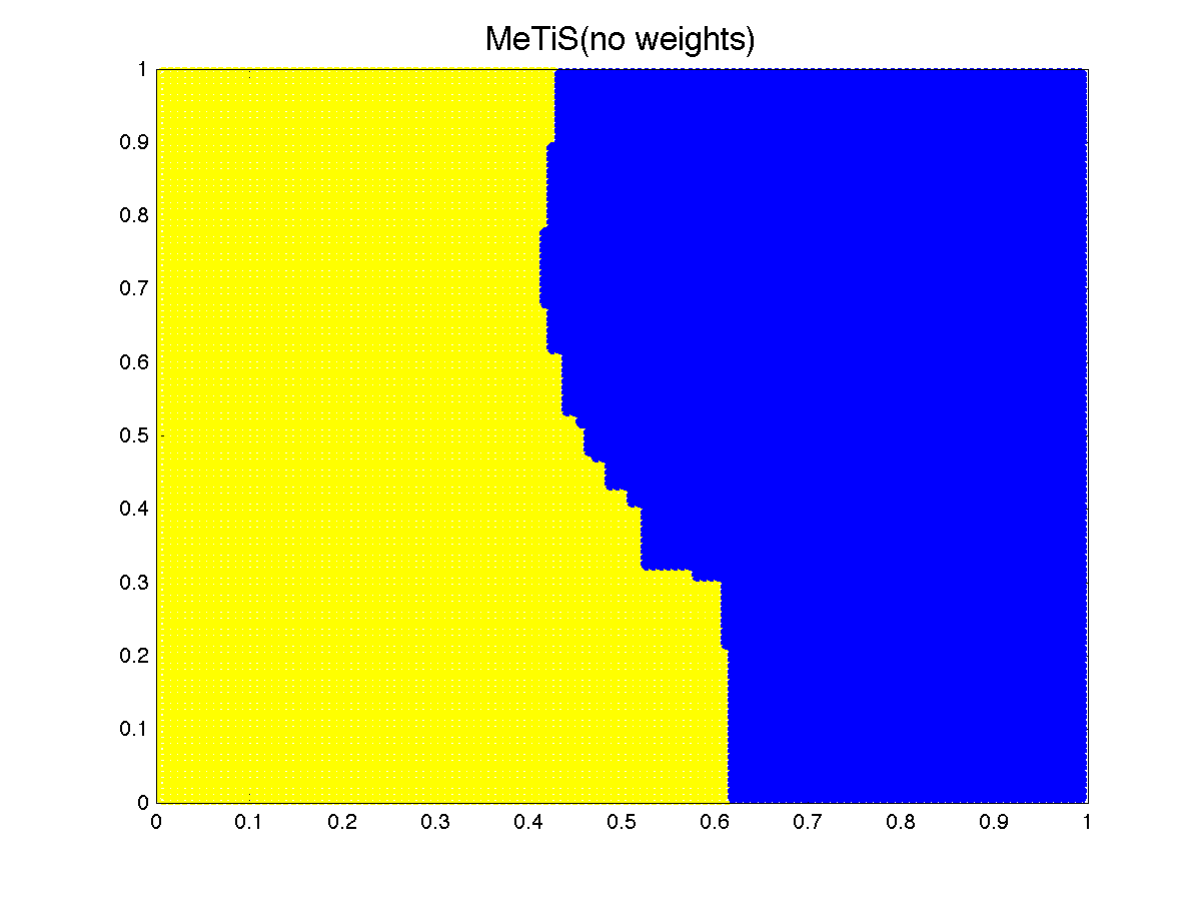}
 \end{center}
 \caption{Bipartitions for problem~(\ref{eqn:pde}) with coefficients in~(\ref{eqn:j_xy_checker}).}
 \label{fig:bipartn1_checker}
\end{figure}

Figure~\ref{fig:bipartn1_checker} demonstrates that, unlike the unweighted MeTiS,
the runs of MeTiS($y$) and MeTiS($t$) deliver partitions that adapt to the geometry of the
jumps, i.e., attempt to follow the boundaries of the ``black'' subregions.
The advantage, with respect to convergence of PCG--AS,
of using matrix coefficients
at the partitioning stage
%on the convergence of PCG--AS
%can be observed from
is further confirmed in Figure~{\ref{fig:ex1_cv_4D}}.

\begin{figure}[h]
 \begin{center}
 \includegraphics[width = 6cm]{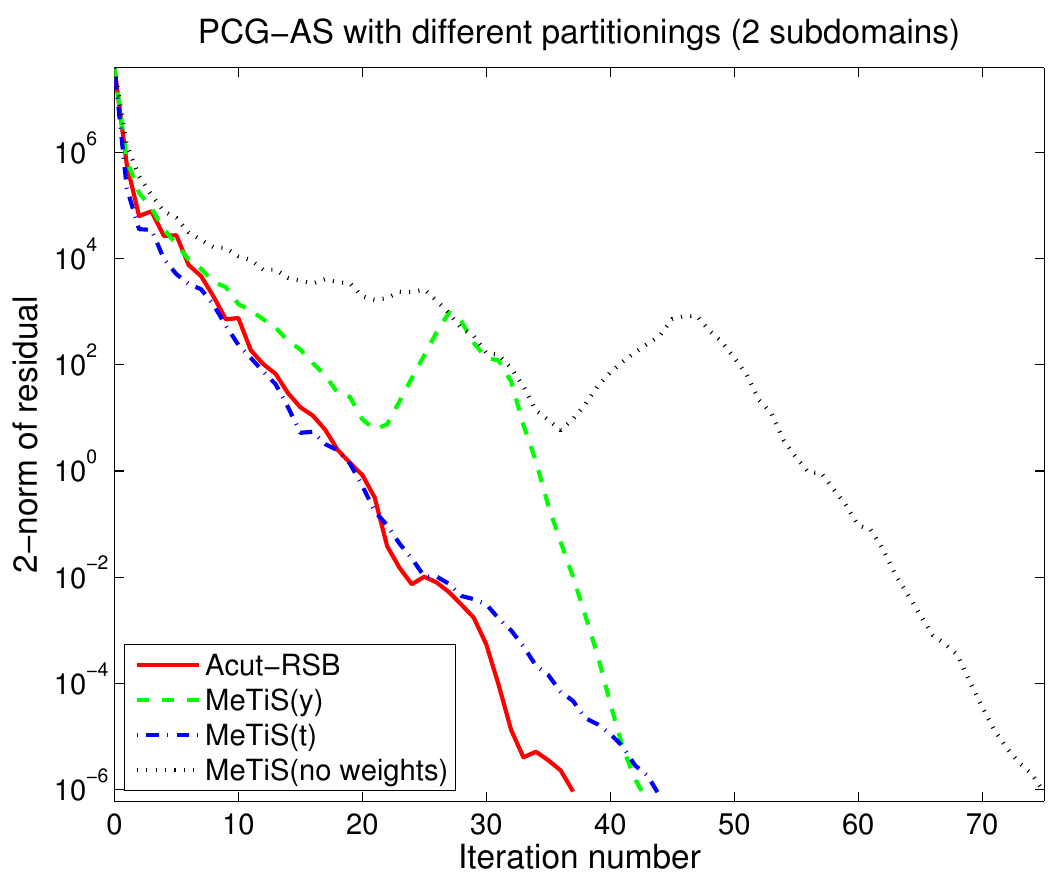}
 \end{center}
 \caption{Convergence of PCG--AS with different bipartitions for problem~(\ref{eqn:pde})
  with coefficients in~(\ref{eqn:j_xy_checker}). The linear system is of size $n = 16,384$.}
 \label{fig:ex1_cv_4D}
\end{figure}

Figure~{\ref{fig:ex1_cv_4D}} shows that, in terms of iteration count, PCG--AS with Acut-RSB
slightly outperforms the analogues based on MeTiS($y$) and MeTiS($t$).
Interestingly, however, the coefficient sum of (absolute values of) the matrix coefficients
discarded to construct the preconditioner, reported in Table~\ref{tbl:bipartn1_checker},
is \textit{not} the smallest for Acut-RSB.
More precisely, the ``relcoef'' value corresponding to Acut-RSB is three times larger than that of
MeTiS($t$).
This observation clearly allows one to conclude that the preconditioning quality is affected not only by the
magnitudes of the discarded matrix entries,
but also by the 
specific combination of edges selected to the cut that defines 
%shape of 
the underlying partitions.

\begin{table}
\begin{center}
\caption{Relative cut sizes and amounts of the coefficient information discarded to construct preconditioners for
bipartitioning problem~(\ref{eqn:pde}) with coefficients in~(\ref{eqn:j_xy_checker}).}
\label{tbl:bipartn1_checker}
\begin{tabular}{|l||c|c|}
%\cline{3-6}
%\multicolumn{1}{c}{} &  & \multicolumn{4}{c|}{PCG iterations (nonoverlapping AS)}\tabularnewline
\hline
Partitioning & relcut & relcoef  \tabularnewline
\hline
\hline
      Acut-RSB         &  1.78 & 0.06   \tabularnewline
\hline
      MeTiS($y$)   &  0.47 & 0.17   \tabularnewline
\hline
      MeTiS($t$)  &  1.23 & 0.02  \tabularnewline
\hline
      MeTiS(no w.)&  0.48 & 0.47  \tabularnewline
\hline
\end{tabular}
\end{center}
\end{table}

\begin{figure}[h]
 \begin{center}
 \includegraphics[width = 6cm]{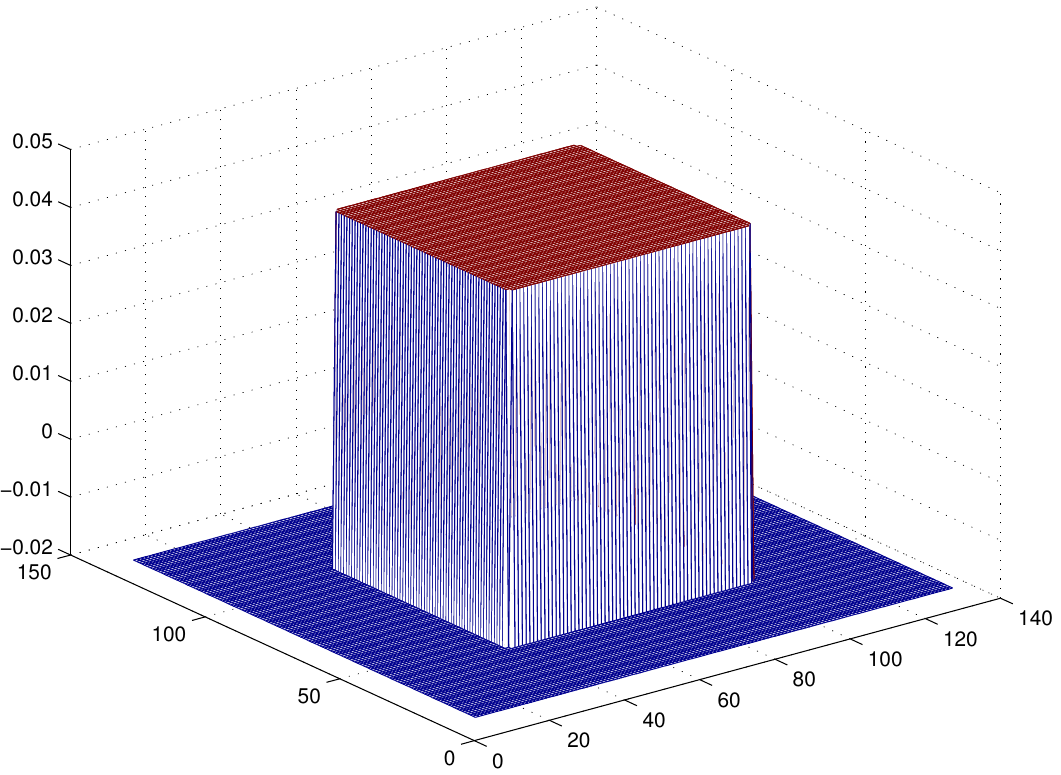}
 \includegraphics[width = 6cm]{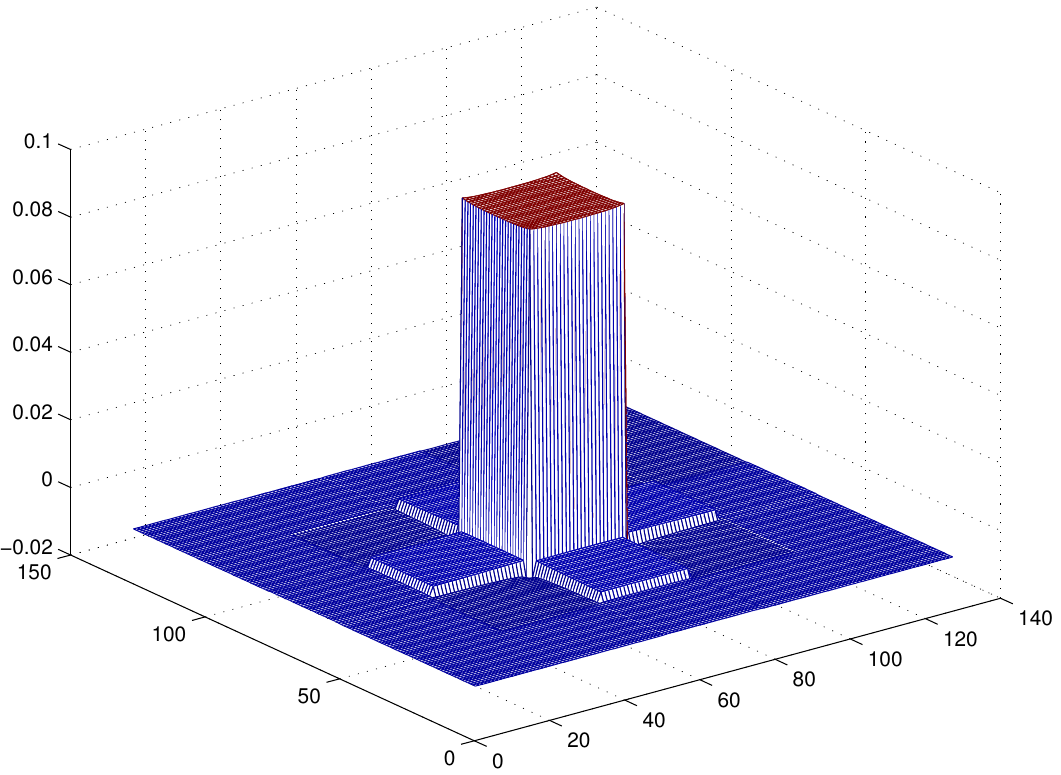}
 \end{center}
 \caption{``Mesh'' plot of the eigenvector $v$ of~(\ref{eqn:evp}) %corresponding to the smallest eigenvalue of~(\ref{eqn:evp})
  used for bipartitioning problem~(\ref{eqn:pde}) with coefficients in~(\ref{eqn:j_xy}) (left) and~(\ref{eqn:j_xy_checker}) (right).
  Both eigenvectors capture the discontinuities in coefficients of the corresponding problems.}
 \label{fig:ev1}
\end{figure}

\paragraph{\textbf{\emph{$3$D linear elasticity}}}
Our next experiment is based on the example constructed by Mandel et al. in~\cite{Mandel.Sousedik.Sistek:12}
to test the performance of the adaptive Balancing Domain Decomposition by Constraints (BDDC)~\cite{Dohrmann:03}
in three dimensions. In this example, the authors consider the 3D linear elasticity problem (see, e.g.,~\cite{Hughes:00})
in a cube with material parameters $E = 10^6$ Pa and $\nu = 0.45$, penetrated by four bars with parameters
$E = 2.1 \times 10^{11}$ Pa and $\nu = 0.3$; see Figure~\ref{fig:cube}. Zero Dirichlet boundary conditions are assumed.

\begin{figure}[h]
 \begin{center}
 \includegraphics[width = 5.5cm]{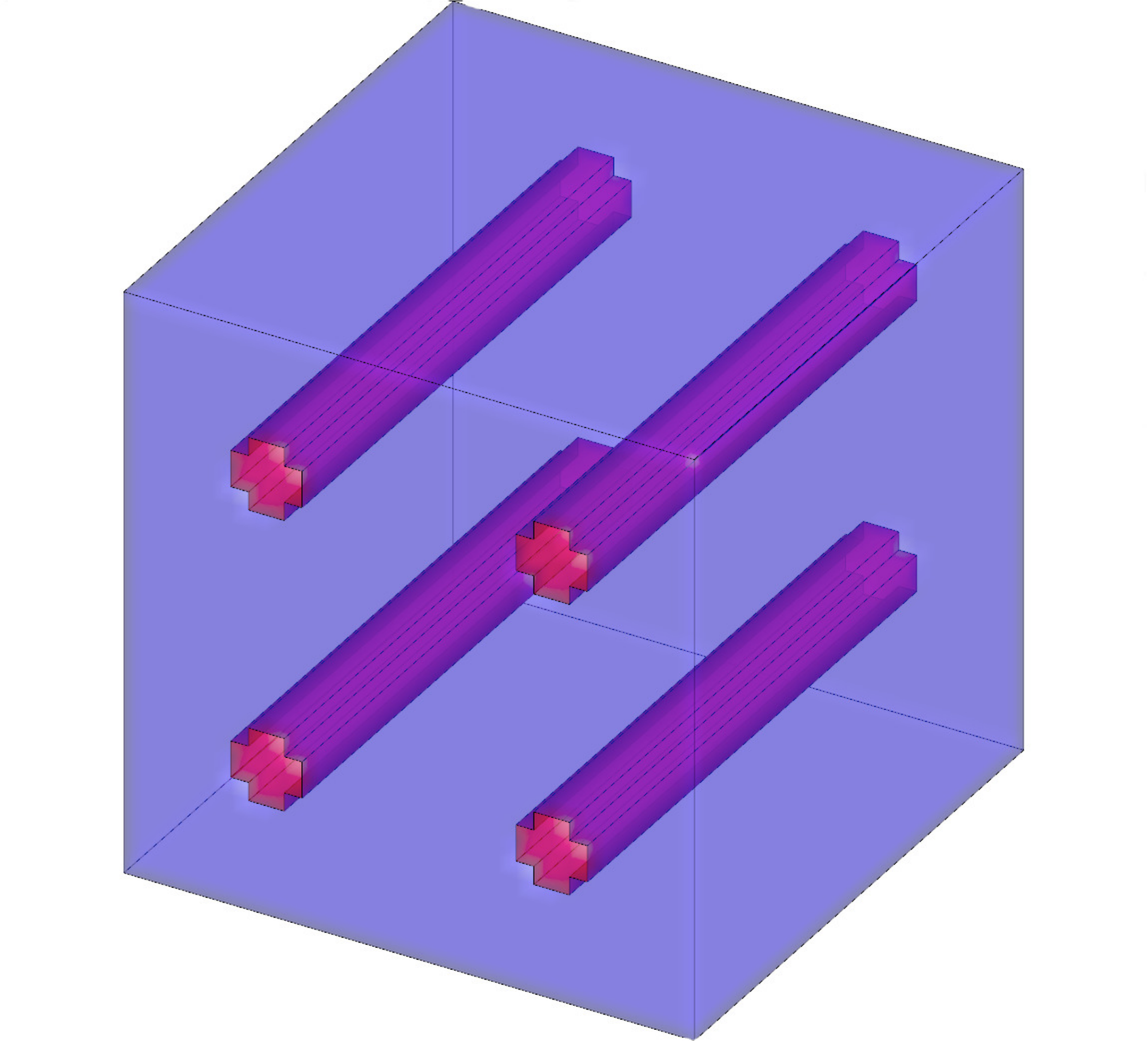}
 \end{center}
 \caption{3D linear elasticity: cube with jumps in coefficients. Example by Mandel et al.~\cite{Mandel.Sousedik.Sistek:12}.}
 \label{fig:cube}
\end{figure}

The problem has been discretized using bilinear finite elements (FE), resulting in
$107,811$ degrees of freedom. In our tests, we apply different
schemes to partition the problem into 8 and 16 subdomains,
and assess the quality of the obtained preconditioners according to the number of PCG--AS
iterations.
The parameters $\gamma$
and $\delta$ for MeTiS($y$) and MeTiS($t$) have been set to $10^4$ and $10^{-4}$,
respectively. The LOBPCG convergence tolerance is $10^{-4}$. In the case of 16 subdomains,
on the bottom level of the Acut-RSB recursion (after producing 8 subdomains),
we decrease the tolerance to $10^{-7}$.
The right-hand side
$b$ is chosen as a random unit vector.

% Test 3 (cube)
%% gamma = 1e+4
% delta = 1e-4 (for metis entry weights)
% ? lobpcg steps  <----??
% tol lobpcg 1e-7 <--- ??
% workspace in cube8.mat and cube16.mat

%
%demonstrates a significant acceleration of the convergence of PCG--AS with respect
%to this partition as compared to PCG--AS with partitions given by the other schemes.

\begin{figure}[h]
 \begin{center}
 \includegraphics[width = 6cm]{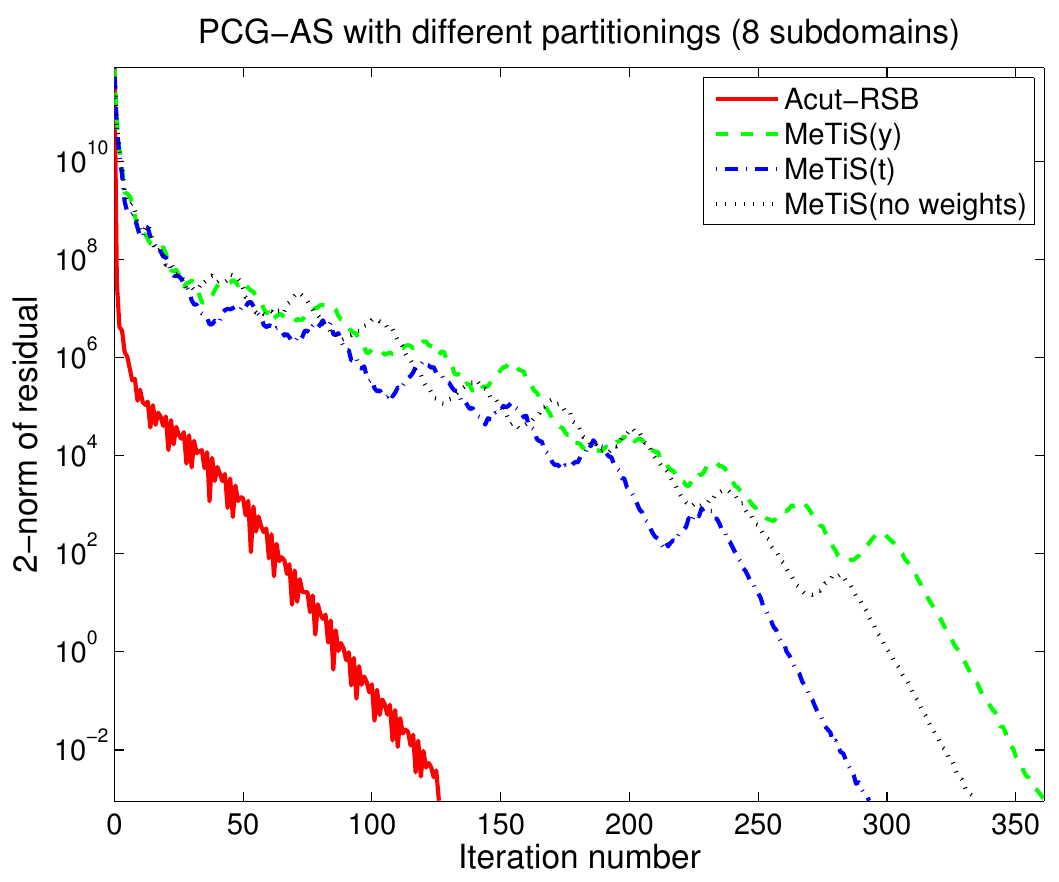}
 \includegraphics[width = 6cm]{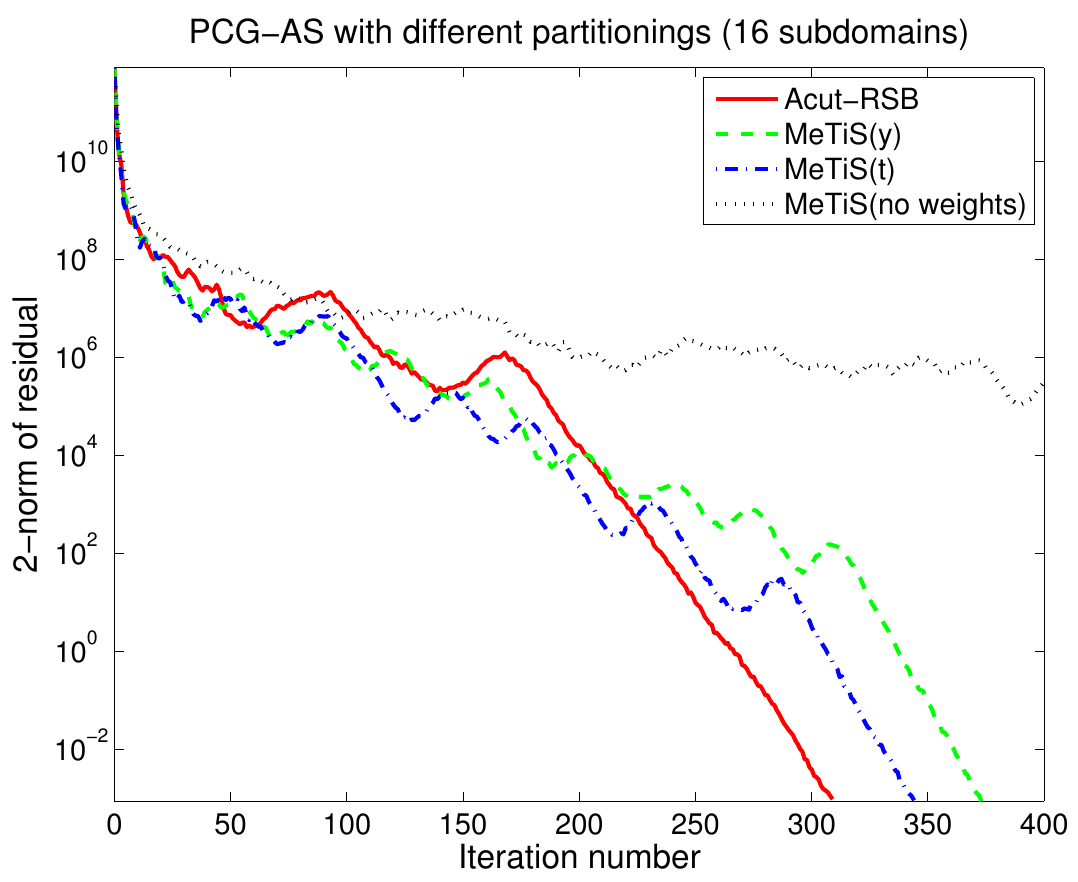}
 \end{center}
 \caption{Convergence of PCG--AS with different partitions (8 (left) and 16 (right) subdomains) for the 3D linear elasticity
in a cube with large jumps in coefficients; see Figure~\ref{fig:cube}. The size of the linear
system arising from the FE discretization is $n = 107,811$.}
 \label{fig:cube_cv}
\end{figure}

Figure~\ref{fig:cube_cv} shows that Acut-RSB leads to the best convergence of PCG--AS with both
8 and 16 subdomains. Note that, according to the results reported in Table~\ref{tbl:cube},
the cut size produced by Acut-RSB is about $2$--$2.5$ times larger than that delivered
by the MeTiS runs, implying extra communication overhead if implemented in parallel.
The gain from this ``loss of parallelism'', however, is the (approximately) $10\%$ to $45\%$ decrease in the iteration count
compared to the closest competitor, MeTiS($t$), and more than (approximately) $60\%$ decrease compared
to the unweighted MeTiS. Once again, from Table~\ref{tbl:cube}, we note that for 16 subdomains
``relcoef'' for Acut-RSB is larger than that for MeTiS($t$), although the convergence of the former is better.

%Acut-RSB & 17.82 & 0.57 & 23.03 & 1.11\tabularnewline
\begin{table}
\begin{center}
\caption{Relative cut sizes and amounts of the coefficient information discarded to construct preconditioners for
3D linear elasticity problem in a cube with jumps in coefficients; see Figure~\ref{fig:cube}.}
\label{tbl:cube}
\begin{tabular}{|c||c|c||c|c|}
\cline{2-5}
\multicolumn{1}{c|}{} & \multicolumn{2}{c||}{8 subdomains} & \multicolumn{2}{c|}{16 subdomains}\tabularnewline
\hline
Partitioning & relcut & relcoef & \multicolumn{1}{c|}{relcut} & relcoef\tabularnewline
\hline
\hline
Acut-RSB & 19.98 & 0.01 & 25.39 & 1.55\tabularnewline
\hline
MeTiS($y$) & 8.03 & 2.51 & 11.14 & 2.51\tabularnewline
\hline
MeTiS($t$) & 7.30 & 1.50 & 11.99 & 1.51\tabularnewline
\hline
MeTiS(no w.) & 6.77 & 2.24 & 10.61 & \multicolumn{1}{c|}{10.28}\tabularnewline
\hline
\end{tabular}
\end{center}
\end{table}

\paragraph{\textbf{\emph{$2$D diffusion equation: unstructured grid}}}

In this concluding example, we consider a FE discretization of diffusion equation~(\ref{eqn:pde})
on an unstructured grid. We assume Dirichlet boundary conditions and $f(x,y)=1$.
The problem domain represents a unit square with four inscribed circles of the same radius;
see~Figure~\ref{fig:irreg_geom} (left).

\begin{figure}[h]
 \begin{center}
 \includegraphics[width = 6cm]{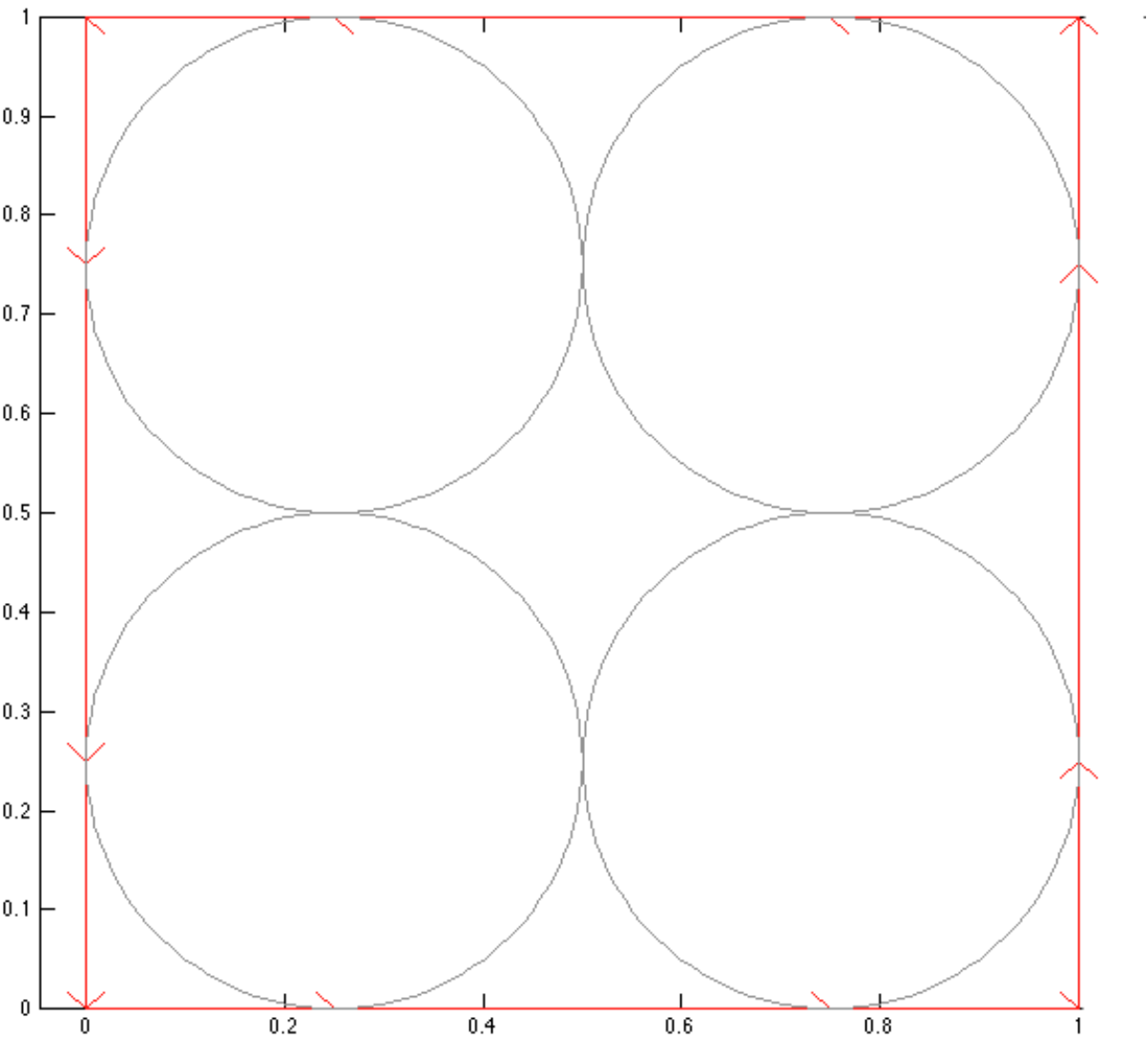}
 \includegraphics[width = 6cm]{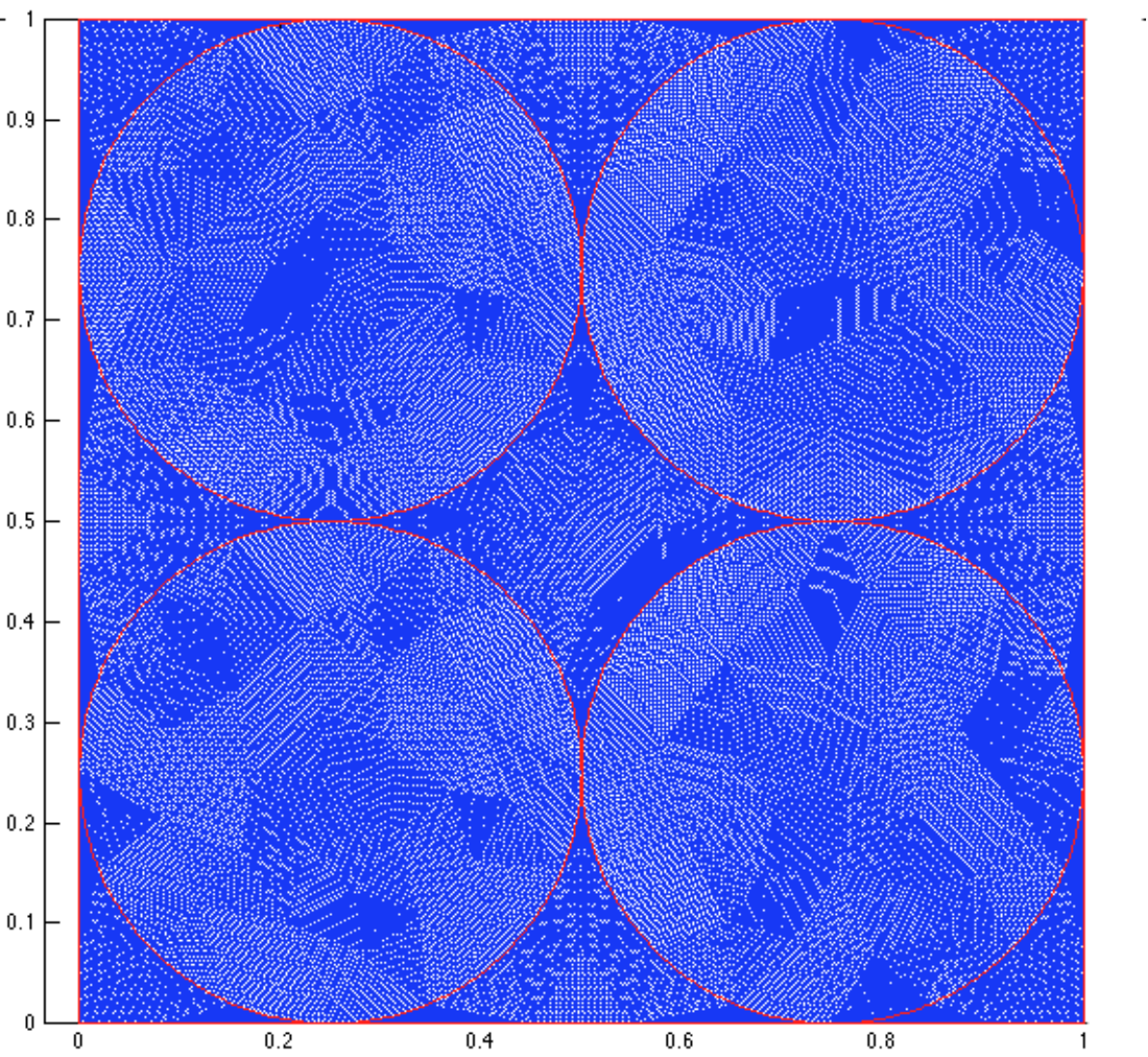}
 \end{center}
 \caption{Problem domain (left) and the corresponding FE mesh (right) for equation~(\ref{eqn:pde}).}
 \label{fig:irreg_geom}
\end{figure}

In order to discretize the equation on an unstructured grid, we use {\sc matlab PDE Toolbox}.
The toolbox allows one to define the problem geometry as well as the PDE
coefficients and to introduce an initial
triangulation,  which is further modified by a few refinement steps. The latter results in
an unstructured FE mesh with $47,713$ degrees of freedom, shown in Figure~\ref{fig:irreg_geom} (right).

The goal of the current example is twofold.
On the one hand, we show that Acut-RSB can be successfully
applied to problems on unstructured grids.
On the other hand, the example reveals potential
difficulties with the new partitioning strategy, which should be addressed in future research.

Let us first consider the following definition of the coefficients:
\begin{equation}\label{eqn:j_xy_irreg}
%\[
a(x,y) = b(x,y) = \left\{\begin{array}{cl}
	10^7, & \mbox{``outside circles and on circles' boundaries''} \\
	1,   & \mbox{``inside circles''}\; ;
	   \end{array}\right.
%\]
\end{equation}
i.e., $a(x,y)$ and $b(x,y)$ are the piecewise constants taking a large value outside of
the four circles and a small value inside.
We partition the problem into two and four subdomains using different partitioning schemes
and observe the convergence of the corresponding PCG--AS runs.
The parameters $\gamma$ and $\delta$ for MeTiS($y$) and MeTiS($t$) are set to
$10^4$ and $10^{-2}$, respectively. The LOBPCG convergence tolerance is $10^{-7}$
for the case of two, and $10^{-8}$ for the case of four, subdomains.
The partitioning results are given in Figures~\ref{fig:irreg_partn}~and~\ref{fig:irreg_partn4D}.
%The convergence results are given in Figure~\ref{fig:irreg_cv}.
%Table~\ref{tbl:irreg} presents the corresponding cut sizes and coefficient sums, ``relcut''
%and ``relcoef'', respectively.

\begin{figure}[h]
 \begin{center}
 \includegraphics[width = 6.4cm]{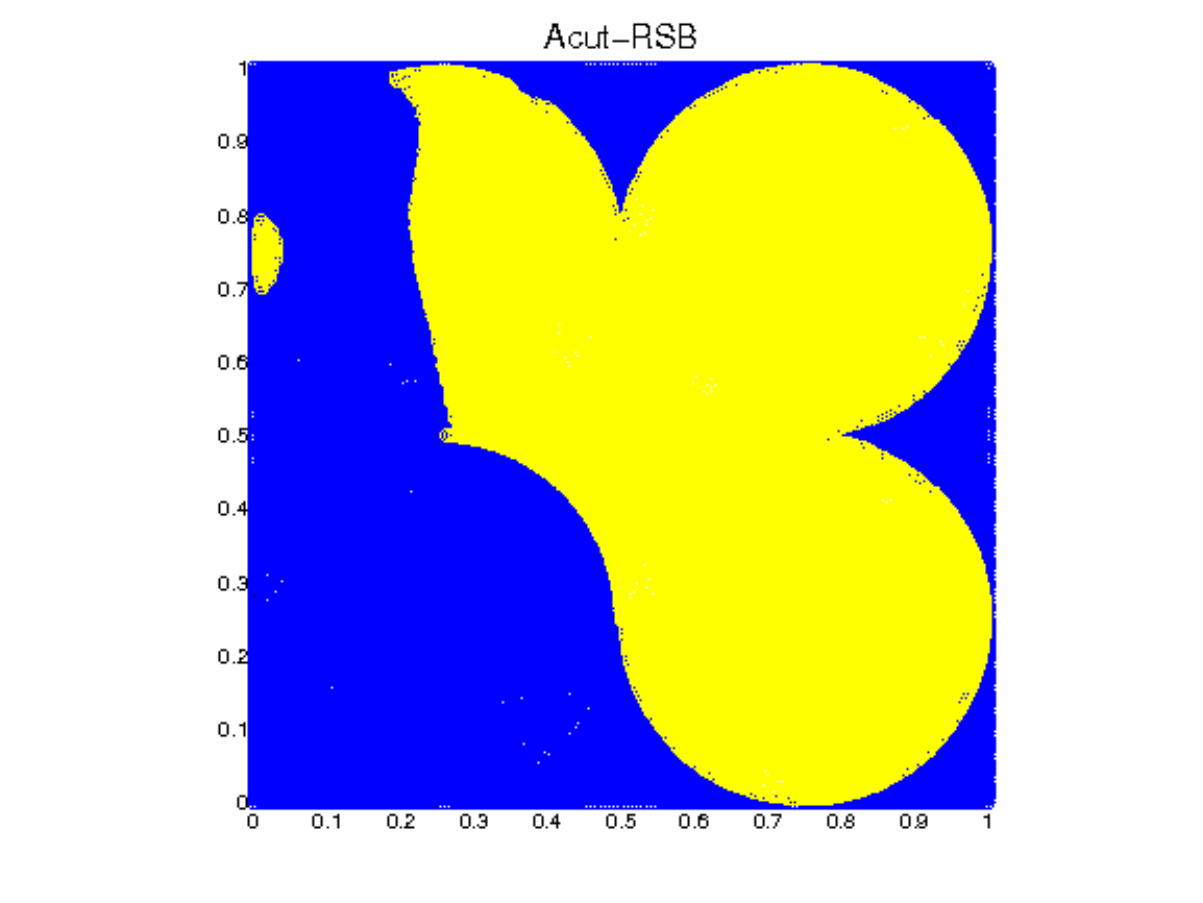}
 \includegraphics[width = 6.4cm]{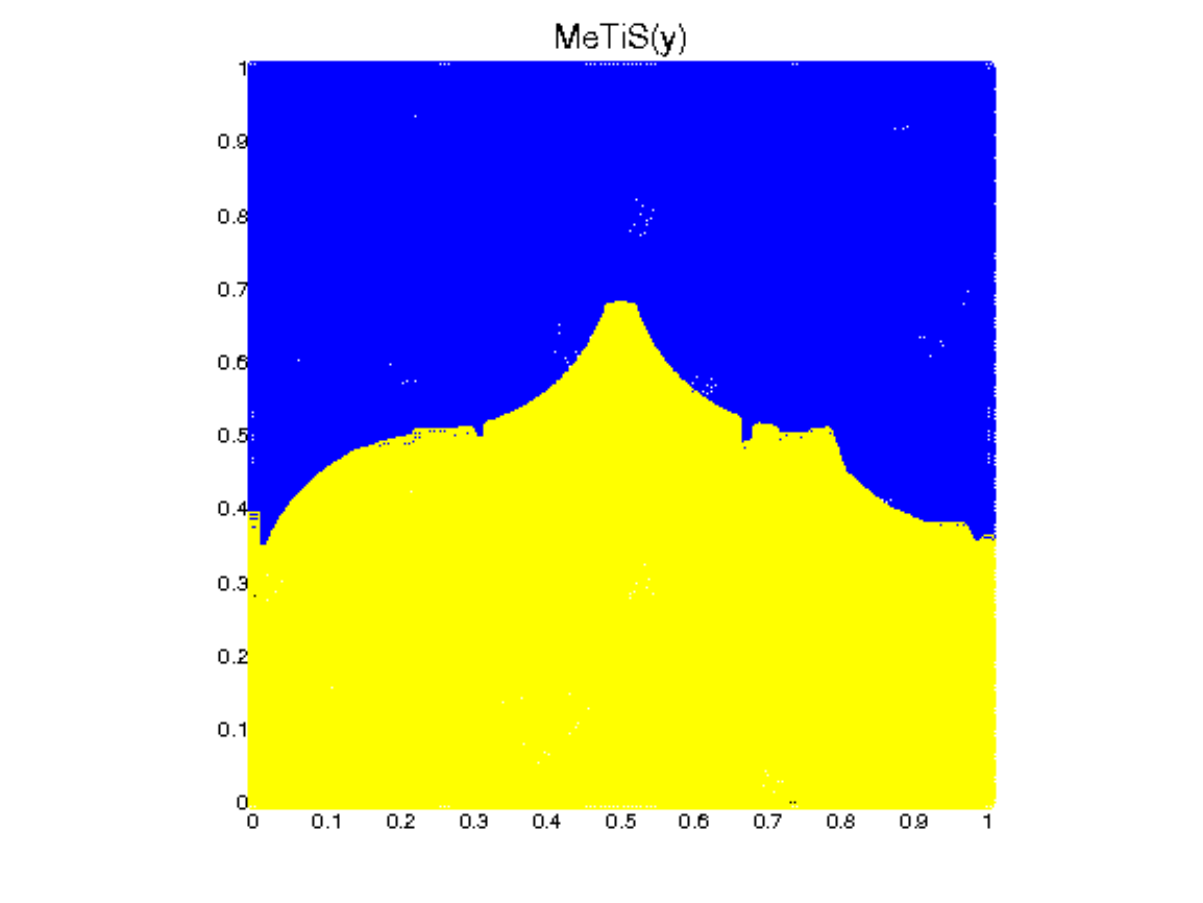} \\
 \includegraphics[width = 6.4cm]{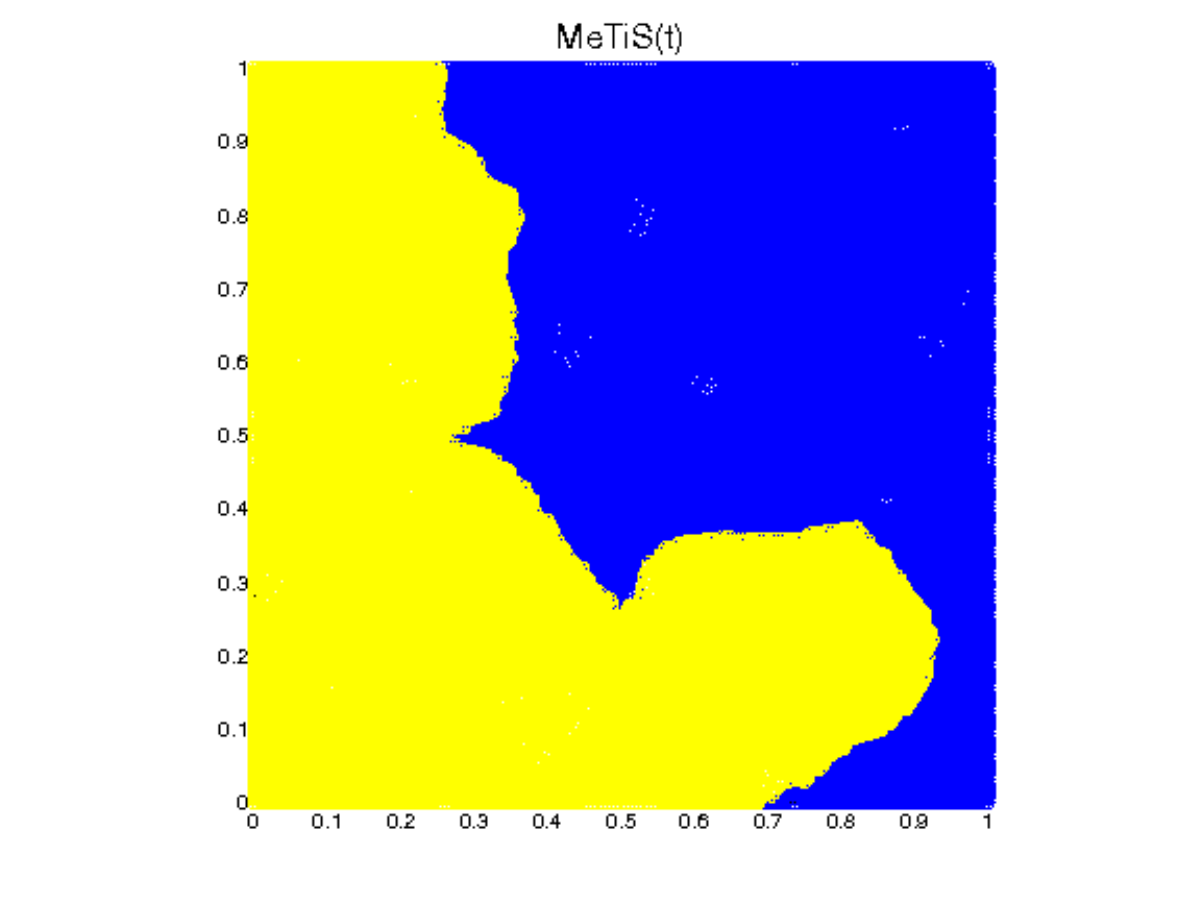}
 \includegraphics[width = 6.4cm]{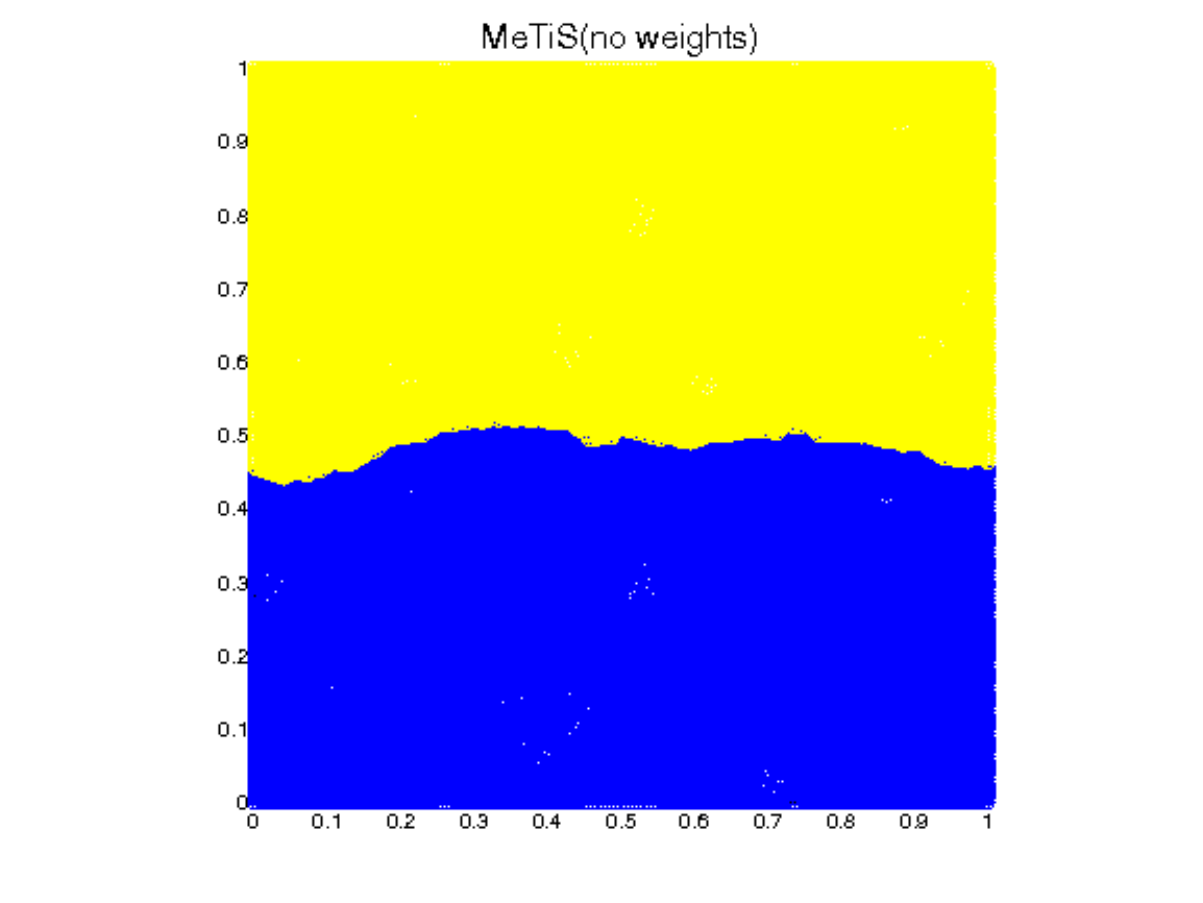}
 \end{center}
 \caption{Bipartitions of unstructured mesh (Figure~\ref{fig:irreg_geom}) for problem~(\ref{eqn:pde}) 
with coefficients in~(\ref{eqn:j_xy_irreg}).}
 \label{fig:irreg_partn}
\end{figure}

\begin{figure}[h]
 \begin{center}
 \includegraphics[width = 6.4cm]{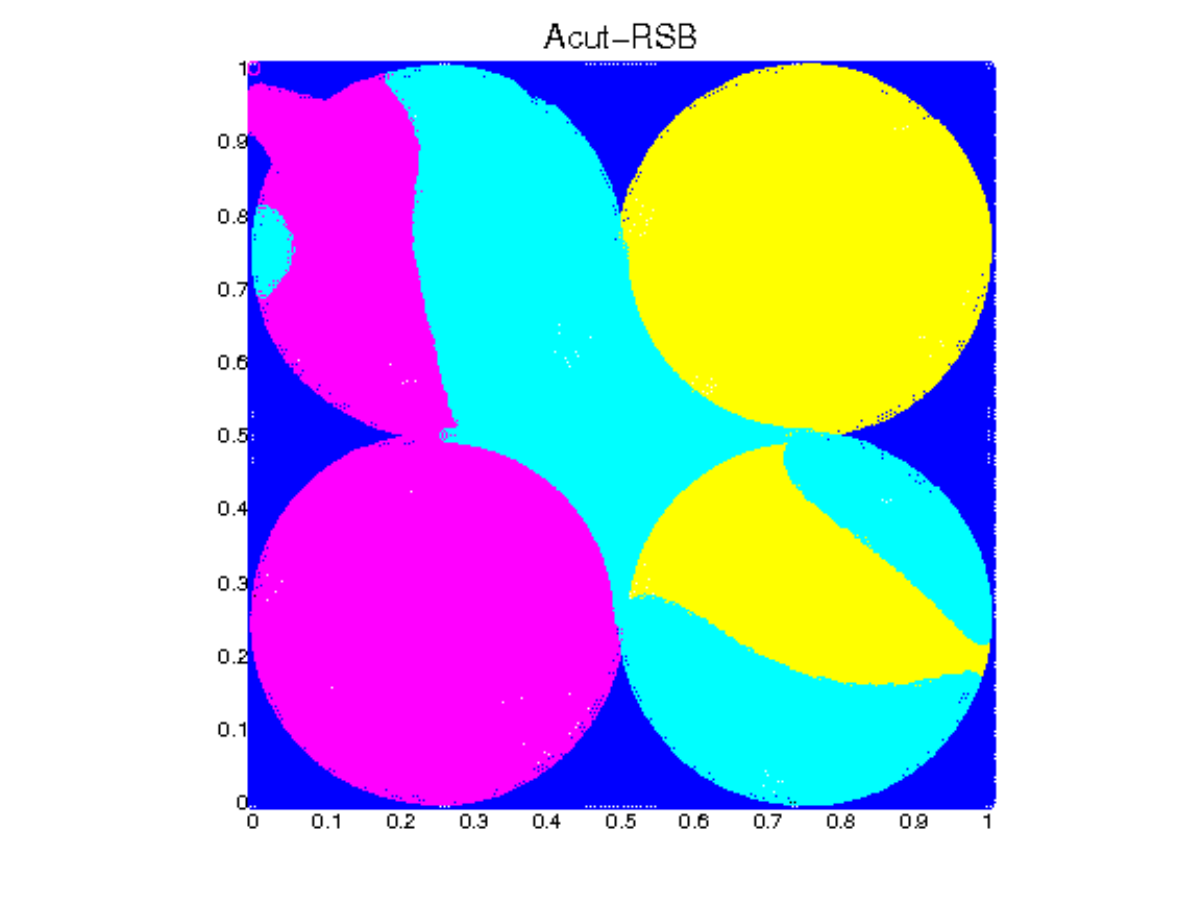}
 \includegraphics[width = 6.4cm]{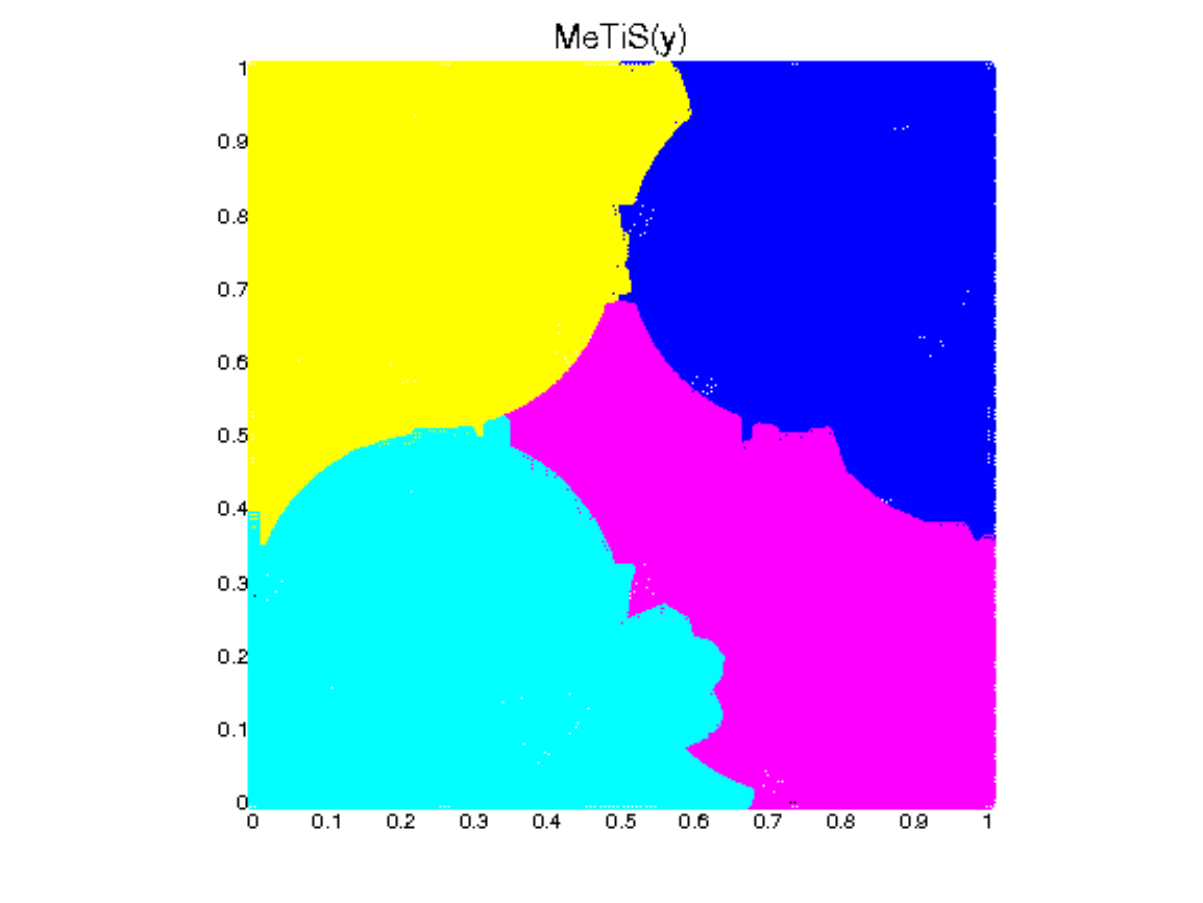} \\
 \includegraphics[width = 6.4cm]{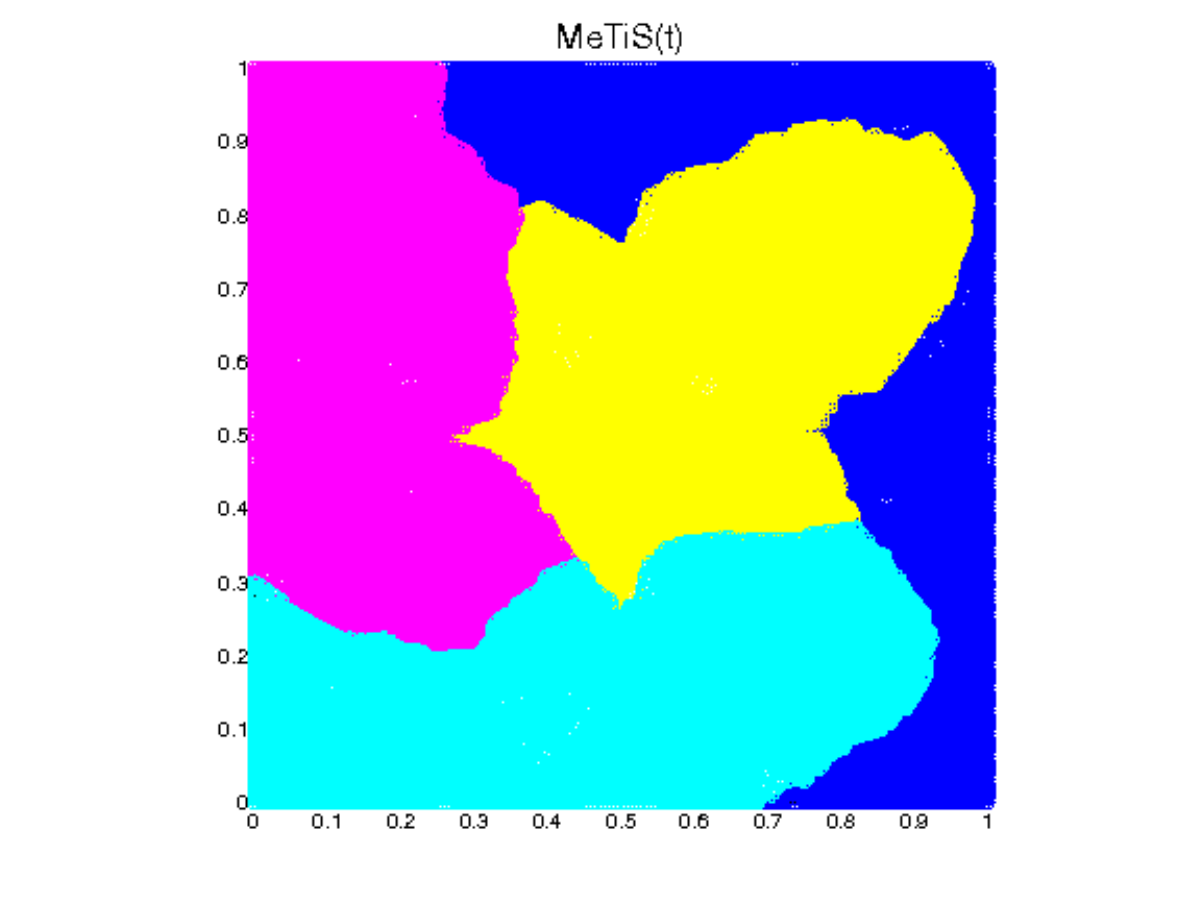}
 \includegraphics[width = 6.4cm]{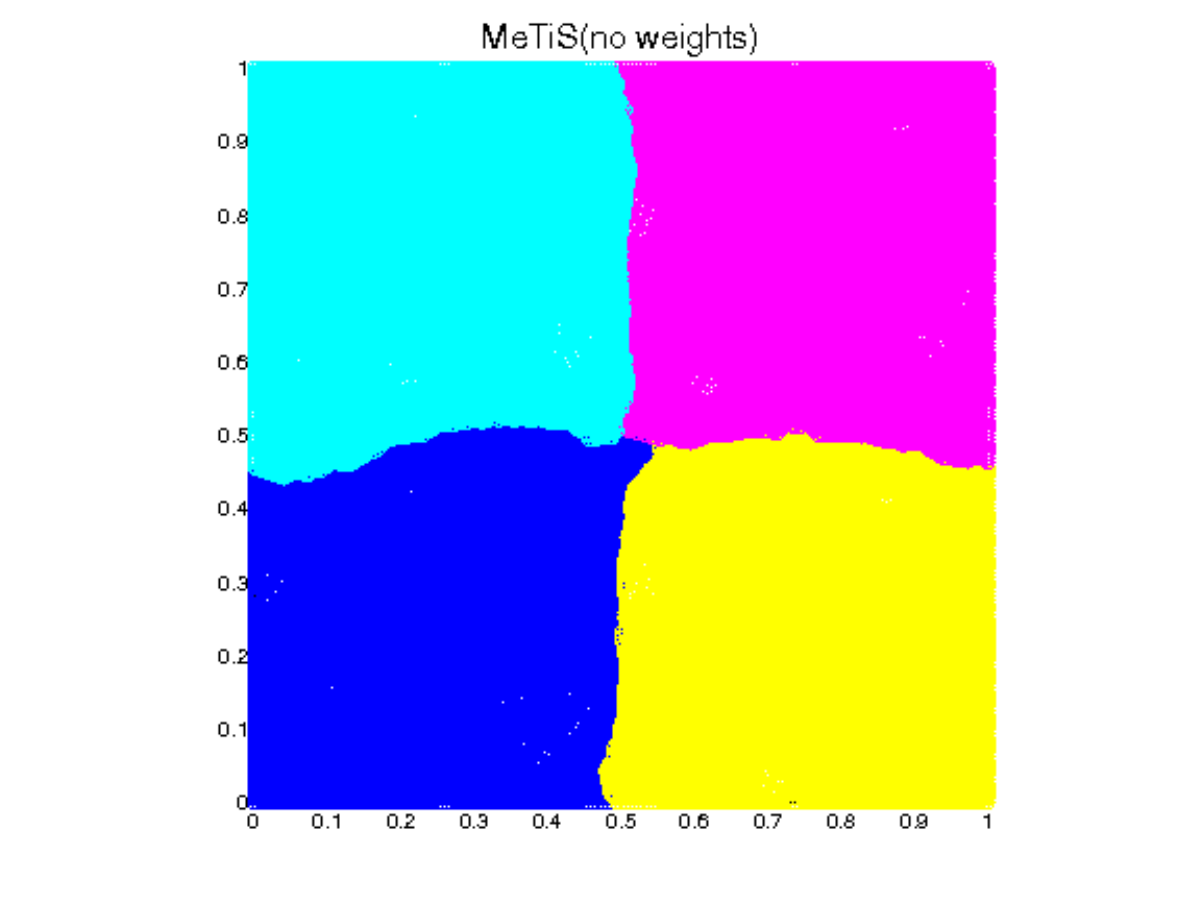}
 \end{center}
 \caption{Partitioning of unstructured mesh (Figure~\ref{fig:irreg_geom}) into 4 subdomains 
for problem~(\ref{eqn:pde}) with coefficients in~(\ref{eqn:j_xy_irreg}).}
 \label{fig:irreg_partn4D}
\end{figure}

% Test: irregular grid n = 47713
% input file: irregular_grid_4circ.mat
% gamma 1e+4 (for metis weights)
% delta = 1e-2
% jump 10^7
% 877 lobpcg steps (for bipartitioning)
% tol lobpcg 1e-7 (for 2 subdomains) and 1e-8 (for 4 subdomains)
% workspace in workspace_4circ.mat    (2 subdomains)
% workspace in workspace_4circ4D.mat  (4 subdomains)

\begin{figure}[h]
 \begin{center}
 \includegraphics[width = 6cm]{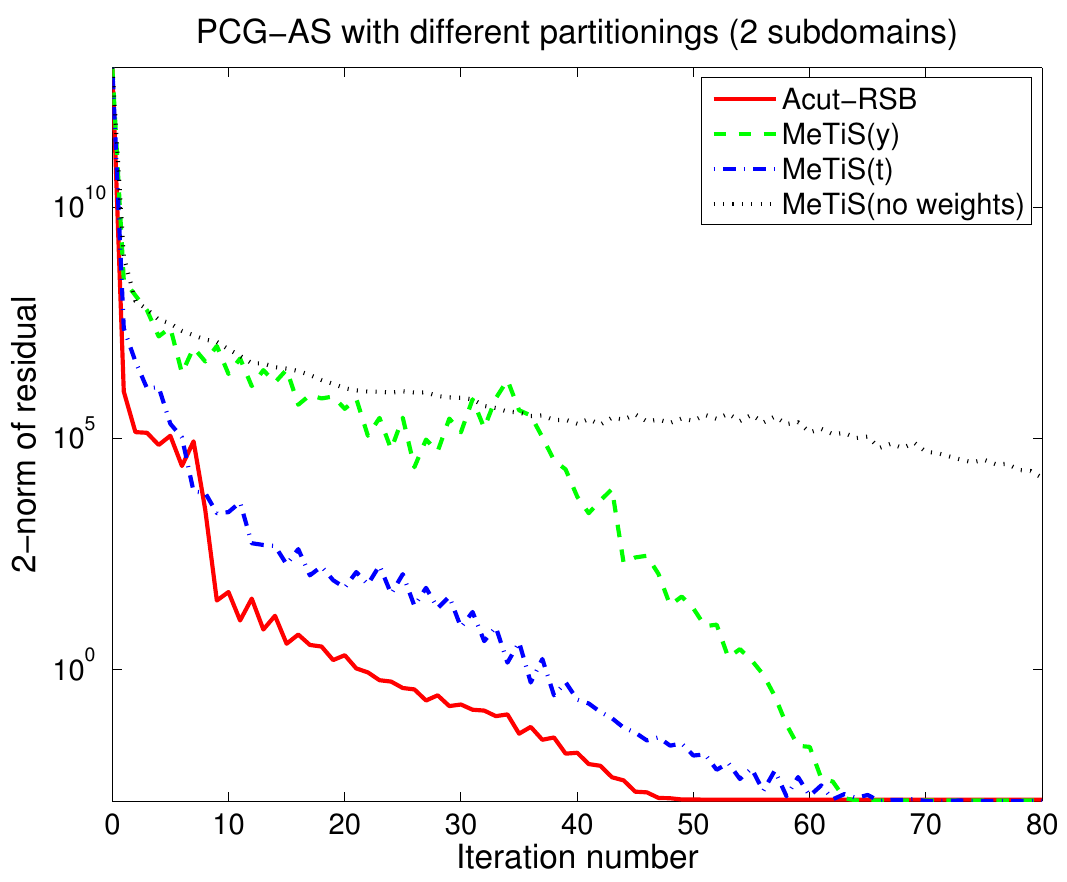}
 \includegraphics[width = 6cm]{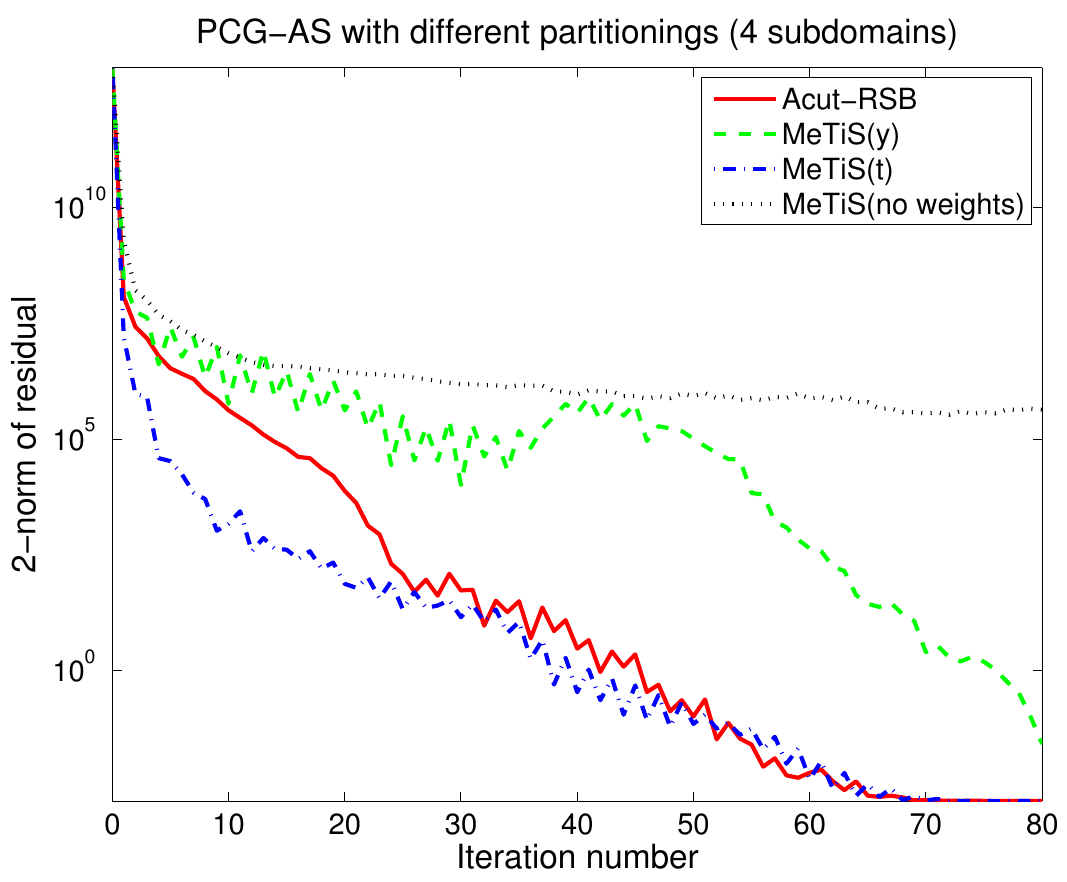}
 \end{center}
 \caption{Convergence of PCG--AS with different partitions for problem~(\ref{eqn:pde})
  with coefficients in~(\ref{eqn:j_xy_irreg}). The domain and the corresponding FE mesh are given
  in Figure~\ref{fig:irreg_geom}. The linear system of size $n = 47,713$ is partitioned into
  2 (left) and 4 (right) subdomains.}
 \label{fig:irreg_cv}
\end{figure}

Figure~\ref{fig:irreg_cv} shows that the use of Acut-RSB reduces the iteration count by
approximately $15\%$ compared to the weighted MeTiS runs for the case of two subdomains,
and gives a result comparable to MeTiS($t$) for four subdomains. Note that in the latter case
the comparable convergence results are produced even though the ``relcoef'' for MeTiS($t$)
is 6 times larger than that of Acut-RSB, as can be seen from Table~\ref{tbl:irreg}.

\begin{table}
\begin{center}
\caption{Relative cut sizes and amounts of the coefficient information discarded to construct preconditioners for
problem~(\ref{eqn:pde}) with coefficients in~(\ref{eqn:j_xy_irreg}); see Figure~\ref{fig:irreg_geom}.}
\label{tbl:irreg}
\begin{tabular}{|c||c|c||c|c|}
\cline{2-5}
\multicolumn{1}{c|}{} & \multicolumn{2}{c||}{2 subdomains} & \multicolumn{2}{c|}{4 subdomains}\tabularnewline
\hline
Partitioning & relcut & relcoef & \multicolumn{1}{c|}{relcut} & relcoef\tabularnewline
\hline
\hline
Acut-RSB & 1.03 & $3 \times 10^{-4}$ & 1.96 & 0.06\tabularnewline
\hline
MeTiS($y$) & 0.46 & 0.10 & 0.85 & 0.16\tabularnewline
\hline
MeTiS($t$) & 0.53 & $5 \times 10^{-4}$ & 1.04 & 0.01\tabularnewline
\hline
MeTiS(no w.) & 0.28 & 0.47 & 0.56 & \multicolumn{1}{c|}{1.24}\tabularnewline
\hline
\end{tabular}
\end{center}
\end{table}

Now, let us ``invert'' the definition of the coefficients, so that
\begin{equation}\label{eqn:j_xy_irreg_bad}
%\[
a(x,y) = b(x,y) = \left\{\begin{array}{cl}
	10^7, & \mbox{``inside circles and on circles' boundaries''} \\
	1,   & \mbox{``outside circles''}\; ;
	   \end{array}\right.
%\]
\end{equation}
i.e., $a(x,y)$ and $b(x,y)$ are the piecewise constants taking a small value outside of
the four circles and a large value inside. The problem geometry and the mesh are the same as
in Figure~\ref{fig:irreg_geom}. The parameters $\gamma$, $\delta$, and LOBPCG convergence tolerance
remain unchanged. As above, the mesh is partitioned into 2 and 4 subdomains. The results
of bipartitioning are shown in Figure~\ref{fig:irreg_partn_bad}.

\begin{figure}[h]
 \begin{center}
 \includegraphics[width = 6.4cm]{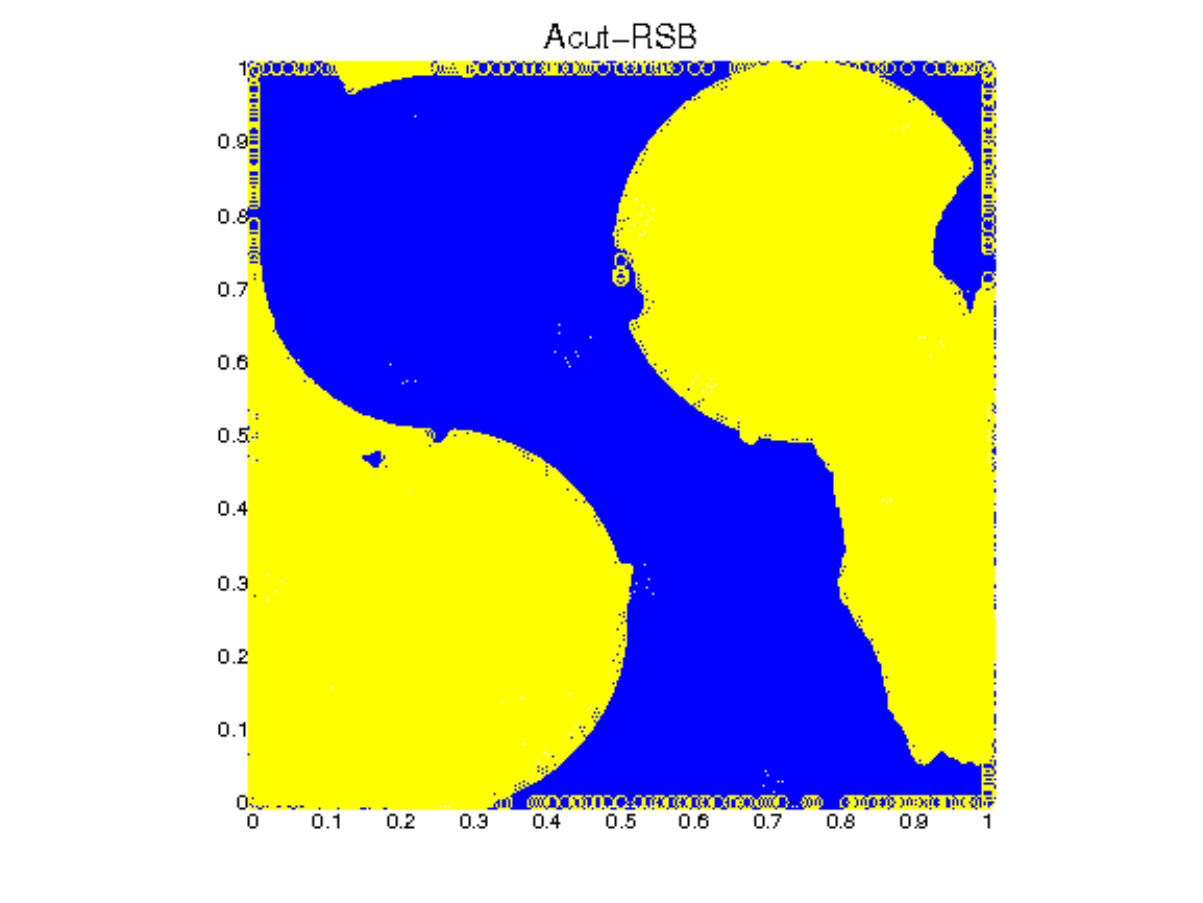}
 \includegraphics[width = 6.4cm]{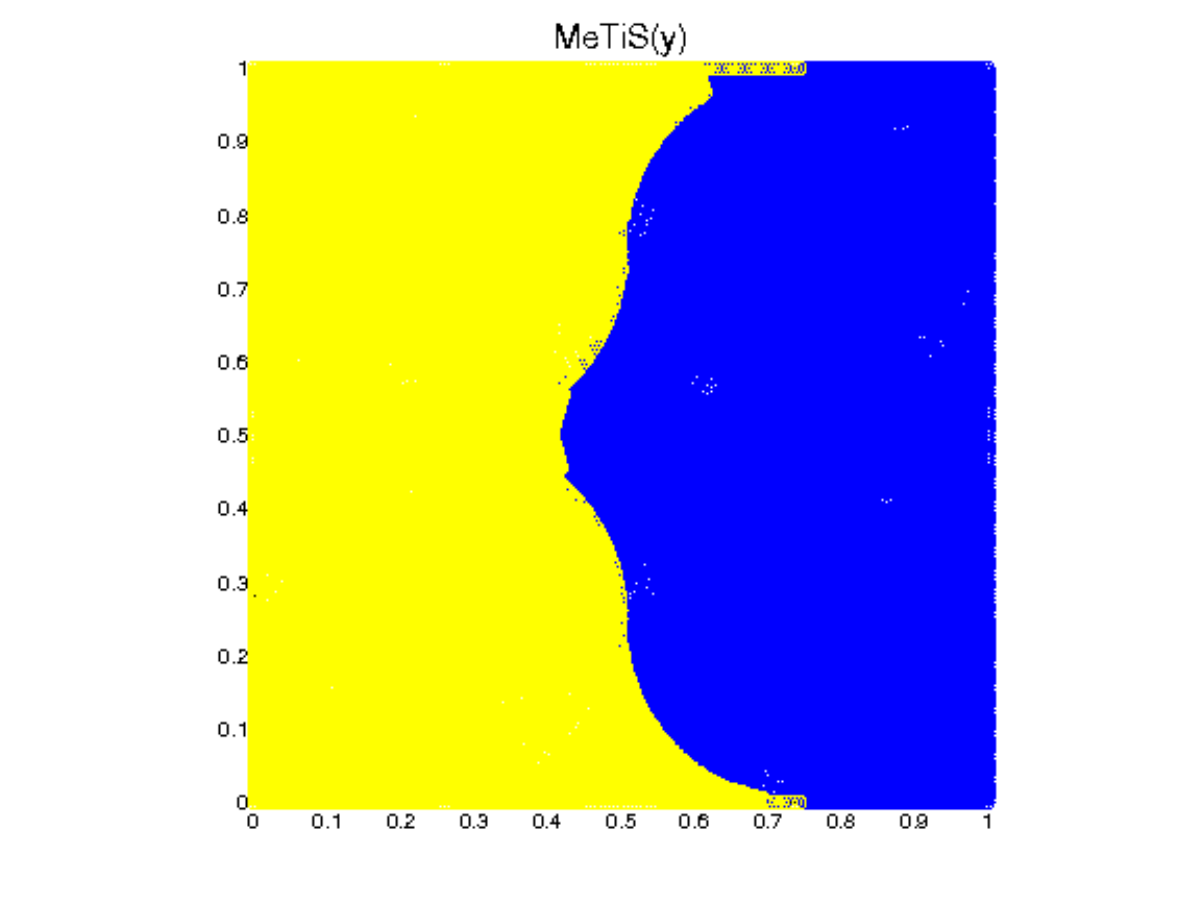} \\
 \includegraphics[width = 6.4cm]{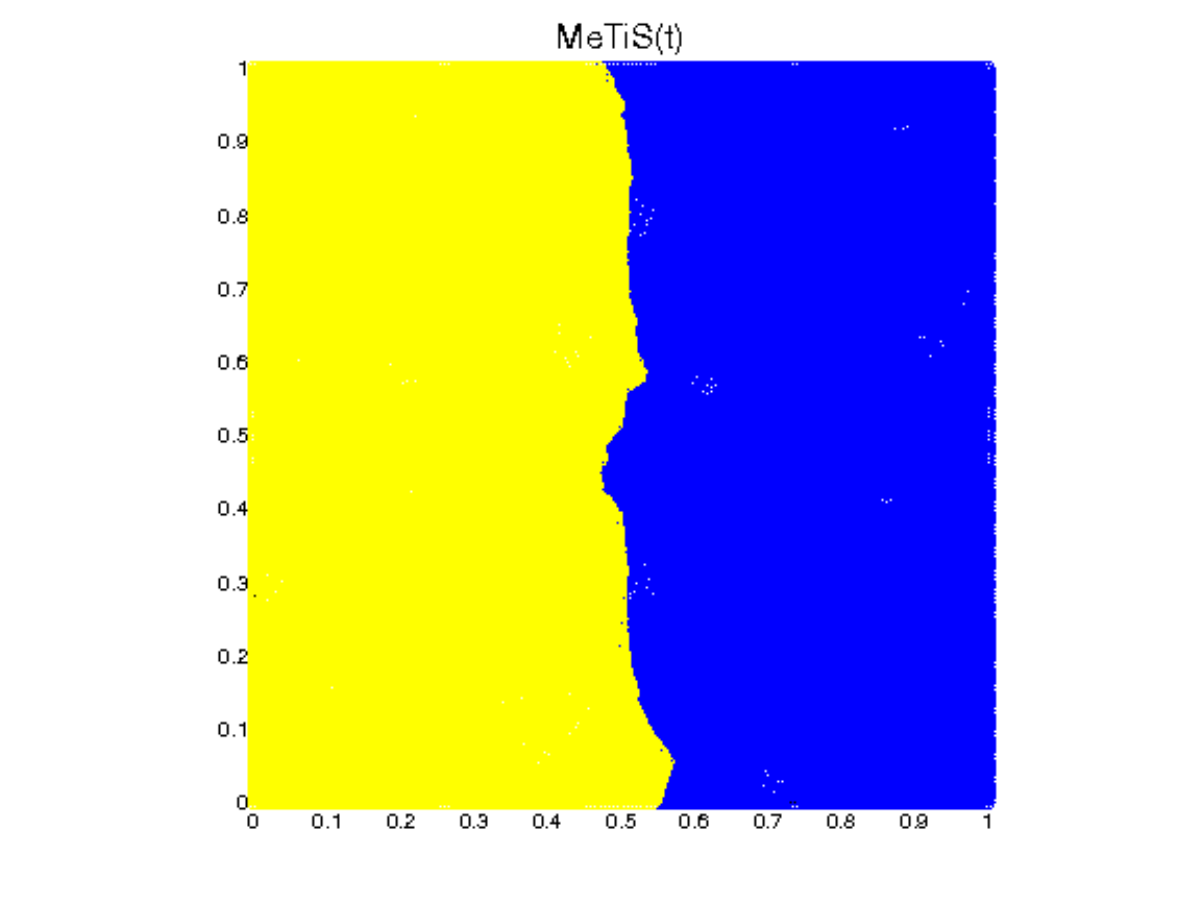}
 \includegraphics[width = 6.4cm]{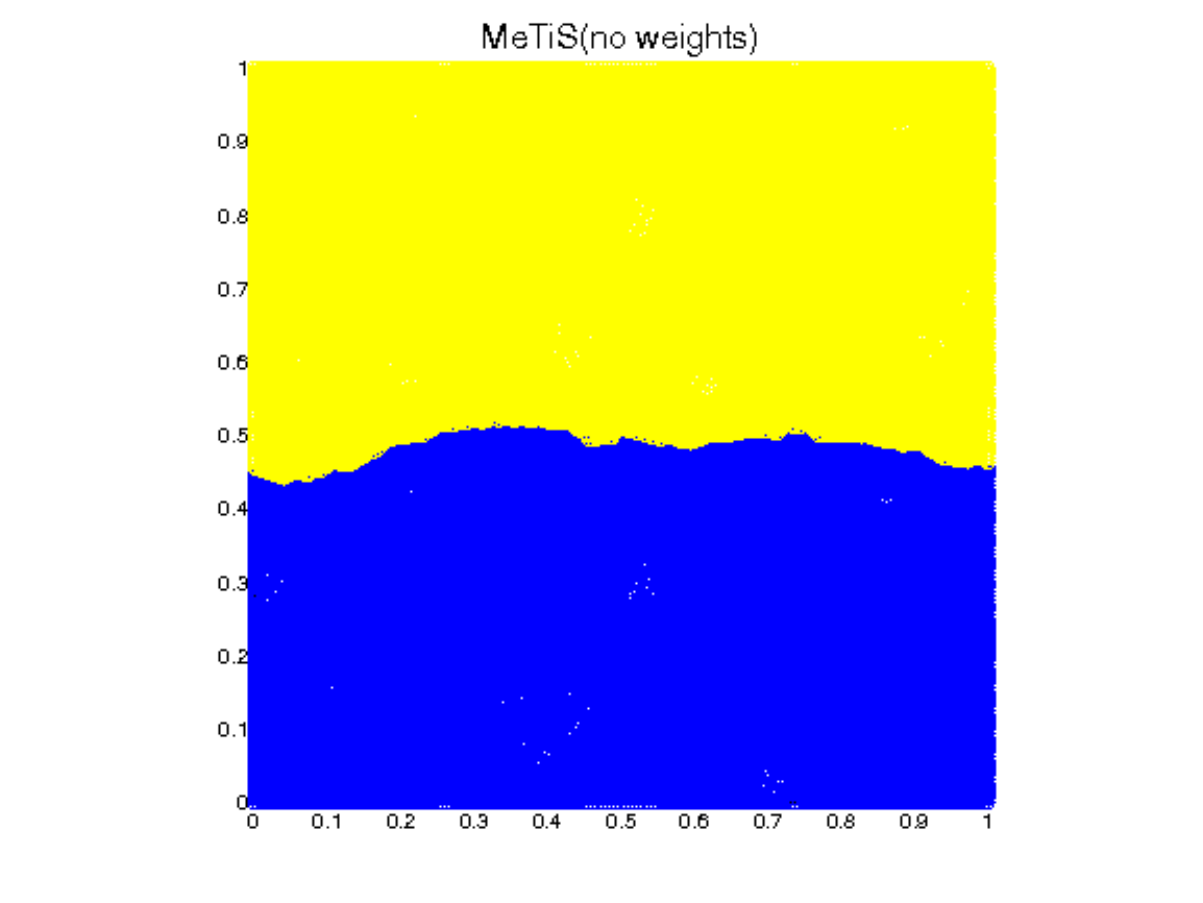}
 \end{center}
 \caption{Bipartitions of unstructured mesh (Figure~\ref{fig:irreg_geom}) for problem~(\ref{eqn:pde}) 
with coefficients in~(\ref{eqn:j_xy_irreg_bad}).}
 \label{fig:irreg_partn_bad}
\end{figure}

% Test: irregular grid n = 47713 (bad example)
% gamma 1e+4 (for metis weights)
% delta = 1e-2
% jump 10^7 inside circles
% 1640 lobpcg steps (for bipartitioning)
% tol lobpcg 1e-6
% workspace in workspace_4circ_bad.mat    (2 subdomains)
% workspace in workspace_4circ4D_bad.mat  (4 subdomains)
% Metis(y) relcoef is 0.00607505% for 2 subdomains and 0.01200761% for 4 subdomains
% Metis(t) relcoef is 0.00607496% for 2 subdomains and 0.01200752% for 4 subdomains

\begin{figure}[h]
 \begin{center}
 \includegraphics[width = 6cm]{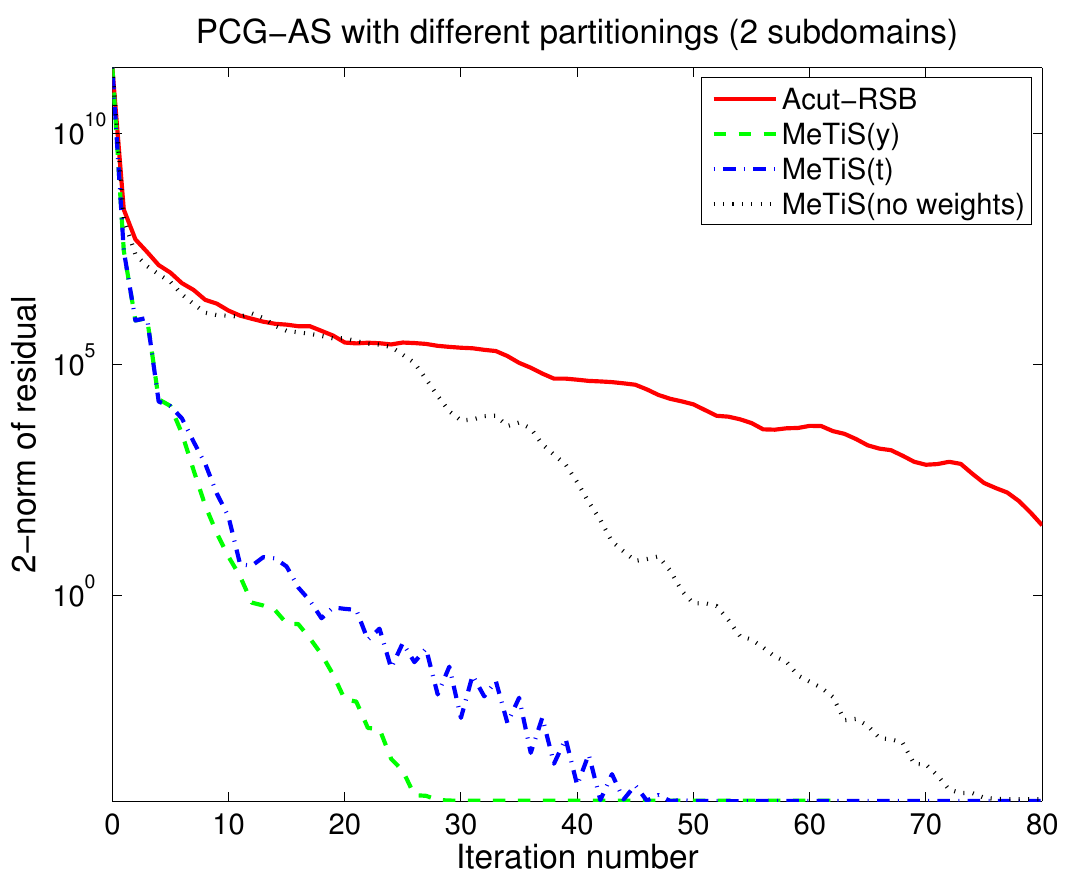}
 \includegraphics[width = 6cm]{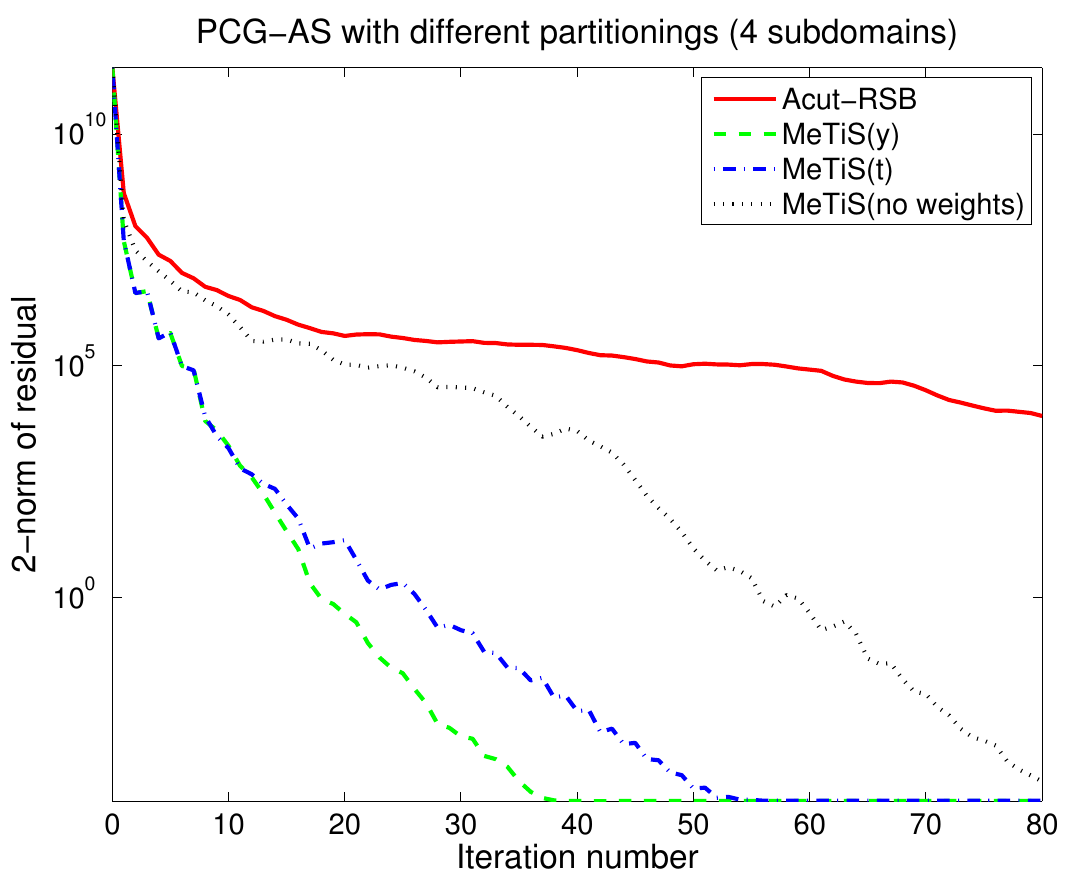}
 \end{center}
 \caption{Convergence of PCG--AS with different partitions for problem~(\ref{eqn:pde})
  with coefficients in~(\ref{eqn:j_xy_irreg_bad}). The domain and the corresponding FE mesh are given
  in Figure~\ref{fig:irreg_geom}. The linear system of size $n = 47,713$ is partitioned into
  2 (left) and 4 (right) subdomains.}
 \label{fig:irreg_cv_bad}
\end{figure}

Figure~\ref{fig:irreg_cv_bad} demonstrates that for the above choice of problem coefficients
Acut-RSB delivers partitions that fail to produce a satisfactory preconditioning quality.
The convergence of the corresponding PCG--AS with Acut-RSB runs is noticeably inferior to that of PCG--AS
with MeTiS partitions.

%%%%%
We explain this poor behavior by presence of a large number
of edges that connect vertices on the circles' boundaries and those outside
the circles; see finer mesh regions along the exterior of the circles' boundaries in
Figure~\ref{fig:irreg_geom} (right). According to~(\ref{eqn:j_xy_irreg_bad}), such vertices belong to subdomains with 
different coefficient magnitudes. Therefore, by definition of weights $w_{ij}$, 
the corresponding edges are the ones targeted by Acut-RSB; see Figure~\ref{fig:irreg_partn_bad} (top left).
%%%%%
%We attribute this poor behavior to the observation that the mesh
%in Figure~\ref{fig:irreg_geom} (right) has a very large number of edges 
%along the \textit{exterior} of 
%the boundary of the circle jump regions, 
%which connect the subdomains with different coefficient magnitudes.
%According to the definition of weights $w_{ij}$, such edges are the ones targeted by Acut-RSB; see
%Figure~\ref{fig:irreg_partn_bad} (top left).
As a result, the cuts produced by Acut-RSB turn out to be  significantly
larger (about 5--6 times) than those made by MeTiS; see ``relcut'' in
Table~\ref{tbl:irreg_bad}.
Even though the entries $|a_{ij}|$ corresponding to the edges in Acut are
relatively small, they accumulate into
an excessively large amount of information discarded from $A$ (see ``relcoef''
in Table~\ref{tbl:irreg_bad}),
thereby leading
to the inferior convergence of PCG--AS. 
In other words, in this example, 
minimization of~(\ref{eqn:subopt}) is mainly contributed by the increase of the number of cut 
edges rather than the decrease of their weight,
%Hence, the shape-capturing facility of the Acut-RSB 
which appears to hinder the convergence.
%performance in this case.

\begin{table}
\begin{center}
\caption{Relative cut sizes and amounts of the coefficient information discarded to construct preconditioners for
problem~(\ref{eqn:pde}) with coefficients in~(\ref{eqn:j_xy_irreg_bad}); see Figure~\ref{fig:irreg_geom}.}
\label{tbl:irreg_bad}
\begin{tabular}{|c||c|c||c|c|}
\cline{2-5}
\multicolumn{1}{c|}{} & \multicolumn{2}{c||}{2 subdomains} & \multicolumn{2}{c|}{4 subdomains}\tabularnewline
\hline
Partitioning & relcut & relcoef & \multicolumn{1}{c|}{relcut} & relcoef\tabularnewline
\hline
\hline
Acut-RSB & 1.80 & 0.44 & 3.67 & 1.51 \tabularnewline
\hline
MeTiS($y$) & 0.38 & 0.01 & 0.76 & 0.01\tabularnewline
\hline
MeTiS($t$) & 0.27 & 0.01 & 0.54 & 0.01\tabularnewline
\hline
MeTiS(no w.) & 0.28 & 0.10 & 0.56 & \multicolumn{1}{c|}{0.11}\tabularnewline
\hline
\end{tabular}
\end{center}
\end{table}

%%%%%%%%%%%%%
In contrast, the number of edges that connect vertices on the circles' boundaries and those inside
the circles is relatively small; see coarser mesh regions along the interior of the 
circles' boundaries in Figure~\ref{fig:irreg_geom} (right). 
If problem coefficients are defined by~(\ref{eqn:j_xy_irreg}), then these are the edges targeted by
Acut-RSB; see Figure~\ref{fig:irreg_partn} (top left). 
However, their number is smaller than that of the edges discraded by the 
(unsuccessful) run of Acut-RSB with coefficients~(\ref{eqn:j_xy_irreg_bad}).
%%%%%%%%%%%%%
%In contrast, if problem coefficients are defined by~(\ref{eqn:j_xy_irreg}),
%then the edges targeted by Acut-RSB are those 
%along the \textit{interior} of the boundary of the circle subregions;  
%see Figure~\ref{fig:irreg_partn} (top left). 
%%
%These edges connect the subdomains corresponding to different coefficient magnitudes; 
%and,
%as can be seen from Figure~\ref{fig:irreg_geom} (right),
%their number is evidently smaller than that of the edges
%targeted by the 
%(unsuccessful) run of Acut-RSB 
%with coefficients~(\ref{eqn:j_xy_irreg_bad}).
%
Therefore, as reported in Table~\ref{tbl:irreg}, the values of ``relcut'' and ``relcoef''
corresponding to Acut-RSB are not large, 
i.e., minimization of~(\ref{eqn:subopt}) is given by a suitable balance between the 
cut weight and the number of cut edges. As a result,
%providing for 
a 
better
%satisfactory 
convergence
behavior is observed in Figure~\ref{fig:irreg_cv}.

Finally, let us remark that, unlike in all the previous examples, the best
convergence in Figure~\ref{fig:irreg_cv_bad}  is given by
PCG--AS with MeTiS($y$). In particular, this shows that it is not clear how to \textit{optimally} define MeTiS
weights if the preconditioning quality becomes an objective.
We also note that the values of ``relcoef'' in Table~\ref{tbl:irreg_bad} for MeTiS($y$) and MeTiS($t$)
are essentially the same (up to the sixth decimal digit), while the convergence of the former is
noticeably superior.

\section{Conclusion}
This paper introduces a new approach for partitioning SPD linear systems.
The suggested technique is based on approximating the so-called Acut of the
matrix adjacency graph. The information about matrix coefficients is utilized
through the graph's edge weights.

The resulting partitioning procedure represents a form of RSB,
where each step of the recursion requires solving
a generalized eigenvalue problem that simultaneously
involves weighted and standard graph Laplacians.
It is shown that, for a number of test problems, 
the new partitioning significantly improves the quality of the associated nonoverlapping AS preconditioners, 
compared to MeTiS with several different
weighting schemes. In the context of parallel solution, the observed increase in the robustness
of the iterative method occurs at a price of extra communication overhead.
% For a few of our examples, however, 
%this overhead may be mitigated by  the substantial
%decrease in the iteration count.

The new partitioning strategy is shown to be effective for test linear systems with 
large variations in matrix coefficients. We have observed that the quality of 
the result strongly depends on the difference in the magnitude of coefficients, as well as 
(in terms of PDE's) on the problem geometry and mesh structure. Future research should address
the development of practical recommendations on when the proposed partitioning method
is preferable to the existing techniques. Since at the current exploratory stage
the test problems have been relatively small and simple, it is of interest to benchmark 
the partitoning quality and performance of the new approach at a larger scale and for more
complex geometries.

%Thus, while getting the fastest partitioner is not a priority 
%at this exploratory stage,
%we are aware that  this is an important issue which should be addressed in the future.

While in this work we have disregarded the requirement on minimizing communication volume,
future research should also address the trade-off  between
preconditioning quality and parallel efficiency.
In particular, we hope that this will lead to the containment of the large cut sizes
produced by Acut. %One example among others has been demonstrated in the
%numerical experiments of Section~\ref{sec:numer}.

This paper considers only the case of nonoverlapping AS (block Jacobi) preconditioners.
However, it is of interest to apply the same partitioning approach for other preconditioning
strategies. For example, our experiments (not reported here) suggest that results
for the \textit{overlapping} AS are similar to the above reported \textit{nonoverlapping} AS.
The overlaps have been introduced by growing a few layers of nodes for each nonoverlapping subdomain.

The new partitioning procedure is heavily rooted in the SPD properties of
 the coefficient matrix. Therefore, it is not clear if successful results
can be obtained for symmetric indefinite or only structurally symmetric matrices.

Finally, we hope that applications of the Acut based partitioning can be found in other
areas of science and engineering, %unrelated to scientific computations,
 such as data clustering and network
analysis and logistics.

\paragraph{\emph{\textbf{Acknowledgments}}} The authors thank Dr.~Bed\v{r}ich Soused\'{\i}k for sharing
the FE assembling codes to reproduce the 3D linear elasticity example from~\cite{Mandel.Sousedik.Sistek:12},
and Ruipeng Li for assisting with SPARSKIT matrix generation.

% \newpage

%\bibliographystyle{plainnat}
\def\refname{

\end{document}